\documentclass[aop]{imsart}

\RequirePackage{amsthm,amsmath,amsfonts,amssymb}
\RequirePackage[numbers]{natbib}

\usepackage[hidelinks]{hyperref}
\usepackage{pgfplots}
\usepackage{amsthm}
\usepackage{latexsym}
\usepackage{dsfont}
\usepackage{bbm}
\usepackage{amssymb}
\usepackage{amsmath}
\usepackage{mathtools}
\usepackage{subcaption}
\usepackage{float}
\usepackage{tikz}
\usepackage{esvect}

\usepackage{enumerate}

\startlocaldefs

\newcommand{\defeq}{\vcentcolon=}
\newcommand{\eqdef}{=\vcentcolon}

\newcommand{\distto}{%
  \mathrel{\vbox{\offinterlineskip\ialign{%
    \hfil##\hfil\cr
    $\scriptscriptstyle\mathrm{d}$\cr
    \noalign{\kern-.05ex}
    $\to$\cr
}}}}
\newcommand{\findimto}{%
  \mathrel{\vbox{\offinterlineskip\ialign{%
    \hfil##\hfil\cr
    $\scriptscriptstyle\mathrm{f.d.}$\cr
    \noalign{\kern-.05ex}
    $\to$\cr
}}}}
\newcommand{\stablyto}{%
  \mathrel{\vbox{\offinterlineskip\ialign{%
    \hfil##\hfil\cr
    $\scriptscriptstyle\mathrm{st\ }$\cr
    \noalign{\kern-.05ex}
    $\to$\cr
}}}}
\newcommand{\Probto}{%
  \mathrel{\vbox{\offinterlineskip\ialign{%
    \hfil##\hfil\cr
    $\scriptscriptstyle\Prob$\cr
    \noalign{\kern-.05ex}
    $\to$\cr
}}}}
\newcommand{\TVto}{%
  \mathrel{\vbox{\offinterlineskip\ialign{%
    \hfil##\hfil\cr
    $\scriptscriptstyle\mathrm{TV}$\cr
    \noalign{\kern-.05ex}
    $\to$\cr
}}}}

\newcommand{\ind}[1]{\mathds{1}_{\{#1\}}}

\newcommand{\N}{\mathds{N}}
\newcommand{\Z}{\mathds{Z}}
\newcommand{\Q}{\mathds{Q}}
\newcommand{\R}{\mathds{R}}
\newcommand{\B}{\mathcal{B}}
\newcommand{\G}{\mathds{G}}
\newcommand{\C}{\mathds{C}}

\newcommand{\Gfct}{\mathcal{G}}
\newcommand{\familytree}{\mathcal{T}}

\newcommand{\I}{\mathcal{I}}
\newcommand{\cJ}{\mathcal{J}}

\newcommand{\imag}{\mathrm{i}}
\newcommand{\Real}{\mathrm{Re}}
\newcommand{\Imag}{\mathrm{Im}}
\newcommand{\Surv}{\mathcal{S}}
\newcommand{\comp}{\mathsf{c}}
\newcommand{\transp}{\mathsf{T}}
\newcommand{\e}{\mathsf{e}}

\newcommand{\Prob}{\mathds{P}}
\newcommand{\Probsf}{\mathsf{P}}
\DeclareMathOperator{\Var}{\mathrm{Var}}
\DeclareMathOperator{\Cov}{\mathrm{Cov}}
\DeclareMathOperator{\Res}{\mathrm{Res}}

\newcommand{\E}{\mathds{E}}
\newcommand{\Esf}{\mathsf{E}}
\newcommand{\U}{\mathsf{U}}		
\newcommand{\A}{\mathcal{A}}
\renewcommand{\L}{\mathcal{L}}
\newcommand{\F}{\mathcal{F}}

\newcommand{\cZ}{\mathcal{Z}}
\newcommand{\cC}{\mathcal{C}}
\newcommand{\1}{\mathds{1}}
\newcommand{\cN}{\mathcal N}

\renewcommand{\S}{\mathcal{S}}

\newcommand{\eqdist}{%
  \mathrel{\vbox{\offinterlineskip\ialign{%
    \hfil##\hfil\cr
    $\scriptscriptstyle\mathrm{law}$\cr
    \noalign{\kern.2ex}
    $=$\cr
}}}}

\newcommand{\dd}{\mathrm{d}}
\newcommand{\ds}{\mathrm{d} \mathit{s}}
\newcommand{\dt}{\mathrm{d} \mathit{t}}

\newcommand{\dx}{\mathrm{d} \mathit{x}}
\newcommand{\dy}{\mathrm{d} \mathit{y}}
\newcommand{\dz}{\mathrm{d} \mathit{z}}

\newcommand{\dProb}{\mathrm{d} \Prob}

\theoremstyle{plain}
\newtheorem{theorem}{Theorem}[section]
\newtheorem{corollary}[theorem]{Corollary}
\newtheorem{lemma}[theorem]{Lemma}
\newtheorem{proposition}[theorem]{Proposition}
\theoremstyle{remark}
\newtheorem{remark}[theorem]{Remark}
\newtheorem{oproblem}{Open problem}
\theoremstyle{definition}

\numberwithin{equation}{section}

\newcounter{assumptions}

\endlocaldefs

\begin{document}
	
\begin{frontmatter}
\title{Asymptotic fluctuations in supercritical Crump-Mode-Jagers processes}
\runtitle{Asymptotic fluctuations of supercritical CMJ processes}
		
\begin{aug}
\author[A]{\fnms{Alexander} \snm{Iksanov}\ead[label=e1]{iksan@univ.kiev.ua}},
\author[B]{\fnms{Konrad} \snm{Kolesko}\ead[label=e2]{konrad.kolesko@math.uni.wroc.pl}}
\and
\author[C]{\fnms{Matthias} \snm{Meiners}\ead[label=e3]{matthias.meiners@math.uni-giessen.de}}
\address[A]{Faculty of Computer Science and Cybernetics,
Taras Shevchenko National University of Kyiv, Ukraine, \printead{e1}}
\address[B]{Mathematical Institute, University of Wroc{\l}aw, Poland, \printead{e2}}
\address[C]{Mathematisches Institut, University of Gie\ss en, Germany, \printead{e3}}
\end{aug}
		
\begin{abstract}
Consider a supercritical Crump--Mode--Jagers process $(\mathcal Z_t^{\varphi})_{t \geq 0}$
counted with a random characteristic $\varphi$.
Nerman's celebrated law of large numbers
[{\it Z.~Wahrsch.~Verw.~Gebiete} 57, 365--395, 1981] states that,
under some mild assumptions, $e^{-\alpha t} \mathcal Z_t^\varphi$
converges almost surely as $t \to \infty$ to $aW$.
Here, $\alpha>0$ is the Malthusian parameter,
$a$ is a constant and $W$ is the limit of Nerman's martingale,
which is positive on the survival event.
In this general situation, under additional (second moment) assumptions,
we prove a central limit theorem for $(\mathcal Z_t^{\varphi})_{t \geq 0}$.
More precisely, we show that there exist a constant $k \in \mathds N_0$
and a function $H(t)$, a finite random linear combination of functions of the form
$t^j e^{\lambda t}$ with $\alpha/2 \leq \Real(\lambda)<\alpha$,
such that $(\mathcal Z_t^\varphi - a e^{\alpha t}W -H(t))/\sqrt{t^k e^{\alpha t}}$
converges in distribution to a normal random variable with random variance.
This result unifies and extends various central limit theorem-type results for specific branching processes.
\end{abstract}
		
		
\end{frontmatter}

\maketitle

\tableofcontents

\section{Introduction}		\label{sec:Introduction}

A general (Crump-Mode-Jagers) branching process starts at time $0$ with a single individual,
the ancestor, who is alive in the random time interval $[0,\zeta)$ for a random variable $\zeta$,
the life span, taking values in $[0,\infty]$.
The ancestor produces offspring born at the points of a reproduction point process $\xi$ on $[0,\infty)$.
No particular assumption about the dependence structure between $\xi$ and $\zeta$
is made.
For each individual $u$ that is ever born there is an independent copy $(\xi_u,\zeta_u)$ of the pair $(\xi,\zeta)$
that determines the birth times of the individual's offspring relative to $u$'s time of birth
and its life span.

The general branching process encompasses e.g.\ the
Bienaym\'e-Galton-Watson process,
the Yule process,
the continuous-time Markov branching process,
the Sevastyanov process,
and the Bellman-Harris process.
We refer to \cite{Jagers:1975} for a more detailed account of the history of the general branching process
and its predecessors.

The general branching process counted with a random characteristic at time $t$
is the sum over all individuals ever born where the contribution
of each individual to the sum is determined by some random characteristic that may take into account
all aspects of the individual's life such as its age at time $t$, its life span, etc.
This formulation makes it possible to treat at one go various quantities of interest derived from the general branching process
such as the number of births up to time $t$, the number of individuals alive at time $t$,
the number of individuals alive at time $t$ which are
younger than a given threshold $a > 0$, etc.
A formal description of the model will be given in Section \ref{sec:Setup}.

General branching processes serve as models of
biological populations such as humans, cells or plants
\cite{Haccou+al:2007,Jagers:1975,Kimmel+Axelrod:2015,Olofsson+Sindi:2014},
as models for tumor growth \cite{Durrett:2015,Kimmel+Axelrod:2015},
but also for neutron chain reactions \cite{Asmussen+Hering:1983}
or fragmentation \cite{Janson+Neininger:2008} (after a change of time) to name but a few.
The general branching process is also an important tool within related fields of applied probability or theoretical computer science.
In fact, its applications in these fields are numerous and any attempt to give a complete survey here is hopeless.
We confine ourselves to mentioning its successful application
in the study of asymptotic properties of random graph growth models driven by preferential attachment dynamics
\cite{Athreya+al:2008,Bhamidi+al:2015,Mori+Rokob:2019,Rudas+Toth:2009}
and particularly random tree growth models \cite{Devroye:1987,Holmgren+Janson:2017,Iksanov+Kabluchko:2018,Jog+Loh:2017,Leckey+Mitsche+Wormald:2020,Pittel:1994}.
It is also used as an approximation for epidemic models
\cite{Bhamidi+al:2014,Trapman+al:2016}
and as a model of the initial phase of epidemics such as SARS, Ebola and SARS-CoV-2
\cite{Ball+al:2014,Britton+al:2019,Britton+Scalia_Tomba:2019},
during which the disease spreads exponentially fast
but the impact of population structure and preventive measures is
still small \cite{Trapman+al:2016}.

The laws of large numbers of the supercritical general branching process
counted with a random characteristic are due to Nerman \cite{Nerman:1979,Nerman:1981} in the single-type, non-lattice case,
that is, when the reproduction point process is not concentrated on any lattice.
There were earlier results for special cases, but here we refrain from sketching the history
and instead refer to the introduction of \cite{Nerman:1981}.
The lattice version of Nerman's law of large numbers was proved
by Gatzouras \cite{Gatzouras:2000}.

In view of the relevance of the general branching process in applications
and the fact that the laws of large numbers date back as far as 1981, it is remarkable, and rather
surprising,
that there is no comprehensive
central limit theorem for the general process counted with a random characteristic in the literature.
However, there are results for related models indicating the intricate nature of the fluctuations that can occur.
For the multi-type continuous-time Markov branching process with finite type space
where individuals give birth only at the time of their death
Athreya \cite{Athreya:1969a,Athreya:1969b} proved a central limit theorem
and Janson \cite{Janson:2004} proved a functional central limit theorem.
Asmussen and Hering \cite[Section VIII.3]{Asmussen+Hering:1983} provide
results for the asymptotic fluctuations of multi-type Markov branching processes
with rather general type space. In principle, these results contain the single-type case
of the general branching process
since such a process can be seen as a Markov process
in which the type of an individual at time $t$ is its entire life history up to time $t$.
However, this type space is large, and the assumptions of \cite{Asmussen+Hering:1983}
are typically not satisfied except in special cases such as the case of the Galton-Watson process.
Recently, Janson studied the asymptotic fluctuations of single-type supercritical general branching processes
in the lattice case \cite{Janson:2018}.
For the non-lattice case, there is a second-order result by Janson and Neininger \cite{Janson+Neininger:2008}
for Kolmogorov's conservative fragmentation model that may be translated into the language
of general branching processes. It gives a central limit theorem
for the number of individuals born up to time $t$,
but it requires that the offspring variable $N \defeq \xi([0,\infty))$ be bounded
and the additional assumption that $\int e^{-x} \, \xi(\dx) = 1$ almost surely,
a rather restrictive assumption in the context of general branching processes.
Another related work is the paper by Charmoy, Croydon, and Hambly \cite{Charmoy+al:2017},
where the authors investigate 
the fluctuations of the eigenvalue counting function related to certain random fractals.
This problem can be addressed using limit theorems for specific Crump-Mode-Jagers processes.
The random characteristics  in this model are no longer assumed to be independent,
which takes it beyond the scope of the present paper.
It is worth noting that limit theorems for general branching processes were previously explored
by Jagers and Nerman \cite{Jagers+Nerman:1984b}.
However, the conditions in this paper can be challenging to verify, even for  relatively simple characteristics.
Another related result is the central limit theorem for Nerman's martingale \cite{Iksanov+Kolesko+Meiners:2021b}.

In the present paper, we close the gap in the literature and present
a central limit theorem for the general branching process counted with a random characteristic.
Our main result, Theorem \ref{Thm:main}, contains and extends all results for single-type processes
summarized above. A non-exhaustive list of applications given in Section \ref{sec:applications}
contains Galton-Watson processes, Nerman's martingale and its complex-valued counterparts,
epidemic models, Crump-Mode-Jagers processes with homogeneous Poisson offspring process and general lifetimes,
and conservative fragmentation models.

\subsection*{Organization of the paper}
The paper is organized as follows.
In Section \ref{sec:Setup} we formally introduce the general branching process
counted with a random characteristic. We further state and discuss the assumptions we are working with.
In Section \ref{subsec:main results}, we state the main result, Theorem~\ref{Thm:main}, and its corollaries.
We then apply our general results to some specific models in Section~\ref{sec:applications}.
Section \ref{sec:preliminaries} contains some preliminaries for the proofs.
Nerman's martingale and further related martingales play a crucial role in our theory.
All these martingales are introduced and discussed in Section~\ref{sec:Nerman's martingales as CMJ processes}.
Section~\ref{sec:proofs} is devoted to proving our main result, Theorem~\ref{Thm:main}.
Our central limit theorem is based on an asymptotic expansion
of the mean of a general branching process counted with a random characteristic.
Such asymptotic expansions are derived in Section \ref{sec:asymptotic expansion of the mean}.
We close the paper with Section~\ref{sec:discussion}, in which possible future research directions are outlined.

\section{Setup, preliminaries and main results}	\label{sec:Setup}

We continue with a formal description of the general branching process.

\subsection{The general branching process counted with a random characteristic}	\label{subsec:General branching process}

We introduce the general (Crump-Mode-Jagers) branching process following Jagers \cite{Jagers:1975,Jagers:1989}.
The process starts with a single individual, the ancestor, born at time $0$.
The ancestor produces offspring born at the points of a reproduction point process
$\xi = \sum_{j=1}^N \delta_{X_j}$ on $[0,\infty)$
where $N = \xi([0,\infty))$ takes values in $\N_0 \cup \{\infty\}$ with $\N_0 \defeq \{0,1,2,\ldots\}$
and $X_j \defeq \inf\{t \geq 0: \xi([0,t]) \geq j\}$.
Here and throughout the paper, the infimum of the empty set is defined to be $\infty$.
The ancestor has a random lifetime $\zeta$, which may be dependent on
$\xi$.
Formally, $\zeta$ is a random variable assuming values in $[0,\infty]$.

Individuals are indexed by $u \in \I = \bigcup_{n \in \N_0} \N^n$ according to their genealogy.
Here, $\N = \{1,2,\ldots\}$ and $\N^0 \defeq \{\varnothing\}$ is the singleton set containing only the empty tuple $\varnothing$.
We use the usual Ulam-Harris notation.
We abbreviate a tuple $u = (u_1, \ldots, u_n) \in \N^n$ by $u_1 \ldots u_n$
and refer to $n$ as the length or generation of $u$; we write $|u| = n$.
In this context, any $u = u_1 \ldots u_n \in \I$ is called (potential) individual.
Its ancestral line is encoded by
\begin{equation*}
\varnothing \to u_1 \to u_1 u_2 \to \ldots \to u_1 \ldots u_n = u
\end{equation*}
where $u_1$ is the $u_1^{\mathrm{th}}$ child of the ancestor, $u_1u_2$ the $u_2^{\mathrm{th}}$ child of $u_1$, etc.
If $v = v_1 \ldots v_m \in \I$, then $uv$ is short for $u_1 \ldots u_n v_1 \ldots v_m$.
For $u \in \I$ and $i \in \N$, the individuals $ui$ will be called children of $u$.
Conversely, $u$ will be called parent of $ui$.
More generally,
$w$ will be called descendant of $u$ (short: $u \preceq w$)
iff $uv = w$ for some $v \in \I$.
Conversely, $u$ will be called an ancestor/progenitor of $w$.
We write $u \prec w$ if $u \preceq w$ and $u \not = w$.
Often, we shall refer to $\N^n$ as the (potential) $n^{\mathrm{th}}$ generation ($n \in \N_0$).
With these notations, we have
\begin{equation*}
|u| = n
\quad	\text{iff}	\quad u \in \N^n
\quad	\text{iff}	\quad u \text{ is an $n^{\mathrm{th}}$ generation (potential) individual}.
\end{equation*}
For $u = u_1 \ldots u_n \in \N^n$ and $k \in \N_0$,
let $u|_k$ denote the ancestor of $u$ in the $k^{\mathrm{th}}$ generation.
Formally, $u|_k$ is the restriction of the vector $u$ to its first $k$ components:
\begin{equation}	\label{eq:v_restr}
u|_k =	\begin{cases}
		\varnothing			& \text{if	$k=0$},					\\
		u_1 \ldots u_k			& \text{if	$1 \leq k \leq |u|$},		\\
		u					& \text{if	$k > |u|$}.
		\end{cases}
\end{equation}
For typographical reasons, we may sometimes write $v|k$ instead of $v|_k$.
For $u \in \I$ let $u\I$ denote the subtree of $\I$ emanating from $u$,
that is,
\begin{equation*}
u\I \defeq \{uv: \, v \in \I\} = \{w \in \I: \, w|_{|u|}=u\}.
\end{equation*}

For each $u \in \I$ there is an independent copy $(\xi_u,\zeta_u)$
of the pair $(\xi,\zeta)$ that determines the birth times of $u$'s offspring relative to its time of birth,
and the duration of its life.
Quantities derived from $(\xi_u,\zeta_u)$ are indexed by $u$.
For instance, $N_u$ is the number of offspring of $u$ and
$X_{u,k}$ is the difference between the birth-time of the $k^{\mathrm{th}}$ child of $u$ and $u$ itself, etc.
The birth-times $S(u)$ for $u \in \I$ are defined recursively.
We set $S(\varnothing) \defeq 0$ and, for $n \in \N_0$,
\begin{equation*}
S(uj) \defeq S(u) + X_{u,j}	\quad	\text{for } u \in \N^n	\text{ and } j \in \N.
\end{equation*}
The family tree of all individuals ever born is denoted by $\familytree \defeq \{u \in \I: S(u) < \infty\}.$
We call
\begin{equation*}
\Surv \defeq \bigcap_{n \in \N} \{\#(\familytree \cap \N^n)\geq 1\}
\end{equation*}
the survival set and its complement $\Surv^\comp=\cup_{n \in \N} \{\#(\familytree \cap \N^n) = 0\}$ the extinction set.
The time of death of individual $u$ is $S(u) + \zeta_u$.
An individual $u$ is alive at time $t \geq 0$ if it is born, but not yet dead at time $t$,
i.e., if
\begin{equation*}
S(u) \leq t < S(u) + \zeta_u.
\end{equation*}

We now construct the canonical space for the general branching process.
For $u \in \I$, let $(\Omega_u,\A_u,P_u)$ be a copy of a given probability space
$(\Omega_\varnothing,\A_\varnothing,P_\varnothing)$,
the \emph{life space}\index{life space} of the ancestor.
An element $\omega \in \Omega_u$ is a possible life career for individual $u$
and any property of interest of $u$ like its mass at some age
or its life span is viewed as a measurable function on the life space.
In particular, $\xi$ and $\zeta$,
the reproduction point process and the life span, are measurable functions defined on $(\Omega_\varnothing,\A_\varnothing)$.

From the life space, we construct the population space:
\begin{equation*}	\textstyle
(\Omega,\F,\Prob)	\defeq \big(\bigtimes_{u \in \I} \Omega_u, \bigotimes_{u \in \I} \A_u, \bigotimes_{u \in \I} P_u \big).
\end{equation*}
For $u \in \I$, we write $\pi_{u}$ for the projection
$\pi_{u}: \bigtimes_{v \in \I} \Omega_v \to \Omega_u$
and $\theta_u$ for the shift $\theta_u((\omega_v)_{v \in \I}) = (\omega_{uv})_{v \in \I}$.
To formally lift an entity $\chi$ defined on the life space, i.e.\ a function $\chi$ on $\Omega_u$,
to the population space,
we define $\chi_u \defeq \chi \circ \pi_{u}$.
In particular, $\xi_u = \xi \circ \pi_{u}$ and $\zeta_u = \zeta \circ \pi_{u}$.
In slight abuse of notation,
if $\chi$ is defined on the life space,
when working on the population space, we write $\chi$ instead of $\chi \circ \pi_{\varnothing}$.
For instance, we sometimes write $\Prob(\zeta \leq t)$
for $P_u(\zeta \leq t) = \Prob(\zeta_\varnothing \leq t)$. A technical remark is in order. Sometimes random variables independent of $\F$ appear.
This means that when required, we work on a suitable extension of the space $(\Omega,\F,\Prob)$.

We are interested in the general branching process counted with a random characteristic.
A \emph{random characteristic} $\varphi$ is a random process on $(\Omega_\varnothing,\A_\varnothing,P_\varnothing)$
taking values in the Skorokhod space of right-continuous functions $f:\R \to \R^d$
with existing left limits at every point in $\R$. Such functions are called c\`adl\`ag for short.
The characteristic $\varphi$ may also be viewed as a stochastic process $\varphi: \Omega_\varnothing \times \R \to \R^d$,
$(\omega,t) \mapsto \varphi(\omega,t)$ with right-continuous paths and existing left limits.
Notice that unlike in some important references \cite{Gatzouras:2000,Nerman:1981},
we allow, 
and actually need at some places, 
that $\varphi(t)\not=0$ for some $t < 0$.
It is known that such a process is product-measurable.
Define $\varphi_u = \varphi \circ \pi_{u}$. By product measurability, $\varphi_u(t-S(u))$ is a random variable.
Note that, for given $u\in \I$, $\varphi_u$ is independent of $S(u)$.
However, $\varphi_u$ and $S(v)$ can be dependent, when $u$ is an ancestor of $v$. The general branching process counted with characteristic $\varphi$
is $\cZ^\varphi = (\cZ_t^\varphi)_{t \in \R}$ where $\cZ_t^\varphi$ is defined by
\begin{equation}	\label{eq:Z_t^phi}
\cZ_t^\varphi \defeq \sum_{u \in \I} \varphi_u(t-S(u)),	\quad	t \in \R.
\end{equation}
Here, we use the convention $\varphi(-\infty) \defeq 0$ and so the above sum involves only terms
associated with individuals that are eventually born.
In the special case $\varphi = \1_{[0,\zeta)}$,
\begin{equation}	\label{eq:Z_t^1_[0,zeta]}
\cZ^{\1_{[0,\zeta)}}_t = \sum_{u \in \I} \1_{[0,\zeta_u)} (t-S(u)) = \sum_{u \in \I} \1_{\{S(u) \leq t < S(u) + \zeta_u\}},
\end{equation}
i.e., $\cZ^{\1_{[0,\zeta)}}_t$ is the number of individuals alive at time $t$.
Similarly,
\begin{equation}	\label{eq:Z_t^1_[0,infty)}
N((t,t+a]) \defeq \cZ^{\1_{[0,a)}}_{t+a} = \sum_{u \in \I} \1_{[0,a)} (t+a-S(u))
= \sum_{u \in \I} \1_{\{t<S(u)\leq t+a\}}
\end{equation}
is the number of individuals born strictly after time $t$ and up to and including time $t+a$, $a > 0$.
The setup covers a wide range of possible applications.
Some special cases and specific examples are covered in Section \ref{sec:applications}.

Notice that $\cZ^\varphi$ is not well-defined {\it a priori}.
Conditions for the finiteness of the general branching process are given in \cite[Section 6.2]{Jagers:1975}.
For instance, the existence of a Malthusian parameter, a condition formally stated as (A\ref{ass:Malthusian parameter}) below
and assumed throughout this paper, implies that the number of individuals born up to and including time $t$ is finite for all $t \geq 0$
almost surely, see \cite[Theorem 6.2.3]{Jagers:1975}.
In particular, if (A\ref{ass:Malthusian parameter}) holds, then $\cZ_t^\varphi$ is well-defined whenever the characteristic $\varphi$
vanishes on the negative half-line since in this case, the sum on the right-hand side of \eqref{eq:Z_t^phi} has only finitely many non-vanishing
summands almost surely.
As, in general, we allow the characteristic $\varphi$ to be real-valued and do not require that it vanishes
on the negative half-line, we need a finiteness result that goes beyond \cite[Theorem 6.2.3]{Jagers:1975}.
Jagers and Nerman \cite{Jagers+Nerman:1984} work under their assumption (6.1),
which corresponds to our condition (A\ref{ass:mean growth}) for $|\varphi|$ below.
However, in our proofs, we shall require the well-definedness of $\cZ_t^{\chi_\lambda}$
for a specific centered characteristic $\chi_\lambda$, defined in Section \ref{sec:Nerman's martingales as CMJ processes},
with $|\chi_\lambda|$ not satisfying (A\ref{ass:mean growth}).
Instead, we work under (A\ref{ass:mean growth}) and (A\ref{ass:variance growth})
to ensure the general branching process counted with a random characteristic to be well-defined.
The corresponding result is Proposition \ref{Prop:L1 existence of Z_t^varphi} below.

\subsection{Assumptions}	\label{subsec:Assumptions}

We write $\mu(\cdot)$ for the intensity measure $\E[\xi(\cdot)]$ of the point process $\xi(\cdot)$,
and $\L \mu$ for its Laplace transform, i.e.,
\begin{equation}	\label{eq:tilde mu}
\L\mu(z) \defeq \int e^{-z x} \, \mu(\dx) = \E \bigg[ \sum_{j=1}^N e^{-z X_j}\bigg]
\end{equation}
for all $z \in \C$ for which the above integral converges absolutely.

Throughout this paper we distinguish between the lattice and the non-lattice case.
Here, we say that $\xi$ is \emph{lattice} if $\mu([0,\infty) \setminus h\N_0) = 0$  for some $h>0$, and
we say that $\xi$ is \emph{non-lattice}, otherwise.
In the lattice case, without loss of generality, we assume that the lattice span is $1$, $\mu([0,\infty) \setminus \N_0) = 0$
and $\mu([0,\infty) \setminus h\N_0) > 0$ for all $h > 1$.
We set $\G \defeq \Z$ in the lattice case and $\G \defeq \R$ in the  non-lattice case.
We use the symbol $\ell$ to denote the counting measure on $\Z$ in the lattice case and the Lebesgue measure in the non-lattice case, respectively.

For a function $f:\G\mapsto \C$   we define the bilateral Laplace transform $\L f$ of $f$ at $z\in\C$ by
\begin{equation*}
\L f(z) \defeq \int_\G e^{-zx} f(x) \, \ell(\dx)
\end{equation*}
whenever the integral converges absolutely.

The following assumption is essential in the law of large numbers \cite{Gatzouras:2000,Nerman:1981}
and, therefore, also for the central limit theorem studied here.

\begin{enumerate}[{\bf{(A}1)}]
	\setcounter{enumi}{\value{assumptions}}
	\item
		There exists a \emph{Malthusian parameter} $\alpha > 0$, i.e., an $\alpha > 0$ satisfying
		\begin{align}
		&\L\mu(\alpha) = \int e^{-\alpha x} \, \mu(\dx) = 1	\quad	\text{and}			\label{eq:Malthusian alpha}	\\
		&\E\bigg[\sum_{j=1}^N X_j e^{-\alpha X_j} \bigg] = -(\L\mu)'(\alpha) \eqdef \beta \in (0,\infty).	\label{eq:beta}
		\end{align}
		\label{ass:Malthusian parameter}
\setcounter{assumptions}{\value{enumi}}
\end{enumerate}
Notice that (A\ref{ass:Malthusian parameter}) implies the supercriticality of the general branching process,
that is, $\E[N] = \mu([0,\infty)) \in (1,\infty]$, which, in turn, ensures that the underlying branching process survives with positive probability
meaning that $\Prob(\Surv)>0$. We stress that the case $\Prob(N=\infty)>0$ is allowed.
For the rest of the paper, we assume that (A\ref{ass:Malthusian parameter}) is satisfied.

In our main results,  we further assume that the Laplace transform $\L\mu$ is finite
on an open half-space $\Real(z) > \vartheta$ for some $\vartheta < \frac\alpha2$:
\begin{enumerate}[{\bf{(A}1)}]
	\setcounter{enumi}{\value{assumptions}}
	\item
	There exists $\vartheta \in (0,\frac\alpha2)$ such that
\begin{equation}	
\L\mu(\vartheta) =\E \bigg[ \sum_{j=1}^N e^{-\vartheta X_j}\bigg] < \infty.
	\end{equation}
	\label{ass:first moment}
	\setcounter{assumptions}{\value{enumi}}		
\end{enumerate}
For the central limit theorem, we need a second moment assumption for the point process $\xi$.
Before we state it, we set $k^*$ to be the maximum of all multiplicities of the roots of $\L\mu(z) = 1$
on the critical line $\Real(z) = \frac\alpha2$ or $k^* \defeq \frac12$ if there is no such root.
\begin{enumerate}[{\bf{(A}1)}]
\setcounter{enumi}{\value{assumptions}}
	\item
		The random variable
		\begin{equation}	\label{eq:Konrad's condition}
		\int(1+x^{k^*-\frac12})e^{-\frac \alpha 2 x} \, \xi(\dx)
		= \sum_{j=1}^N \big(1+X_j^{k^*-\frac12}\big) e^{-\frac \alpha 2 X_j}
		\end{equation}
		has finite second moment.
\label{ass:second moment}
\setcounter{assumptions}{\value{enumi}}
\end{enumerate}

\begin{remark}	\label{Rem:Janson's 2nd moment condition}
Notice that Condition (A6) in \cite{Janson:2018}, namely, the existence of a $\vartheta < \frac\alpha2$ such that
\begin{equation*}	
\E\bigg[\bigg(\sum_{j=1}^N e^{-\vartheta X_j}\bigg)^{\!\!2}\,\bigg] < \infty,
\end{equation*}
implies both our conditions (A\ref{ass:first moment}) and (A\ref{ass:second moment}).
Janson's condition (A6) may be easier to check in cases where it holds.
\end{remark}

The existence of the Malthusian parameter 
allows us to define a nonnegative martingale,
called \emph{Nerman's martingale}, namely,
\begin{align}	\label{eq:Nerman's martingale}
W_t = W_t(\alpha) = \sum_{u \in \cC_t} e^{-\alpha S(u)},	\quad	t \geq 0
\end{align}
where
\begin{align}	\label{eq:coming generation}
\cC_t \defeq \{uj \in \familytree: S(u) \leq t < S(uj)\}
\end{align}
is the coming generation at time $t$. For the proof of the martingale property under (A\ref{ass:Malthusian parameter}) see \cite[Proposition 2.4]{Nerman:1981}.
We denote the almost sure limit of Nerman's martingale by $W$.
Martingale theory implies
that $\E[W]=1$ iff $(W_t)_{t \geq 0}$ is uniformly integrable.
Sufficient conditions for
the latter can be found in \cite[Corollary 3.3]{Nerman:1981}, \cite[Theorem 2.1]{Olofsson:1998} and \cite[Theorems 2.1 and 3.3]{Gatzouras:2000}.
In the given situation,
$(W_t)_{t \geq 0}$ is uniformly integrable iff
\begin{equation*}	\tag{$Z \log Z$}	\label{eq:Z log Z}
\E[Z_1 \log_+ Z_1] < \infty
\end{equation*}
holds where
\begin{align}	\label{eq:Biggins's martingale}
Z_n = \sum_{|u|=n} e^{-\alpha S(u)},	\quad	n \in \N_0.
\end{align}
The process $(Z_n)_{n \in \N_0}$ is also a nonnegative martingale, called Biggins' martingale,
and it has the same almost sure limit $W$ as Nerman's martingale $(W_t)_{t \geq 0}$ \cite[Theorem 3.3]{Gatzouras:2000}.
Since (A\ref{ass:second moment}) immediately implies \eqref{eq:Z log Z},
we infer that validity of (A\ref{ass:Malthusian parameter}) and (A\ref{ass:second moment})
implies that both martingales, $(W_t)_{t \geq 0}$ and $(Z_n)_{n \in \N_0}$,
converge almost surely and in $L^1$ to the same limit $W \geq 0$.
Hence, in our theorems, \eqref{eq:Z log Z} will not be imposed explicitly,
but will hold automatically whenever (A\ref{ass:Malthusian parameter}) and (A\ref{ass:second moment}) are assumed to hold.

We continue with assumptions concerning the random characteristic $\varphi$.
These assumptions are not made throughout the paper, but in certain results only.
It will be explicitly stated, when this is the case.

Throughout the paper, if $\varphi$ is a nonnegative or integrable characteristic
(meaning that $\E[|\varphi(t)|]$ is finite for every $t \in \R$),
then we write $\E[\varphi]$ for the (measurable) function that maps $t \mapsto \E[\varphi](t) \defeq \E[\varphi(t)]$.
This notation has the advantage that if $X$ is a random variable, then we can write $\E[\varphi](X)$,
which is again a random variable.
Similarly, we write $\Var[\varphi]$ for the variance function $\E[(\varphi-\E[\varphi])^2]$,
so $\Var[\varphi](t) = \Var[\varphi(t)]$.
We start with an assumption regarding the mean of the characteristic.
\begin{enumerate}[{\bf{(A}1)}]
\setcounter{enumi}{\value{assumptions}}
	\item
	$\varphi(t) \in L^1$ for every $t \in \R$ and $t \mapsto \E[\varphi](t)e^{-\alpha t}$ is directly Riemann integrable.
	\label{ass:mean growth}
	\setcounter{assumptions}{\value{enumi}}
\end{enumerate}
If (A\ref{ass:Malthusian parameter})
is fulfilled, and if $\varphi$ is a real-valued characteristic such that $|\varphi|$ satisfies
(A\ref{ass:mean growth}), then, in the non-lattice case,
the law of large numbers by Nerman (see \cite[Theorem 6.1]{Jagers+Nerman:1984}) states that
\begin{equation}	\label{eq:LLN CMJ}
e^{-\alpha t} \cZ^{\varphi}_t \to \beta^{-1} \L\E[\varphi](\alpha) W= \beta^{-1} \int e^{-\alpha x} \E[\varphi](x) \, \dx\cdot  W\quad	\text{as } t \to \infty
\end{equation}
in probability. 
If, additionally, \eqref{eq:Z log Z} holds, then the convergence in \eqref{eq:LLN CMJ} holds in $L^1$.
To see this, first recall that \eqref{eq:Z log Z} implies $\E[W]=1$ and hence $e^{-\alpha t} \E[\cZ^{\varphi}_t]$ converges
by the two-sided version of the key renewal theorem \cite[Satz 2.5.3]{Alsmeyer:1991}
to $\beta^{-1} \int e^{-\alpha x} \E[\varphi](x) \, \dx$, which is the expectation of the random variable on the right-hand side of \eqref{eq:LLN CMJ}.
If $\varphi$ is nonnegative, then the convergence of the first moment
in combination with convergence in probability gives the convergence in $L^1$ by Proposition 4.12 in \cite{Kallenberg:2002}.
The case of general $\varphi$ can be reduced to the case of nonnegative $\varphi$ using the decomposition $\varphi = \varphi_+-\varphi_-$
of $\varphi$ into its positive part minus its negative part.

What is more, (A\ref{ass:second moment}) implies $\L\mu(\frac\alpha2) < \infty$ and hence the holomorphy of $\L\mu$
on the half-space $\Real(z) > \frac\alpha2$, which implies that all higher derivatives of $\L\mu$ in the point $z=\alpha$
exist. This in turn implies (5.4) in \cite{Nerman:1981}
(for instance with $g(t) = 1 \wedge t^{-2}$ there).
Hence,
Conditions 5.1 of \cite{Nerman:1981}, 3.2 of \cite{Gatzouras:2000} and (3.2) and (3.4) of \cite{Meiners:2010} are satisfied.
This ensures
that the convergence in \eqref{eq:LLN CMJ} holds in the almost sure sense
and in $L^1$ provided that
\begin{itemize}
	\item
$\varphi$ vanishes on $(-\infty,0)$ and satisfies Condition 5.2 of \cite{Nerman:1981} in the non-lattice case;
	\item

$\varphi$ vanishes on $(-\infty,0)$ and satisfies Condition 3.1 in \cite{Gatzouras:2000} in the lattice case or
	\item
$\varphi$ satisfies Eq.\ (3.3) in \cite{Meiners:2010} in the case that it does not vanish on the negative half-line.
\end{itemize}

The next two assumptions are conditions on the second moments of the characteristic $\varphi$.
\begin{enumerate}[{\bf{(A}1)}]
	\setcounter{enumi}{\value{assumptions}}
	\item
	$\varphi(t) \in L^2$ for every $t \in \R$ and  $t \mapsto \Var[\varphi](t)e^{-\alpha t}$ 
is directly Riemann integrable.
	\label{ass:variance growth}
	\setcounter{assumptions}{\value{enumi}}
\end{enumerate}

\begin{enumerate}[{\bf{(A}1)}]
\setcounter{enumi}{\value{assumptions}}
	\item
	For any $t \in \R$ there is an $\varepsilon > 0$ such that the family
	\begin{equation*}	\textstyle
	(|\varphi(x)|^2)_{|x-t| \leq \varepsilon}	\quad	\text{is uniformly integrable.}
	\end{equation*}
	\label{ass:local ui of phi^2}
	\setcounter{assumptions}{\value{enumi}}
\end{enumerate}

Notice that if $\varphi$ is deterministic real-valued, then (A\ref{ass:local ui of phi^2}) holds since $\varphi$ is c\`adl\`ag, in particular,
locally bounded.

In the lattice case, if $t \in \Z$, then $t-S(u) \in \Z$ for all individuals $u$ with $S(u)<\infty$.
Then $\cZ_t^\varphi$ depends only on the values $\varphi_u(x)$ for $x \in \Z$ ($u \in \I$).
In particular, the values of $\varphi$ on $\R \setminus \Z$ are irrelevant for our purposes.
Therefore, in the lattice case, we make the assumption that $\varphi$ has paths that are
constant on intervals of the form $[n,n+1)$, $n \in \Z$.
With this assumption, condition (A\ref{ass:local ui of phi^2}) is meaningful also in the lattice case,
but reduces to the condition that $\varphi(x) \in L^2$ for all $x \in \Z$,
a condition  contained in (A\ref{ass:variance growth}).

We continue with a proposition giving sufficient conditions for the general branching process counted with characteristic $\varphi$
to be well-defined. Before this, we introduce the notion of an \emph{admissible ordering} of $\I$.
We call a sequence $u_1,u_2,\ldots \in \I$ an admissible ordering of $\I$ if
\begin{itemize}
	\item	$\I_n \defeq \{u_1,\ldots,u_n\}$ is a subtree of the Ulam-Harris tree $\I$ of cardinality $n$,
	\item	$\I = \bigcup_{n \in \N} \I_n$.
\end{itemize}
Admissible orderings exist.
Indeed, we can construct $(u_n)_{n \in \N}$ recursively.
First, let $u_1=\varnothing$. If we have constructed $u_i$ for $i =1,\ldots,2^k$ where $k \in \N_0$,
then, for any $2^k <i \leq 2^{k+1}$,
we set $u_i \defeq u_{i-2^k}j$ with the smallest $j \in \N$ such that $u_{i-2^k}j \not \in \{v_1,\ldots,u_{2^k}\}$,
see Figure \ref{fig:ordering}.
\begin{figure}[h]
	\includegraphics[clip, trim=3.7cm 21cm 3.7cm 3cm, width=1.00\textwidth]{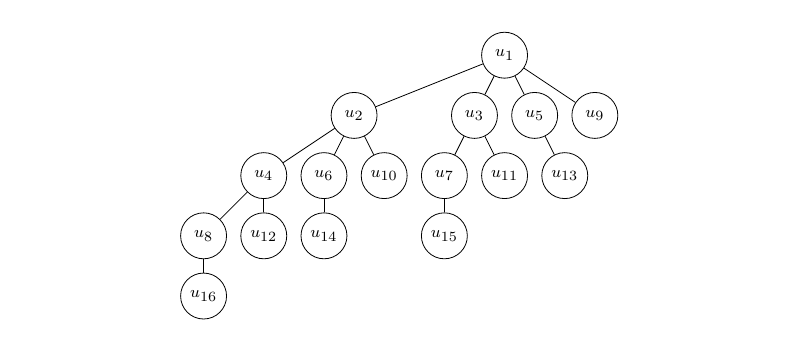}
	\caption{The subtree $\I_{16}$ for the particular admissible ordering of $\I$ given above.}
	\label{fig:ordering}
\end{figure}

 Recall that a series $\sum_{n\in\N} x_n$ in a Banach space $(X,\|\cdot\|)$ is said to converge unconditionally if, for any $\varepsilon >0$,
 there is a finite $I\subseteq\N$ such that $\| \sum_{n\in J} x_n\|<\varepsilon$ for any finite $J\subseteq \N \setminus I$.
 An equivalent definition is that the series converges for any rearrangement.
 For this and other characterizations we refer the reader to \cite{Hildebrandt:1940}.

\begin{proposition}	\label{Prop:L1 existence of Z_t^varphi}
Suppose that (A\ref{ass:Malthusian parameter}) holds and that $\varphi$ is a random characteristic satisfying
(A\ref{ass:mean growth}) and (A\ref{ass:variance growth}).
Then, for every $t \in \R$,
\begin{equation*}
\cZ_t^\varphi \defeq \sum_{u\in\I} \varphi_{u}(t-S(u))
\end{equation*}
converges unconditionally in $L^1$
  and almost surely over every admissible ordering of $\I$.
\end{proposition}
The proof of the proposition will be given in Section \ref{subsec:expectation}.

\begin{remark}	\label{Rem:L1 existence of Z_t^varphi}
Notice that $\cZ_t^\varphi = \cZ_t^{\varphi\1_{(-\infty,t]}}$ and thus from the proposition,
we infer that $\cZ_t^\varphi$ converges unconditionally in $L^1$ and almost surely for every $t \in \R$
if (A\ref{ass:Malthusian parameter}) holds and, for every $t \in \R$, $\varphi \1_{(-\infty,t]}$
satisfies (A\ref{ass:mean growth}) and (A\ref{ass:variance growth}).
\end{remark}

\begin{remark}
Notice that by Proposition \ref{Prop:L1 existence of Z_t^varphi},
the process $\cZ^\varphi$ is defined almost surely for any fixed $t\in\R$.
In other words, it is defined only up to a modification.
\end{remark}

\begin{remark}	\label{Rem:dRi}
Notice that if the random characteristics $\varphi$ and $\psi$ satisfy condition (A\ref{ass:local ui of phi^2}),
then so does any linear combination of them.
Further, by the dominated convergence theorem,
both the expectation function and the variance function of any linear combination of $\varphi$ and $\psi$
are c\`adl\`ag.
This particularly implies that these functions are locally bounded and a.e.\ continuous.
Consequently, if, in addition to (A\ref{ass:local ui of phi^2}), also (A\ref{ass:variance growth}) holds for
$\varphi$ and $\psi$, then (A\ref{ass:variance growth})
also holds for any linear combination of $\varphi$ and $\psi$.
Indeed, for $\beta_1, \beta_2\in\R$,
\begin{equation*}
\Var[\beta_1\varphi(t)+\beta_2\psi(t)]e^{-\alpha t}\leq 2\beta_1^2 \Var[\varphi(t)]e^{-\alpha t}+2\beta_2^2 \Var[\psi(t)]e^{-\alpha t},
\quad t\in\R.
\end{equation*}
By \cite[Remark 3.10.5 on p.~237]{Resnick:1992},
the function in focus is directly Riemann integrable as a locally Riemann integrable function
dominated by a directly Riemann integrable function.
\end{remark}

The next proposition gives sufficient conditions for
(A\ref{ass:mean growth}), (A\ref{ass:variance growth}) and (A\ref{ass:local ui of phi^2}).
To formulate it, we introduce the following notation.
If $f: \R \to \R$ is a function, we set $f^*(t) \defeq \sup_{|x-t| \leq 1} |f(x)|$.
This notation extends immediately (pathwise) to random c\`adl\`ag functions such as random characteristics $\varphi$.

\begin{proposition}	\label{Prop:f^*}
Suppose that (A\ref{ass:Malthusian parameter}) holds.
\begin{enumerate}[(a)]
	\item
		If $f:\R \to \R$ is c\`adl\`ag and $\int f^*(x) \, \dx < \infty$,
		then $f$ is directly Riemann integrable.
		Conversely, if $f:\R \to \R$ is directly Riemann integrable,
		then so is $f^*$. ((A\ref{ass:Malthusian parameter}) is not needed.)
	\item
		If a random characteristic $\varphi$ satisfies
		\begin{equation}	\label{eq:int Evarphi*e^-alphax dx finite}
		\int \E[\varphi^*](x) e^{-\alpha x} \, \dx < \infty,
		\end{equation}
		then (A\ref{ass:mean growth}) holds for $\varphi$.
	\item
		If a random characteristic $\varphi$ satisfies
		\begin{equation}	\label{eq:int Evarphi*^2e^-alphax dx finite}
		\int \E\big[(\varphi^*)^2 \big](x) e^{-\alpha x} \, \dx < \infty,
		\end{equation}
		then $\varphi$ also satisfies (A\ref{ass:variance growth}) and (A\ref{ass:local ui of phi^2}).
\end{enumerate}
\end{proposition}

By $\Lambda$ we denote the set of solutions to the equation
\begin{align}	\label{eq:roots}
\L\mu(\lambda)=1
\end{align}
such that $\Real(\lambda) > \frac\alpha2$ and by $\partial\Lambda$
we denote the set of roots on the \emph{critical line} $\Real(\lambda) = \frac\alpha 2$.
In the lattice case, $\L\mu$ is $2\pi \imag$-periodic, and we define
$\Lambda$ to be the set of $\lambda$ with $\Real(\lambda) > \frac\alpha2$ satisfying \eqref{eq:roots}
and $\Imag(\lambda) \in (-\pi, \pi]$.
Analogously, in this case, $\partial \Lambda$ denotes the set of roots $\lambda$ with $\Real(\lambda) = \frac\alpha 2$
satisfying $\Imag(\lambda) \in (-\pi, \pi]$.
Finally, in both the non-lattice and the lattice case, we set $\Lambda_{\geq} \defeq \Lambda \cup \partial \Lambda$.
Notice that $\alpha \in \Lambda$ and that every other element $\lambda \in \Lambda_{\geq}, \lambda \not= \alpha$
satisfies $\Real(\lambda) \in [\frac\alpha2,\alpha)$ and $\Imag(\lambda) \not = 0$.
Further, $\lambda = \theta + \imag \eta \in \Lambda_{\geq}$ implies
that the complex conjugate $\overline{\lambda} = \theta - \imag \eta \in \Lambda_\geq$
except if $\eta = \pi$ in the lattice case.

Although one may
consider cases where $\Lambda_{\geq}$ contains infinitely many elements,
in all relevant examples $\Lambda_{\geq}$ is finite.
Therefore, and for simplicity, we assume throughout the paper the following:
\begin{enumerate}[{\bf{(A}1)}]
	\setcounter{enumi}{\value{assumptions}}
	\item
	The set of roots $\Lambda_{\geq}$ is finite.	\label{ass:Lambda finite}
	\setcounter{assumptions}{\value{enumi}}
\end{enumerate}
We stress that if $\mu$ has a density with respect to the Lebesgue measure and (A\ref{ass:first moment}) holds,
then also (A\ref{ass:Lambda finite}) holds.
This is justified in the proof of Lemma \ref{Lem:asymptotics of E[N(t)] non-lattice}
 in combination with Remark \ref{Rem:Riemann-Lebesgue}.

\subsection{Main results}		\label{subsec:main results}

For each root $\lambda \in \C$ of the function $\L\mu - 1$, we denote its multiplicity by $k(\lambda) \in \N$.
Then, for any $j=0,\ldots,k(\lambda)-1$,
we can define
\begin{equation}	\label{eq:W_t^(j)(lambda)}
W_t^{(j)}(\lambda) \defeq (-1)^{j} \sum_{u \in \cC_t} S(u)^j e^{-\lambda S(u)},	\qquad	t \in \R,
\end{equation}
where $\cC_t$ is the coming generation at time $t$ formally defined in \eqref{eq:coming generation}.
The Malthusian parameter $\alpha>0$ is a root of multiplicity $1$
and gives rise to one martingale, namely, Nerman's martingale $(W_t)_{t \in \R} = (W_t(\alpha))_{t \in \R}$
defined in \eqref{eq:Nerman's martingale},
which is of great importance in the law of large numbers for the general branching process.
On the other hand,
the martingales corresponding to $\lambda \in \Lambda$
are relevant in the central limit theorem.

\begin{theorem}	\label{Thm:martingale convergence}
Suppose that (A\ref{ass:Malthusian parameter}) through (A\ref{ass:second moment}) hold.
Then, for any $\lambda \in \Lambda$ and $j=0,\ldots,k(\lambda)-1$,
the process $(W^{(j)}_t(\lambda))_{t\ge0}$ is a martingale and there is a random variable $W^{(j)}(\lambda) \in L^2$
such that
\begin{equation*}
W_t^{(j)}(\lambda) \to W^{(j)}(\lambda)	\quad	\text{a.\,s.\ and in } L^2	\text{ as } t \to \infty.
\end{equation*}
\end{theorem}

There are more technicalities to deal with before the main result (Theorem \ref{Thm:main} below) can be stated in its most general form.
Therefore, we shall 
first present illustrative special cases through
Theorems \ref{Thm:no roots}, \ref{Thm:main non-lattice} and \ref{Thm:main lattice}
(the proofs indeed reveal that they are corollaries of our main result, Theorem \ref{Thm:main}). 
We start with the non-lattice case. To this end, we need one more piece of notation.
For a function $f:\R \mapsto \R$ we define the total variation function $\mathrm{V}\!f$ by
\begin{equation}	\label{eq:def of Vf}
\mathrm{V}\!f(x) \defeq \sup\bigg\{\sum_{j=1}^n |f(x_j)-f(x_{j-1})|: -\infty < x_0 < \ldots < x_n \leq x,\ n \in \N \bigg\}
\end{equation}
for $x \in \R$.
In some theorems related to the non-lattice case we require the following additional assumption on the characteristic $\varphi$:
\begin{align}	\label{eq:VEphi integrable}
\int (\mathrm{V}\!\E[\varphi])(x) \big(e^{-\vartheta x} +e^{-\alpha x}\big) \, \dx < \infty
\end{align}
for some $\vartheta < \frac\alpha2$.
We use the symbol $\vartheta$ both in (A\ref{ass:first moment}) and in \eqref{eq:VEphi integrable}
to denote some parameter $<\frac\alpha2$ at which the corresponding condition is satisfied.
If we make both assumptions at the same time, there is no harm in assuming that the $\vartheta$'s coincide,
which is why we do not distinguish them by our notation.

\begin{theorem}	\label{Thm:no roots}
Suppose that (A\ref{ass:Malthusian parameter}) through (A\ref{ass:second moment}) hold and that the intensity measure $\mu$ has a density
with respect to the Lebesgue measure.
Further, suppose that there are no roots of the equation $\L\mu(z)=1$
in the strip $\vartheta < \Real(z) < \alpha$.	
Then, for any characteristic $\varphi$ satisfying
(A\ref{ass:variance growth}), (A\ref{ass:local ui of phi^2}) and \eqref{eq:VEphi integrable},
there exists $\sigma \geq 0$ such that,
for $a_\alpha \defeq \beta^{-1} \int \E[\varphi(x)] e^{-\alpha x} \, \dx$ and
a standard normal random variable $\cN$ independent of $W$,
\begin{align*}
e^{-\frac\alpha2 t} \big(\cZ_t^\varphi - a_\alpha e^{\alpha t} W \big)
\distto \sigma  \sqrt{\tfrac{W}{\beta}} \cN
\quad	\text{as } t \to \infty.
\end{align*}
The constant $\sigma$ can be explicitly computed,
see the formula \eqref{eq:sigma} given in Theorem \ref{Thm:main} below.
\end{theorem}

\begin{theorem}	\label{Thm:main non-lattice}
Suppose that (A\ref{ass:Malthusian parameter}) through (A\ref{ass:second moment}) hold
and that the intensity measure $\mu$ has a density with respect to the Lebesgue measure.
Then (A\ref{ass:Lambda finite}) holds and
there are $b_{\lambda,l}$, $l=0,\ldots,k(\lambda)-1$, $\lambda \in \Lambda_{\geq}$
satisfying $b_{\overline{\lambda},l} = \overline{b_{\lambda,l}}$
such that for any characteristic $\varphi$ satisfying
(A\ref{ass:variance growth}), (A\ref{ass:local ui of phi^2}) and \eqref{eq:VEphi integrable}
there exists $\sigma \geq 0$ such that, for a standard normal random variable $\cN$ independent of $W$,
the following assertions hold.
\begin{enumerate}[(i)]
	\item
		If there are no roots of $\L\mu(z)=1$ on the critical line $\Real(z)=\frac\alpha2$, then
		\begin{align*}
		e^{-\frac\alpha2 t} \bigg(\cZ_t^\varphi-\sum_{\lambda\in\Lambda} e^{\lambda t}
		\sum_{l=0}^{k(\lambda)-1} b_{\lambda,l}
		\sum_{j=0}^l \binom{l}{j}W^{(j)}(\lambda)\int(t\!-\!x)^{l-j} \E[\varphi(x)] e^{-\lambda x}\, \dx \bigg)	\\
		\distto \sigma  \sqrt{\tfrac{W}{\beta}} 
\cN
		\end{align*}
		as $t \to \infty$.
        \item
		Otherwise, let $k \in \N$ be the maximal multiplicity $k(\lambda)$ of a root $\lambda \in \partial \Lambda$. Then
		\begin{align*}
		e^{-\frac\alpha2 t}t^{-k+\frac12} \bigg(\cZ_t^\varphi-\sum_{\lambda\in\Lambda} e^{\lambda t}
		\sum_{l=0}^{k(\lambda)-1} b_{\lambda,l}
		\sum_{j=0}^l \binom{l}{j} W^{(j)}(\lambda)
		\int (t\!-\!x)^{l-j} \E[\varphi(x)] e^{-\lambda x} \, \dx \bigg)	\\
		\distto \sigma \sqrt{\tfrac{W}{\beta}} \cN
		\end{align*}
		as $t \to \infty$.
\end{enumerate}
\end{theorem}

 The lattice analogue of Theorem \ref{Thm:main non-lattice} is given next.
\begin{theorem}	\label{Thm:main lattice}
Suppose that (A\ref{ass:Malthusian parameter}) through (A\ref{ass:second moment}) hold
and that the intensity measure $\mu$ is lattice with span $1$.
Then there are $b_{\lambda,l}$, $l=0,\ldots,k(\lambda)-1$, $\lambda \in \Lambda_{\geq}$
satisfying $b_{\overline{\lambda},l} = \overline{b_{\lambda,l}}$
such that for any characteristic $\varphi$ satisfying
\begin{equation*}
\sum_{n \in \Z} |\E[\varphi(n)]|(e^{-\vartheta n}+e^{-\alpha n})<\infty,
\end{equation*}
for some $\vartheta <\frac \alpha 2$ and
\begin{align*}
\sum_{n\in\Z}\Var[\varphi](n)e^{-\alpha n}<\infty
\end{align*}
there exists $\sigma \geq 0$ such that, for a standard normal random variable $\cN$ independent of $W$,
the following assertions hold.
\begin{enumerate}[(i)]
	\item
	If there are no roots of $\L\mu(z)=1$ on the critical line $\Real(z)=\frac\alpha2$, then
	\begin{align*}
		e^{-\frac\alpha2 t} \bigg(\cZ_t^\varphi-\sum_{\lambda\in\Lambda} e^{\lambda t}
		\sum_{l=0}^{k(\lambda)-1} b_{\lambda,l}
		\sum_{j=0}^l \binom{l}{j}W^{(j)}(\lambda)\sum_{n\in\Z}(t\!-\!n)^{l-j} \E[\varphi(n)] e^{-\lambda n}\ \bigg)	\\
		\distto \sigma  \sqrt{\tfrac{W}{\beta}}
		\cN
	\end{align*}
	as $t \to \infty$, $t \in \N$.
	\item
	Otherwise, let $k \in \N$ be the maximal multiplicity $k(\lambda)$ of the roots $\lambda \in \partial \Lambda$. Then
	\begin{align*}
	e^{-\frac\alpha2 t}t^{-k+\frac12} \bigg(\cZ_t^\varphi-\sum_{\lambda\in\Lambda} e^{\lambda t}
	\sum_{l=0}^{k(\lambda)-1} b_{\lambda,l}
	\sum_{j=0}^l \binom{l}{j} W^{(j)}(\lambda)
	\sum_{n\in\Z} (t\!-\!n)^{l-j} \E[\varphi(n)] e^{-\lambda n} \bigg)	\\
	\distto \sigma \sqrt{\tfrac{W}{\beta}}
	\cN
	\end{align*}
	as $t \to \infty$, $t \in \N$.
\end{enumerate}
\end{theorem}

\begin{remark}
In Theorems \ref{Thm:main non-lattice} and \ref{Thm:main lattice},
we do not exclude the case $\sigma=0$.
There, a more precise limit theorem can be derived with the help of Theorem \ref{Thm:main}.
In particular, the expression in parentheses in (i) is
just a deterministic function of the order $O(e^{\theta t})$ for some $\theta < \frac\alpha2$ (cf.\ Theorem \ref{Thm:main}(i)).
In case (ii), we need to additionally subtract a linear combination of $e^{\lambda t} t^l$,
where $\lambda$ runs over the roots on the critical line and $l < k(\lambda)$
with $k(\lambda)$ denoting the order of the root,
and to use a different normalization in order to get a nontrivial limit.
\end{remark}

\begin{remark}
The constants $b_{\lambda,l}$, $l=0,\ldots,k(\lambda)-1$, $\lambda \in \Lambda_{\geq}$
can be computed using Proposition~\ref{Prop:determining the coefficients}.
\end{remark}

\begin{remark}
Note that the  fluctuations in Theorem~\ref{Thm:main lattice} and \ref{Thm:main non-lattice} are very similar to fluctuations obtain by Janson \cite{Janson:2004} in the multi-type case where the birth times constitutes a homogeneous Poisson point process. This indicates that there may be a general theorem describing the fluctuations of  multitype CMJ process, which covers the aforementioned models (cf.~Open Problem \ref{OP:multitype_CMJ} in Section \ref{sec:discussion}).
\end{remark}

%

In order to obtain
the asymptotic expansion of $\cZ^\varphi_t$  as presented in Theorems~\ref{Thm:no roots},
\ref{Thm:main non-lattice} and \ref{Thm:main lattice},
we first need an expansion for the mean $m^{\varphi}_t \defeq \E[\cZ^\varphi_t]$, $t \in  \G$.
It turns out that, under suitable assumptions, the following holds:
\begin{align}	\label{eq:expansion of mean}
m^{\varphi}_t = \1_{[0,\infty)}(t) \sum_{\lambda\in\Lambda_{\geq}} e^{\lambda t} \sum_{l=0}^{k(\lambda)-1} a_{\lambda,l} t^l + r(t),
\qquad	 t\in\G
\end{align}
for some constants $a_{\lambda,l} \in\C$ and
 a function $r$ satisfying $|r(t)| \leq C e^{\alpha t/2}/(1+t^2)$ for all $t \in \G$
and some finite constant $C \geq 0$.

 We shall provide three different sets of sufficient conditions for \eqref{eq:expansion of mean} to hold.
The first case is when the characteristic $\varphi$ is chosen in such a way that $\cZ^\varphi_t$ is a rescaled martingale
(see Section~\ref{sec:Nerman's martingales as CMJ processes}). In this case, \eqref{eq:expansion of mean} holds trivially.
The second case is when $\G=\Z$. Then 
expansion \eqref{eq:expansion of mean} is obtained in Lemma~\ref{Lem:lattice asymptotic of the mean} via generating functions. The third case is the non-lattice case where under
the additional (technical) assumption \eqref{eq:Lmu becomes small},
expansion \eqref{eq:expansion of mean} is derived in Lemma~\ref{Lem:nonlattice asymptotic of the mean}.
There might be more examples of $\cZ^\varphi_t$ that are not covered by any of the three sufficient conditions,
even though the corresponding expansion of $\E[\cZ^\varphi_t]$ is of 
form \eqref{eq:expansion of mean}.
For this reason, we formulate our main result, Theorem \ref{Thm:main},
for processes $\cZ_t^\varphi$ for which $\E[\cZ^\varphi_t]$ satisfies \eqref{eq:expansion of mean}.
What is more, in some examples, it is in fact more convenient to directly check that \eqref{eq:expansion of mean} holds
rather than checking the assumptions of Theorem \ref{Thm:main non-lattice} or \ref{Thm:main lattice},
see e.\,g.~Sections~\ref{subsec:GWP} and \ref{subsec:Nerman's martingales}.

From \eqref{eq:expansion of mean} we can obtain an asymptotic expansion
of $\cZ_t^\varphi$
where the principal terms are of the form a constant times $e^{\lambda t} t^j W^{(j)}(\lambda)$
for $\lambda \in \Lambda$ and $j=0,\ldots,k(\lambda)-1$.
More precisely, the principal terms are given by the expression
\begin{align}	\label{eq:H_Lambda}
H_\Lambda(t) \defeq \sum_{\lambda \in \Lambda} e^{\lambda t}
\sum_{l=0}^{k(\lambda)-1} \sum_{j=0}^{l} a_{\lambda,l} \binom{l}{j} t^j W^{(l-j)}(\lambda).
\end{align}
If, additionally, there are roots $\lambda \in \partial \Lambda$,
then the next terms in the expansion are given by the following deterministic sum
\begin{align}	\label{eq:H_dLambda}
H_{\partial \Lambda}(t) \defeq \sum_{\lambda\in\partial\Lambda} e^{\lambda t} \sum_{l=0}^{k(\lambda)-1} a_{\lambda,l} t^l,	\qquad	t \in \R.
\end{align}
(Of course, if $\Lambda_\geq = \partial \Lambda \cup \{\alpha\}$,
the terms from $H_{\partial \Lambda}(t)$ are the subleading terms.)
We set $H(t) \defeq H_\Lambda(t) + H_{\partial \Lambda}(t)$, $t \in \R$.
Further, for any $\lambda \in \partial\Lambda$ and $l=0,\ldots,k(\lambda)-1$,
we define a random variable $R_{\lambda,l}$ by
\begin{align}
	\label{eq:definition of R_lambda}
R_{\lambda,l}
\defeq \sum_{j=l}^{k(\lambda)-1}a_{\lambda,j}\binom{j}{l} \sum_{k=1}^N (- X_k)^{j-l} e^{-\lambda X_k}.
\end{align}
Assumption (A\ref{ass:second moment}) guarantees that $R_{\lambda,l} \in L^2$ for all $\lambda \in \partial \Lambda$
and $l=0,\ldots,k(\lambda)-1$.
We may thus define
\begin{equation*}
\rho_l^2 \defeq \sum_{\substack{\lambda \in \partial\Lambda:\\ k(\lambda) > l}} \Var[R_{\lambda,l}]
\end{equation*}
where $\Var[R_{\lambda,l}]=\E[|R_{\lambda,l}|^2]-|\E [R_{\lambda,l}]|^2$.
In general, throughout the paper, if $Y$ is a complex-valued random variable with finite mean,
we set $\Var[Y] \defeq \E[|Y-\E[Y]|^2]$.

As a final preparation for our main result,
we recall the fact that if a sequence of random variables $(Y_n)_{n \in \N_0}$
defined on $(\Omega,\F,\Prob)$
converges in distribution to some random variable $Y$, this convergence is said to be \emph{stable}
if for all continuity points $y$ of the distribution function of $Y$ and all $E \in \F$,
the limit $\lim_{n \to \infty} \Prob(\{Y_n \leq y\} \cap E)$ exists.
In this case, we write $Y_n \stablyto Y$.
An alternative characterization is the following.
There is a copy $Y^*$ of $Y$ defined on some extension of the probability space
$(\Omega,\F,\Prob)$ such that, for every $\F$-measurable random variable $X$, it holds that
\begin{equation*}
(Y_n,X) \distto (Y^*,X)\text{ as }n \to \infty,
\end{equation*}
see \cite[Condition (B')]{Aldous+Eagleson:1978}.
Without loss of generality we may and will assume that $Y=Y^*$ whenever we write $Y_n \stablyto Y$.
The second definition is more convenient as it allows to manipulate the sequence $Y_n$ (e.g.\ by multiplying with another random variable)
without loosing the convergence in distribution. In Remark \ref{Rem:application of stable convergence} below we prove \eqref{eq:partic case}
using such an argument.

\begin{remark}
Although it is not stated explicitly but it follows from the proofs that the convergences in
Theorems \ref{Thm:no roots}, \ref{Thm:main non-lattice} and \ref{Thm:main lattice} are in fact stable.
\end{remark}

For a measurable function $f$  and a measure $\nu$ (possibly random) on the Borel $\sigma$-field,
we write $f* \nu = \nu * f$ for the Lebesgue-Stieltjes convolution of $f$ and  $\nu$,
i.e.,  $f*\nu(t) = \nu*f(t) = \int f(t-x) \, \nu(\dx)$ whenever the integral exists.
In particular,
\begin{equation*}
f*\xi(t)=\int f(t-x)\, \xi(\dx)=\sum_{j=1}^N f(t-X_j),\quad t\in\R.
\end{equation*}
\begin{theorem}	\label{Thm:main}
Suppose that $\xi$ satisfies (A\ref{ass:Malthusian parameter}) through (A\ref{ass:second moment})  and (A\ref{ass:Lambda finite})
and that the real-valued characteristic $\varphi$ satisfies
 (A\ref{ass:mean growth}) through (A\ref{ass:local ui of phi^2}).
Further, assume that $m^\varphi_t$ satisfies \eqref{eq:expansion of mean}  with $\sup_{t \in \G} (1+t^2)e^{-\frac{\alpha}{2}t}|r(t)| < \infty$,
and let $n \defeq \max\{l \in \N_0: \rho_l > 0\}$ with $n=-1$ if the set is empty;
in that case we set $\rho_{-1} \defeq 0$.
Then there exists a finite constant $\sigma \geq 0$ such that, with
\begin{equation*}	\textstyle
a_t^2 \defeq \sigma^2 + \frac{\rho_n^2}{2n+1}t^{2n+1},	\qquad	t > 0,
\end{equation*}
it holds
\begin{enumerate}[(i)]
	\item	if $\sigma^2 = \rho_n^2 = 0$, then, $ t \mapsto H(t)+r(t)$ is a  c\`adl\`ag modification of the process $\cZ^\varphi$,
	\item	if $\sigma^2 > 0$ or $\rho_n^2 > 0$, then
		\begin{align}	\label{eq:limit theorem}
		a_t^{-1} e^{-\frac\alpha2 t} \big(\cZ^\varphi_t-H(t)\big)
		\stablyto \sqrt{\tfrac{W}{\beta}} \cN
		\end{align}
		as $t\to\infty$, $t\in\G$ where $\cN$ is a standard normal random variable independent of  $\F$
		and $\beta$ is as defined in \eqref{eq:beta}.
\end{enumerate}
If $n=-1$ the constant $\sigma$ can be explicitly computed, namely,
\begin{align}	\label{eq:sigma}
\sigma^2 = \int\Var\big[\varphi(x)+h^\varphi*\xi(x)\big]e^{-\alpha x} \, \ell(\dx)
\end{align}
where
\begin{align*}
h^\varphi(t)\defeq m^\varphi_t-\sum_{\lambda\in\Lambda} e^{\lambda t} \sum_{l=0}^{k(\lambda)-1} a_{\lambda,l} t^l.
\end{align*}
\end{theorem}

In the situation of Theorem \ref{Thm:main}, the following remarks are in order.

\begin{remark}	\label{Rem:conjugate}
Observe that, in the non-lattice case, if $\lambda\in\Lambda_{\geq}$, then so is $\overline \lambda$
and $k(\overline \lambda)=k(\lambda)$. Moreover, as $m_t^{\varphi}$ is real for any $t\in\R$, we have
\begin{align*}
m_t^\varphi &=\sum_{\lambda\in\Lambda_{\geq}} \sum_{l=0}^{k(\lambda)-1} a_{\lambda,l} t^l e^{\lambda t} + r(t)
=\sum_{\lambda\in\Lambda_{\geq}} \sum_{l=0}^{k(\lambda)-1} a_{{\overline{\lambda}},l} t^l e^{\overline{\lambda} t} + r(t)	\\
&=\sum_{\lambda\in\Lambda_{\geq}} \sum_{l=0}^{k(\lambda)-1} \overline{a_{\overline\lambda,l}} t^l e^{\lambda t}+\overline{r(t)}
=\overline{m_t^\varphi},	\quad t \in \R,
\end{align*}
whence
\begin{align*}
0=m_t^\varphi -  \overline{m_t^\varphi}
= \sum_{\lambda\in\Lambda_{\geq}} \sum_{l=0}^{k(\lambda)-1} (a_{\lambda,l}-\overline{a_{\overline\lambda,l}}) t^l e^{\lambda t}
+ o(e^{\frac\alpha 2 t})	\quad	\text{as } t \to \infty,\ t \in \R.
\end{align*}
Next, we can choose $h>0$ such that the $e^{\lambda h}$, $\lambda\in\Lambda_{\geq}$ are distinct.
Recalling that $\Real(\lambda)\geq \frac\alpha2$ for $\lambda\in\Lambda_{\geq}$,
Lemma \ref{lem:asymptotic linear independent} gives that $a_{\lambda,l}-\overline{a_{\overline\lambda,l}}=0$,
that is, $\overline{a_{\lambda,l}}=a_{\overline\lambda,l}$. In particular, $r(t)$ is real for any $t\in\R$.

A similar reasoning in the lattice case gives $\overline{a_{\lambda,l}} = a_{\overline{\lambda},l}$ for all $\lambda \in \Lambda$ with $|\Imag(\lambda)|<\pi$
and $a_{\lambda,l} \in \R$ if $\Imag(\lambda)=\pi$.
\end{remark}

 \begin{remark}
 	\label{Rem:application of stable convergence}
Observe that for $N(t) \defeq \cZ_t^{\1_{[0,\infty)}}$,
the number of individuals born up to and including time $t$, by \eqref{eq:LLN CMJ}, we have
\begin{equation*}
e^{-\alpha t}N(t)\to \frac1\beta \int \limits_{[0,\infty)} e^{-\alpha x} \,\ell(\dx)\, W = \frac{c_\alpha}{\beta} W\quad\text{a.\,s.~as}~t\to\infty,~ t\in\G,
\end{equation*}
where $c_\alpha=(1-e^{-\alpha})^{-1}$ in the lattice case,
$c_\alpha=\alpha^{-1}$ in the non-lattice case.
The stable convergence in \eqref{eq:limit theorem} yields
\begin{equation}	\label{eq:partic case}
a_t^{-1} \sqrt{\frac{c_\alpha}{N(t)}} \big(\cZ^\varphi_t-H(t)\big)
\distto \cN	\quad	\text{as }t \to \infty	\quad\text{conditionally given }  \Surv.
\end{equation}
Indeed, with $G(t)\defeq a_t^{-1} e^{-\frac\alpha2 t} \big(\cZ^\varphi_t-H(t)\big)$,

\eqref{eq:partic case} is equivalent to
\begin{equation*}
\Prob(\Surv)^{-1}\cdot\E\Big[g\Big(\sqrt{\tfrac{c_\alpha e^{\alpha  t}}{ N(t)}}G(t)\Big)\1_\Surv\Big]\to\E[g(\cN)]
\end{equation*}
as $t\to\infty$, for any continuous, nonnegative function $g$ bounded by $1$.
With $F(t)\defeq\sqrt{\frac{c_{\alpha}We^{\alpha t}}{\beta N(t)}}\1_\Surv+\1_{\Surv^c}$,
which is well-defined since $N(t)>0$ on $\Surv$,
the above convergence can be rewritten as
\begin{align}	\label{eq:conditional convergence}
\Prob(\Surv)^{-1}\cdot\E\Big[g\Big(\sqrt \beta\tfrac{F(t)G(t)}{\sqrt {W}}\Big)\1_\Surv\Big]\to\E[g(\cN)]\quad	\text{as }t \to \infty.
\end{align}
For any fixed $\varepsilon>0$, there is some $\delta>0$ such that $\Prob(\Surv\cap\{W<\delta\})=\Prob(0<W<\delta)\le \varepsilon$  and consequently
\begin{align*}
	\Big|\E\Big[g\Big(\sqrt\beta\tfrac{F(t)G(t)}{\sqrt {W}}\Big)\1_\Surv\Big]-\E\Big[g\Big(\sqrt\beta\tfrac{F(t)G(t)}{\sqrt {W\vee \delta}}\Big)\1_\Surv\Big]\Big|\le \varepsilon
\end{align*}
as well as
\begin{align*}
	\Big|\E[g(\cN)\1_\Surv]-\E\Big[g\Big(\tfrac{\sqrt W \cN}{\sqrt {W\vee \delta}}\Big)\1_\Surv\Big]\Big|\le \varepsilon.
\end{align*}
On the other hand, \eqref{eq:limit theorem} yields
\begin{align*}
(G(t),W,\1_\Surv)\distto(\sqrt{\tfrac{W}{\beta}}\cN,W,\1_\Surv)\quad	\text{as }t \to \infty.
\end{align*}
Since $F(t)\to 1$ almost surely, Slutsky's theorem implies
\begin{align*}
(G(t),W,\1_\Surv,F(t))\distto(\sqrt{\tfrac{W}{\beta}}\cN,W,\1_\Surv,1)\quad	\text{as }t \to \infty.
\end{align*}
Taking advantage of the fact that the function $(u,v,x,y)\mapsto g(\sqrt\beta\frac{uy}{\sqrt{v\vee\delta}})(|x|\wedge1)$ is bounded and continuous, we conclude
\begin{align*}
	\E\Big[g\Big(\sqrt\beta\tfrac{F(t)G(t)}{\sqrt {W\vee \delta}}\Big)\1_\Surv\Big]\to \E\Big[g\Big(\tfrac{\sqrt W \cN}{\sqrt {W\vee \delta}}\Big)\1_\Surv\Big]
\end{align*}
as $t\to\infty$, and consequently
\begin{align*}
	\limsup_{t\to\infty}\Big|\E\Big[g\Big(\sqrt\beta\tfrac{F(t)G(t)}{\sqrt {W}}\Big)\1_\Surv\Big] -\E[g(\cN)\1_\Surv] \Big|\le 2\varepsilon.
\end{align*}
Letting now $\varepsilon$ to 0 and using the independence of $\cN$ and $\1_\Surv$
we get \eqref{eq:conditional convergence} and thereby \eqref{eq:partic case}.
\end{remark}

\begin{remark}	\label{Rem:sigma}
(i)
Notice that formula \eqref{eq:sigma} is not well-defined in the case $n \geq 0$ as then the integral diverges.
However, in this case, the exact value of $\sigma$ is irrelevant and can be set to $ \sigma \defeq 0$.	\smallskip

\noindent
(ii)
In the non-lattice case the variance $\sigma^2$ given by \eqref{eq:sigma} can be calculated using the bilateral Laplace transform $\L  h^\varphi$ of $ h^\varphi$.
Indeed, by Plancherel's theorem,
\begin{align*}
\sigma^2
&=	\int\Var[\varphi(x)+ h^\varphi*\xi(x)]e^{-\alpha x} \, \dx		\\
&=	\E\bigg[\int \big(\big(\varphi(x)\!-\!\E[\varphi](x)+( h^\varphi*(\xi\!-\!\mu))(x)\big) e^{-\frac\alpha2 x }\big)^2 \, \dx\bigg]	\\
&=	\frac{1}{2\pi}\E\bigg[\int\big|\L((\varphi(\cdot)\!-\!\E[\varphi](\cdot))e^{-\alpha \cdot/2 })(\imag \eta)+\L\big(((\xi\!-\!\mu)* h^\varphi)(\cdot)e^{-\frac\alpha2 \cdot}\big)(\imag\eta)\big|^2 \, \dd \eta\bigg]	\\
&=	\frac{1}{2\pi}\E\bigg[\int\big|\L(\varphi\!-\!\E[\varphi])(\tfrac\alpha2+\imag \eta)+\L\big(((\xi\!-\!\mu)* h^\varphi)(\cdot)\big)(\tfrac\alpha2+\imag\eta)\big|^2 \, \dd \eta\bigg]		\\
&=	\frac{1}{2\pi} \int \limits_{\Real(z)=\frac\alpha2}\Var[\L\varphi(z)+\L\xi(z)\L  h^\varphi(z)] \, |\dd z|,
\end{align*}
 where the variance of a complex random variable is defined in terms of absolute squares.
An analogous formula holds in the lattice case, see~\cite{Janson:2018}. We refrain from giving further details.	\smallskip

\noindent
(iii)
Suppose now that $\varphi$ vanishes on $(-\infty,0)$.
Then so do $m_t^\varphi$ and  the remainder function $r$ from the expansion \eqref{eq:expansion of mean}.
Additionally, assume that $r(t)=O(e^{(\frac \alpha 2 -\varepsilon) t})$ for some $\varepsilon>0$,
which holds in typical cases (see Section \ref{sec:asymptotic expansion of the mean}).
In the non-lattice case of Theorem \ref{Thm:main}, if
all the roots in $\Lambda$ are simple and there are no roots on the critical line $\{\Real(z)=\frac\alpha2\}$,
the bilateral Laplace transform $\L  h^\varphi$ of $ h^\varphi$ coincides on a neighborhood of $\{\Real(z)=\frac\alpha2\}$
with the function
\begin{equation*}
z\mapsto\frac{\L (\E[\varphi])(z)}{1-\L \mu(z)}.
\end{equation*}
To see this, notice that $ r(t)= m^\varphi_t-\sum_{\lambda\in\Lambda} a_{\lambda,0} e^{\lambda t}\1_{[0,\infty)}(t)$
and let  $h_{r}(t)\defeq h^\varphi(t)-r(t)$.
The bilateral Laplace transform $\L  r$ is well-defined on $\{\Real(z)>\frac\alpha 2-\varepsilon\}$.
Moreover, for $\Real(z)>\alpha$,
\begin{align*}
\L  r(z)=\L m^\varphi(z)-\sum_{\lambda\in\Lambda} \frac{a_{\lambda,0}}{z-\lambda}=
\frac{\L (\E[\varphi])(z)}{1-\L \mu(z)}-\sum_{\lambda\in\Lambda} \frac{a_{\lambda,0}}{z-\lambda}. 
\end{align*}
The right-hand side, being 
holomorphic  on $\{\Real(z)>\frac\alpha 2-\varepsilon\}$ (all the singularities are removable), 
coincides with $\L  r$ on that domain.
On the other hand, decreasing $ \varepsilon$ if needed 
we can and do 
assume that $\Lambda\subseteq\{\Real(z)>\frac\alpha 2+\varepsilon\}$. In particular, $\L h_{ r}(z)$ is well-defined  on $\{\Real(z)<\frac\alpha 2+\varepsilon\}$ and
equal to $\sum_{\lambda\in\Lambda} \frac{a_{\lambda,0}}{z-\lambda}$.
As a result, on the set $\{\frac\alpha 2-\varepsilon<\Real(z)<\frac\alpha 2+\varepsilon\}$
we have $\L  h^\varphi(z)=\L  r(z)+\L h_{ r}(z)=\frac{\L (\E[\varphi])(z)}{1-\L \mu(z)}$.
\end{remark}

\begin{remark}	\label{Rem:joint convergence}
Suppose that $\xi$ satisfies (A\ref{ass:Malthusian parameter}) through (A\ref{ass:second moment})
and that the real-valued characteristics $\varphi_1,\ldots,\varphi_d$ satisfy
(A\ref{ass:mean growth}) through (A\ref{ass:local ui of phi^2}).
Further, assume that each $m^{\varphi_j}_t$ satisfies \eqref{eq:expansion of mean}
(with coefficients $a_{\lambda,l}^j$ and remainder $r_j$ depending on $j$).
Then Theorem \ref{Thm:main} gives joint convergence in distribution
of the vector $(\cZ_t^{\varphi_1},\ldots,\cZ_t^{\varphi_d})$.
Indeed, by the Cram\'er-Wold device, convergence in distribution of the vector
is equivalent to convergence of all linear combinations of the form
\begin{equation*}
\sum_{j=1}^d c_j \cZ_t^{\varphi_j} = \cZ_t^{\sum_{j=1}^d c_j \varphi_j}.
\end{equation*}
A routine verification shows that the characteristic $\sum_{j=1}^d c_j \varphi_j$
satisfies the assumptions of Theorem \ref{Thm:main}.
\end{remark}

As a particular case of Remark \ref{Rem:joint convergence}
with $\varphi_j(\cdot) = \varphi(\cdot-s_j)$ for
$-\infty < s_1<\ldots < s_d < \infty$ and a given random characteristic $\varphi$,
we get the following result for the finite-dimensional distributions:

\begin{corollary}	\label{Cor:fd convergence}
In the situation of Theorem \ref{Thm:main} suppose that $\sigma^2 \not = 0$ or $\rho_n \not = 0$.
Then for $d \defeq (2n+1) \vee 0$,
\begin{equation*}
t^{-\frac d2} e^{-\frac\alpha2 t} \big(\cZ^\varphi_{t-s} -H(t-s)\big)_{s\in\R}
\findimto	\sqrt{\tfrac{W}{\beta}} (G_s)_{s\in\R}
\end{equation*}
where $(G_s)_{ s \in \R}$ is a centered Gaussian process  with the covariance function
\begin{align*}
\E[G_s G_u] =\int \Cov\big[\varphi(x-s)-h^\varphi*\xi(x-s),\varphi(x-u)-h^\varphi*\xi(x-u)\big] e^{-\alpha x} \, \ell(\dx)
\end{align*}
for any $s,u\in\R$ if $d=0$, i.e., $n=-1$ and
\begin{align*}
\E[G_sG_u]&=\frac{1}{d}
\sum_{\substack{\lambda \in \partial\Lambda:\\ k(\lambda)\geq n+1}} \Cov\bigg[\sum_{j=n}^{k(\lambda)-1}a_{\lambda,j}\binom{j}{n} \sum_{k=1}^N (- X_k-s)^{j-n} e^{-\lambda (X_k+s)},\\
&\hphantom{=\sum_{\substack{\lambda\in\partial\Lambda:\\k(\lambda) \geq n+1}}
	\Cov\bigg[}~\sum_{j=n}^{k(\lambda)-1}a_{\lambda,j}\binom{j}{n} \sum_{k=1}^N (- X_k-u)^{j-n} e^{-\lambda (X_k+u)}
\bigg]
\end{align*}
if $d=2n+1$ with $n \geq 0$.
\end{corollary}

We close this section with a figure displaying the way to the proofs of our main results.

\begin{figure}[H]
			\includegraphics[clip, trim=3.7cm 15cm 3.7cm 3.7cm, width=1.00\textwidth]{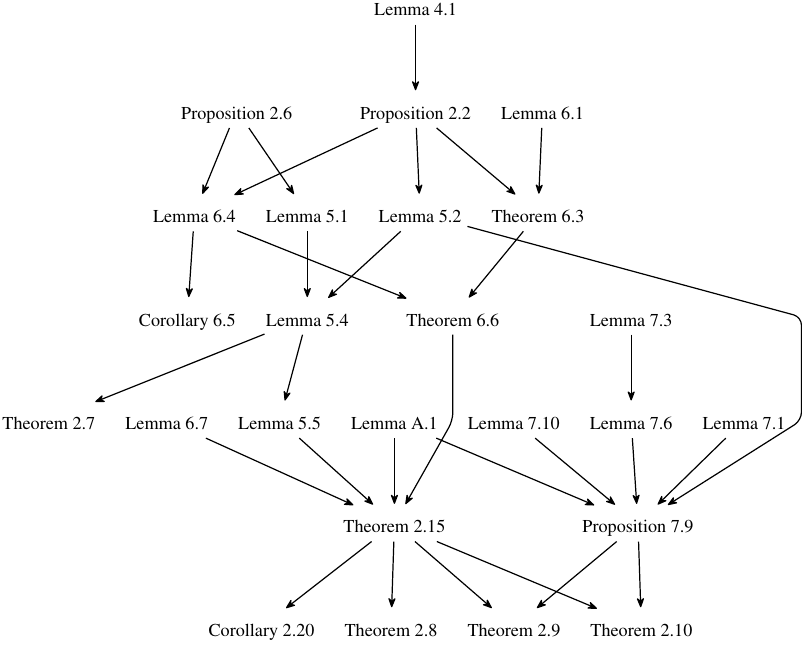}
	\caption{Diagram representing the dependence of the results.}
	\label{fig:dependence of the results}
\end{figure}

\section{Applications}	\label{sec:applications}

\subsection{The Galton-Watson process}	\label{subsec:GWP}

Consider a supercritical Gal\-ton-Watson branching process,
i.e., $\xi = \sum_{k=1}^N \delta_1 = N \delta_1$
where $N$ is a random variable taking values in $\N_0$ with
$m \defeq \E[N] \in (1,\infty)$ and $\E[N^2] < \infty$.
Then $\xi$ is lattice with span $1$.
Further,
\begin{equation*}
\L\mu(\lambda) = \E\bigg[\sum_{k=1}^N e^{-\lambda}\bigg] = m e^{-\lambda},	\quad	\lambda \in \C.
\end{equation*}
The equation $\L\mu(z) = 1$ is equivalent to $e^z = m$ and has only one solution
in the strip $\Imag(z) \in (-\pi, \pi ]
$.
In particular,
(A\ref{ass:Malthusian parameter}) holds,
i.e., there is a Malthusian parameter $\alpha > 0$, namely, $\alpha = \log m$,
and $\Lambda_{\geq} = \{\alpha\}$. In this case the parameter $\beta$  defined by \eqref{eq:beta} is equal to 1.
Then
\begin{equation*}
\E\bigg[\bigg( \sum_{k=1}^N e^{-\theta }\bigg)^{\!2} \bigg] = e^{-2\theta} \E[N^2] < \infty	\quad	\text{for all } \theta \in \R.
\end{equation*}
By Remark \ref{Rem:Janson's 2nd moment condition}, this implies that (A\ref{ass:first moment})
and (A\ref{ass:second moment}) hold.

Consider the characteristic $\phi(t) \defeq 
\1_{[0,1)}(t)$.
Then for any $n\in\N_{0}$, $\cZ_n^{\phi} $ is the number of individuals in the $n^{\mathrm{th}}$
generation and the corresponding Nerman's martingale \eqref{eq:Nerman's martingale} is the size of $n^{\mathrm{th}}$ generation
normalized by its expectation $m^{n}$, i.e., $W_{n}=e^{-\alpha n}\cZ_n^\phi$.
Clearly, $\phi$
satisfies (A\ref{ass:mean growth}), (A\ref{ass:variance growth}) and (A\ref{ass:local ui of phi^2}). 
Therefore, we may apply the lattice  version of Theorem  \ref{Thm:main} with $\rho_{-1}=0$, which yields
\begin{align*}
m^{-\frac{n}2} \big(m^{n} W_{n} - a_\alpha m^n W\big)
= e^{-\frac\alpha2 n} \big(\cZ_n^\phi - a_\alpha e^{\alpha n} W \big)
\distto \sigma \sqrt{W} \cN
\quad	\text{as } n \to \infty
\end{align*}
where $\cN$ is standard normal and independent of $W$.
To calculate $a_\alpha\defeq a_{\alpha, 0}$, we use \eqref{eq:expansion of mean}:
$m^{n} = \E[\cZ_n^\phi] = a_\alpha m^n$, i.e., $a_\alpha=1$.
Further, $\sigma > 0$ is given by \eqref{eq:sigma},
i.e.,
\begin{align*}
\sigma^2 = \sum_{n \in \Z} \Var\big[\phi(n)+h*\xi(n)\big]e^{-\alpha n}
=\sum_{n \in \Z} \Var\big[h*\xi(n)\big]m^{- n},
\end{align*}
where
\begin{align*}
h(n)= m^{\phi 
}_n - a_\alpha e^{\alpha n} = m^{n}\1_{\{n\ge 0\}}-m^n=-m^n\1_{\{n< 0\}}.
\end{align*}
Now, since $h*\xi(n)=Nh(n-1)$ we infer
\begin{align*}
\sigma^2
=\sum_{n<1} \Var[N]m^{2n-2}m^{- n}=\frac{\Var[N]}{m^2-m}.
\end{align*}
Consequently,
\begin{align*}
m^{\frac{n}2} \big(W_{n} - W\big)
\distto \Big(\frac{\Var[N] W}{m^2-m}\Big)^{\frac12} \cN
\quad	\text{as } n \to \infty.
\end{align*}
We have thus just rediscovered Heyde's classical central limit theorem for the martingale in the Galton-Watson process
\cite{Heyde:1970}.

We can also deal with the total number of individuals in the generations $0,\ldots,n$.
Indeed, this number is $\cZ_n^{f}$ for $f(t) \defeq \1_{[0,\infty)}(t)$, $t \in \R$,
which satisfies  (A\ref{ass:mean growth}), (A\ref{ass:variance growth}) and (A\ref{ass:local ui of phi^2}).
Invoking once again Theorem \ref{Thm:main} with $\rho_{-1}=0$ we obtain 
\begin{align*}
e^{-\frac\alpha2 n} \big(\cZ_n^{f} - a_\alpha e^{\alpha n} W \big)
\distto \sigma \sqrt{W} \cN
\quad	\text{as } n \to \infty.
\end{align*}
This time $a_\alpha$ can be computed with the help of \eqref{eq:expansion of mean} as follows.
We have the asymptotic expansion
\begin{equation*}
m_n^f = \E[\cZ_n^f] = \sum_{k=0}^n m^k = \frac{m^{n+1}-1}{m-1} = \frac{m}{m-1} e^{\alpha n}  - \frac1{m-1},
\end{equation*}
for $n\ge0$ and 0 otherwise.
Consequently, $a_\alpha = \frac{m}{m-1}$ and thereupon
\begin{align*}
m^{-\frac{n}2} \Big(\cZ_n^f - \frac{m^{n+1}}{m-1} W \Big) \distto \sigma \sqrt{W} \cN
\quad	\text{as } n \to \infty.
\end{align*}
This time $\sigma > 0$ is given by $\sigma^2=\sum_{n \in \Z} \Var[N]|h(n-1)|^2 m^{- n}$
with
\begin{align*}
h(n)
= m_n^f-\frac{m}{m-1} e^{\alpha n}.
\end{align*}
Therefore,	
\begin{align*}
\sigma^2
&= \Var[N]   \Big( \sum_{n <1  } \Big(\frac{m}{m-1} m^{ n-1}
\Big)^2 m^{- n} + \sum_{n \geq 1} \Big(\frac1{m-1}\Big)^2m^{- n}\Big)	\\
&=\frac1{(m-1)^{2}} \Var[N] \bigg(\sum_{n \le 0} m^{n} + \sum_{n > 0} m^{-n} \bigg)
=\frac{m+1}{(m-1)^3}  \Var[N].
\end{align*}

\subsection{Nerman's martingales}	\label{subsec:Nerman's martingales}

Suppose that $\xi$ is non-lattice and satisfies (A\ref{ass:Malthusian parameter}) through (A\ref{ass:second moment}),
and let $\lambda = \theta + \imag\eta$ be a root to $\L \mu(z)=1$
with $0 \leq \Real(\lambda)=\theta<\frac\alpha2$. Further, suppose that
\begin{align}
	\label{eq:L^2 assumption beyond the criticality}
\E\bigg[\bigg(\sum_{k=1}^N e^{-\theta X_k} \bigg)^{\!\! 2} \bigg]<\infty.
\end{align}
For simplicity let $Z_1(\lambda) \defeq \sum_{k=1}^N e^{-\lambda X_k}$.
We can view the complex variable $Z_1(\lambda)$ as a random variable taking values
in $\R^2$.
We denote by $\Sigma^\lambda$ the corresponding covariance matrix.
The aforementioned condition guarantees that $\Sigma^\lambda$ is well-defined.

Let $(W_t(\lambda) )_{t \geq 0}$ be defined by \eqref{eq:W_t^(j)(lambda)} for $j=0$.
But, since $\Real(\lambda) < \frac\alpha2$,
we cannot apply Theorem \ref{Thm:martingale convergence},
though one can still wonder what the long-term behavior of the process is.
To analyze this, we shall apply a special case of our main result, Theorem \ref{Thm:main}.
Let
\begin{equation*}
\phi(t) \defeq e^{\lambda t} \sum_{j=1}^N \1_{[0,X_j)}(t) e^{-\lambda X_j},	\quad	 t \in \R.
\end{equation*}
 Then
\begin{align*}
|\phi(t)| \leq e^{ \theta t}\1_{[0,\infty)}(t)\bigg|\sum_{j=1}^N  e^{- \theta X_j}\bigg|,
\end{align*}
and, by \eqref{eq:L^2 assumption beyond the criticality}, we conclude that the functions
\begin{align*}
	t \mapsto  e^{-(\alpha-p \theta) t}\1_{[0,\infty)}(t) \E\bigg[\bigg|\sum_{j=1}^N  e^{- \theta X_j}\bigg|^p\bigg]\quad \text{with }p=1,2
\end{align*}
are directly Riemann integrable. By \cite[Remark 3.10.5]{Resnick:1992},
 characteristics $\phi,\Real(\phi)$ and $\Imag(\phi)$
fulfill (A\ref{ass:mean growth}), (A\ref{ass:variance growth}) and (A\ref{ass:local ui of phi^2}).  Remark \ref{Rem:Nerman's martingales} below
(applied to $\phi=\phi_{\lambda,1}$)
gives that
\begin{equation*}
\cZ^{\phi}_t = e^{\lambda t} W_t(\lambda)
\end{equation*}
and $\E[\cZ^{\phi}_t]=e^{\lambda t} \1_{[0,\infty)}(t)$.

By Theorem \ref{Thm:slowly growing mean}, we deduce
\begin{equation*}
e^{-\frac\alpha2 t} (\cZ^{\Real(\phi)}_t,\cZ^{\Imag(\phi)}_t) \distto \sqrt{\tfrac{W}{\beta}} \cN
\end{equation*}
or equivalently
\begin{equation*}
(\Real(e^{(\lambda - \frac\alpha2) t} W_t(\lambda)),\Imag(e^{(\lambda-\frac\alpha2) t} W_t(\lambda)) )
\distto \sqrt{\tfrac{W}{\beta}}\cN
\end{equation*}
where $\cN$ is a $2$-dimensional centered Gaussian vector with covariance matrix $\Sigma$,
which can be explicitly computed. Indeed, we have
\begin{align*}
\Sigma_{11}
&\defeq
\int \Var\Big[\Real(\phi_\lambda(x)) + \sum_{j=1}^N \Real(e^{\lambda(x-X_j)})\1_{[0,\infty)}(x-X_j)\Big] e^{-\alpha x} \, \ell(\dx)
\\
&= \int \limits_0^\infty \Var\big[\Real(e^{\lambda x}Z_1(\lambda))\big] e^{-\alpha x} \, \dx	\\
&= \int \limits_0^\infty e^{2 \theta x} \Big(\cos^2(\eta x)\Sigma_{11}^{\lambda}-2\sin(\eta x)\cos(\eta x)\Sigma_{12}^{\lambda}  +\sin^2(\eta x)\Sigma_{22}^{\lambda}\Big)e^{-\alpha x} \, \dx\\
=&\frac{\Sigma_{11}^{\lambda}}{\alpha-2\theta}\frac{2\eta^2+(\alpha-2\theta)^2}{4\eta^2+(\alpha-2\theta)^2}
-{\Sigma_{12}^{\lambda}}\frac{2\eta}{4\eta^2+(\alpha-2\theta)^2}
+\frac{\Sigma_{22}^{\lambda}}{\alpha-2\theta}\frac{2\eta^2}{4\eta^2+(\alpha-2\theta)^2}.
\end{align*}
Similarly,
\begin{align*}
\Sigma_{22}
&=	\int \limits_0^\infty \Var\big[\Imag(e^{\lambda x}Z_1(\lambda))\big] e^{-\alpha x} \, \dx \\
&=	\frac{\Sigma_{22}^{\lambda}}{\alpha-2\theta}\frac{2\eta^2+(\alpha-2\theta)^2}{4\eta^2+(\alpha-2\theta)^2}
	+{\Sigma_{12}^{\lambda}}\frac{2\eta}{4\eta^2+(\alpha-2\theta)^2}
	+\frac{\Sigma_{11}^{\lambda}}{\alpha-2\theta}\frac{2\eta^2}{4\eta^2+(\alpha-2\theta)^2},
\end{align*}
and
\begin{align*}
\Sigma_{12}
&=	\int \limits_0^\infty \Cov\big[\Real(e^{\lambda x}Z_1(\lambda)),\Imag(e^{\lambda x}Z_1(\lambda))\big] e^{-\alpha x} \, \dx\\
&=	\frac{(\Sigma_{11}^\lambda-\Sigma_{22}^\lambda)\eta}{4\eta^2+(\alpha-2\theta)^2}+{\Sigma_{12}^{\lambda}}\frac{\alpha-2\theta}{4\eta^2+(\alpha-2\theta)^2}.
\end{align*}

\subsection{Epidemic models}	\label{sec:epidemic models}

In this section, 
we consider the epidemic model discussed in \cite{Britton+Scalia_Tomba:2019}.
In this model, the role of the ancestor is that of the first person in a community
infected by an infectious disease. Birth events become infection events etc.

Suppose that $\xi$ is a Poisson point process on $[0,\infty)$
with intensity measure $R_0 g(x) \, \dx$ where
\begin{equation*}
g(x) = \1_{(0,\infty)}(x) \frac{b^a x^{a-1}}{\Gamma(a)} e^{-b x},	\quad	x \in \R
\end{equation*}
is the density of the Gamma distribution with parameters $a,b>0$ and $R_0 > 1$ is the basic reproduction mean.
(No additional difficulties would occur if $R_0$ was replaced by a positive random variable $N$ with mean $R_0$
and finite variance.)
The function $g$ is the infection rate scaled to become a probability density.
It models the time delay between the infection of a person and a random person infected by that person.
Characteristics of interest are $I(t) = R_0 g(t)$ and $f(t) = \1_{[0,\infty)}(t)$
with $\cZ_t^I$ being the incidence at time $t$ and $\cZ_t^f$ the number of infections up to time $t$.
In the given situation,
the Laplace transform $\L\mu$ can be calculated explicitly in terms of $a,b$ and $R_0$, namely,
\begin{align*}
\L\mu(\lambda)
&= \int_0^\infty e^{-\lambda x} \, \mu(\dx)
= R_0 \Big(\frac{b}{b+\lambda}\Big)^{\! a},	\quad	\Real(\lambda) > - b.
\end{align*}
Hence the equation $\L\mu(\lambda) = 1$ takes the form 
\begin{equation*}
R_0 \Big(\frac{b}{b+\lambda}\Big)^a = 1.
\end{equation*}
Write $\frac{b}{b+\lambda} = r e^{\imag \varphi}$ with $r > 0$ and $|\varphi| < \pi/2$.
Then $R_0 (r e^{\imag \varphi})^a = 1$ is equivalent to
\begin{equation*}
e^{-\imag a \varphi} = r^a R_0.
\end{equation*}
This implies $r= R_0^{-1/a}$ and $\varphi \in (2\pi/a) \Z \cap (-\pi/2, \pi/2)$.
Solving for $\lambda$ yields
\begin{equation}	\label{eq:roots epidemic model}
\lambda = b(R_0^{1/a} e^{-\imag \varphi}-1)
\end{equation}
with $\varphi \in (2\pi/a) \Z \cap (-\pi/2,\pi/2)$, cf. Figure \ref{fig:solutions in epidemic model}.
The Malthusian parameter is obtained by setting $\varphi=0$, i.e.,
\begin{equation*}
\alpha = b(R_0^{1/a}-1).
\end{equation*}
The real part of a root $\lambda$ as in \eqref{eq:roots epidemic model} is given by
\begin{equation}	\label{eq:real part of roots epidemic model}
\Real(\lambda) = b(R_0^{1/a} \cos \varphi -1).
\end{equation}
A second root exists only if $a>4$ 
(otherwise $(2\pi/a) \Z \cap (-\pi/2,\pi/2) = \{0\}$),
in which case the root $\lambda \not = \alpha$ with largest real part is $\lambda = b(R_0^{1/a} e^{\imag 2 \pi / a}-1)$
with
\begin{equation*}
\Real(\lambda) = b(R_0^{1/a} \cos(\tfrac{2\pi}{a}) -1).
\end{equation*}
Further
\begin{equation*}
\Real(\lambda) = b(R_0^{1/a} \cos(\tfrac{2\pi}{a}) -1) \geq \frac\alpha2 = \frac{b}2(R_0^{1/a}-1)
\end{equation*}
if and only if $a > 6$ and $R_0 \geq R_0(a) \defeq (2 \cos(\frac{2\pi}{a}) - 1)^{-a}$.
Notice that $R_0(a) \to \infty$ for $a \downarrow 6$ and $R_0(a) \to 1$ for $a \to \infty$.
If $R_0 < R_0(a)$, Theorem \ref{Thm:no roots} applies and yields Gaussian fluctuations of $\cZ_t^I$ and $\cZ_t^f$.
That is, for $\cZ_t^f$ we have
\begin{align*}
e^{-\frac\alpha2 t} \big(\cZ_t^f - a_\alpha e^{\alpha t} W \big)
\distto \sigma \sqrt{\tfrac W\beta} \cN
\quad	\text{as } t \to \infty,
\end{align*}
with $\beta \defeq R_0 ab^a(b+\alpha)^{-a-1}$ and 
$a_\alpha \defeq (\alpha\beta)^{-1}$. 
Left with calculating $\sigma$ we obtain with the help of Remark~\ref{Rem:sigma}
\begin{align*}
\sigma^2&=\frac{1}{2\pi} \int \limits_{\Real(z)=\frac\alpha2}\Var\Big[\L f(z)+\L\xi(z)\frac{\L f(z)}{1-\L \mu(z)}\Big] \, |\dd z|\\
&=\frac{1}{2\pi} \int \limits_{\Real(z)=\frac\alpha2}\Big|\frac{1}{z(1-\L \mu(z))}\Big|^2\Var[\L\xi(z)] \, |\dd z|\\
&=\frac{1}{2\pi} \int \limits_{\Real(z)=\frac\alpha2}\Big|\frac{1}{z(1-\L \mu(z))}\Big|^2\L\mu(\alpha) \, |\dd z|\\
&=\frac{1}{2\pi}  \int_{-\infty}^\infty \frac{4}{(\alpha ^2+4t^2)g(t)}\, \dd t,
\end{align*}
with $g(t)\defeq 1+R_0^2\Big(\frac{4b^{2}}{(2b+\alpha)^2+4t^2}\Big)^{a}
-2R_0\Big(\frac{4b^{2}}{(2b+\alpha)^2+4t^2}\Big)^{\frac a2}\cos\Big(a\arctan\Big(\frac{2t}{2b+\alpha}\Big)\Big)$.

If $R_0 \geq R_0(a)$, the more general Theorem \ref{Thm:main} applies and gives additional periodic fluctuations
of greater magnitude than the Gaussian fluctuations.
We refrain from providing further details.


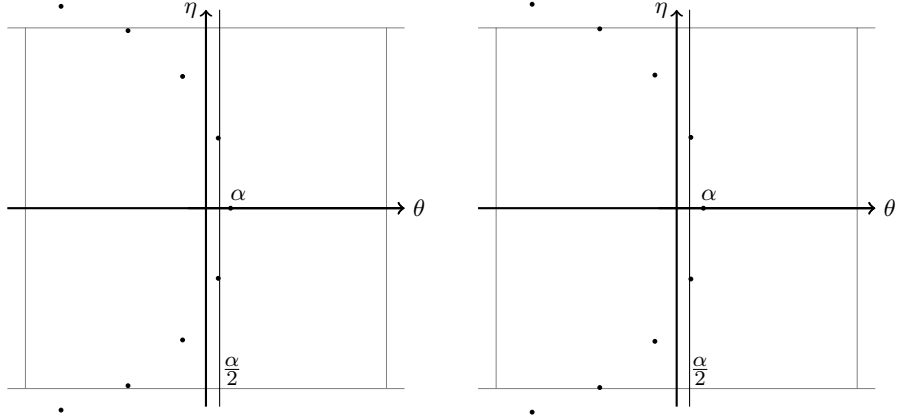
\begin{figure}[h]
\begin{subfigure}[t]{0.3\textwidth}
\begin{center}
\begin{tikzpicture}[scale=2.4]
		\draw [help lines] (-1.1,-1) grid (1.1,1);
		\draw [thick,->] (0,-1.1) -- (0,1.1);
		\draw [thick,->] (-1.1,0) -- (1.1,0);
		\draw (1.1,0) node[right]{$\theta$};
		\draw (0,1.1) node[left]{$\eta$};
		\filldraw (0.1364637,0) circle (0.3pt);
		\draw (0.18,0) node[above]{$\alpha$};
		\filldraw (0.06792652,0.3886935) circle (0.3pt);
		\filldraw (0.06792652,-0.3886935) circle (0.3pt);
		\filldraw (-0.1294183,0.7305048) circle (0.3pt);
		\filldraw (-0.1294183,-0.7305048) circle (0.3pt);
		\filldraw (-0.4317682,0.9842064) circle (0.3pt);
		\filldraw (-0.4317682,-0.9842064) circle (0.3pt);
		\filldraw (-0.8026552,1.119198) circle (0.3pt);
		\filldraw (-0.8026552,-1.119198) circle (0.3pt);
		\draw [thin,-,color=black] (0.1364637/2+0.0075,-1.1)--(0.1364637/2+0.0075,1.1);
		\draw (0.1364637/2-0.03,-0.9) node[right,color=black]{$\tfrac\alpha2$};
		\draw [thick,->] (-0.1,0) -- (1.1,0);
\end{tikzpicture}
\end{center}
\end{subfigure}
\qquad\qquad\qquad
\begin{subfigure}[t]{0.3\textwidth}
\begin{center}
\begin{tikzpicture}[scale=2.4]
		\draw [help lines] (-1.1,-1) grid (1.1,1);
		\draw [thick,->] (0,-1.1) -- (0,1.1);
		\draw [thick,->] (-1.1,0) -- (1.1,0);
		\draw (1.1,0) node[right]{$\theta$};
		\draw (0,1.1) node[left]{$\eta$};
		\filldraw (0.1480334,0) circle (0.3pt);
		\draw (0.18,0) node[above]{$\alpha$};
		\filldraw (0.07879849,0.3926505) circle (0.3pt);
		\filldraw (0.07879849,-0.3926505) circle (0.3pt);
		\filldraw (-0.1205554,0.7379416) circle (0.3pt);
		\filldraw (-0.1205554,-0.7379416) circle (0.3pt);
		\filldraw (-0.4259833,0.9942261) circle (0.3pt);
		\filldraw (-0.4259833,-0.9942261) circle (0.3pt);
		\filldraw (-0.8006461,1.130592) circle (0.3pt);
		\filldraw (-0.8006461,-1.130592) circle (0.3pt);
		\draw [thin,-,color=black] (0.1480334/2-0.0025,-1.1)--(0.1480334/2-0.0025,1.1);
		\draw (0.1480334/2-0.05,-0.9) node[right,color=black]{$\tfrac\alpha2$};
		\draw [thick,->] (-0.1,0) -- (1.1,0);
\end{tikzpicture}
\end{center}
\end{subfigure}
\caption{Solutions to $\L\mu(\lambda)=1$ in the cases $a=18$, $b=1$, $R_0=10$ (left figure)
and $a=18$, $b=1$, $R_0=12$. In the left figure, the 
root $\lambda \not = \alpha$ with largest real part has $\Real(\lambda) < \frac\alpha2$,
in the right figure $\Real(\lambda) > \frac\alpha2$.}
\label{fig:solutions in epidemic model}
\end{figure}

\subsection{Supercritical binary homogeneous Crump-Mode-Jagers processes}	\label{sec:Poisson}
In this section, we assume that
\begin{equation*}
\xi \defeq \sum_{j \geq 1} \1_{[0,\zeta]}(P_j)\delta_{P_j}
\end{equation*}
where
$(P_j)_{j \geq 1}$ are the arrival times of a  homogeneous Poisson process with intensity $b>0$,
independent of  the $[0,\infty]$-valued random variable $\zeta$.
We are interested in $\cZ^{\1_{[0,\zeta)}}_t$ the number of individuals alive at time $t$,
see \eqref{eq:Z_t^1_[0,zeta]}. Thus, the corresponding characteristic $\phi$ is given by $\phi(t)\defeq\1_{[0,\zeta)}(t)$ for $t \geq 0$.

We put $\L_\zeta(z)\defeq\E[e^{-z\zeta}]$ for $\Real(z)\geq 0$ and start by noting that
\begin{align*}
\L\mu(z)&=\E\bigg[\sum_{j \geq 1}\1_{[0,\zeta]}(P_j) e^{-z P_j}\bigg]
= \E\bigg[\int\1_{[0,\zeta]}(x)e^{-z x} b \dx\bigg]
= \int \limits_0^\infty e^{-z x}b \Prob(\zeta \geq x) \, \dx	\\
&=\frac{b(1-\L_\zeta(z))}{z},\quad \Real(z)>0.
\end{align*}
The Malthusian parameter $\alpha$ is the unique real number that satisfies
\begin{align*}
1-\L_\zeta(\alpha)=\frac\alpha b,
\end{align*}
and the parameter $\beta$ is given by
\begin{equation*}
\beta=\frac{1}{\alpha}(1+b\L_\zeta'(\alpha))=\frac{1}{\alpha}(1-b\E[\zeta e^{-\alpha\zeta}]).
\end{equation*}
Now we shall show
that $\alpha$ is the only solution to $\L\mu(z)=1$ with $\Real(z)>0$.
Indeed, for positive $\theta$ and $\eta$
\begin{align*}
-\Imag&\big(\L\mu(\theta+\imag\eta)\big)=\int_0^\infty \sin(\eta x) e^{-\theta x} b \Prob(\zeta \geq x)\dx\\
&=b\sum_{l\geq 0}\int\limits_{2\pi l/\eta}^{\pi (2l+1)/\eta}\sin(\eta x)\Big(e^{-\theta x} \Prob(\zeta \geq x)-e^{-\theta (x+\pi/\eta)} \Prob(\zeta \geq x+\pi/\eta)\Big)\dx>0
\end{align*}
which, together with $\L\mu(\overline{\lambda})=\overline{\L\mu(\lambda)}$, shows that
$\Lambda_{\ge}\cap \{z:\Real(z)>0\}=\{\alpha\}$. Since $\L(\E[\phi])(\alpha)=b^{-1}\L\mu(\alpha)=b^{-1}$,
an application of Theorem~\ref{Thm:no roots} yields
\begin{align*}
e^{-\frac{\alpha}{2}t}\Big(\cZ^{\1_{[0,\zeta)}}_t-e^{\alpha t}\frac{W}{b\beta}\Big)\distto \sigma \sqrt{\tfrac{W}{\beta}}\cN.
\end{align*}
Next,	
we express the variance $\sigma^2$ in terms of the parameters $b$, $\alpha$ and $\beta$.
By Remark~\ref{Rem:sigma}
\begin{align*}
\sigma^2=\frac{1}{2\pi}\int \limits_{\Real(z)=\frac\alpha2}\Var\Big[\L\phi(z)+\L\xi(z)\frac{\L (\E[\phi])(z)}{1-\L \mu(z)}\Big] \, |\dd z|.
\end{align*}
For $\Real(z)>0$
\begin{align*}
\Var\Big[\L\phi(z)+\L\xi(z)\frac{\L (\E[\phi])(z)}{1-\L \mu(z)}\Big]
&=\E\Big[\Var\Big[\L\phi(z)+\L\xi(z)\frac{\L \mu(z)}{b(1-\L \mu(z))}\Big|{\zeta}\Big]\Big]	\\
&\hphantom{=}+ \Var\Big[\E\Big[\L\phi(z)+\L\xi(z)\frac{\L \mu(z)}{b(1-\L \mu(z))}\Big|\zeta\Big]\Big]
\eqdef \mathit{I}+\mathit{I\!I}
\end{align*}
with
\begin{align*}
\mathit{I}
&=\E\Big[\Var\Big[\L\xi(z)\frac{\L \mu(z)}{b(1-\L \mu(z))}\Big|\zeta\Big]\Big]	\\
&=\E\big[\Var\big[\L\xi(z)\big|\zeta\big]\big]\Big|\frac{\L \mu(z)}{b(1-\L \mu(z))}\Big|^2
=\E\Big[\int_{0}^\zeta e^{-2\Real(z)x}b\, \dd x\Big]\Big|\frac{\L \mu(z)}{b(1-\L \mu(z))}\Big|^2,
\end{align*}
 where we have used properties of the Poisson process, and
\begin{align*}
\mathit{I\!I}
&=\Var\Big[\frac{1-e^{-z\zeta}}{z}+\frac{b(1-e^{-z\zeta})}{z}\frac{\L \mu(z)}{b(1-\L \mu(z))}\Big]
=\Big|\frac{1}{z(1-\L\mu(z))}\Big|^2\Var[e^{-z\zeta}].
\end{align*}
Assuming now that $\Real(z)=\frac \alpha 2$ we arrive at
\begin{align*}
\mathit{I} = \Big|\frac{\L \mu(z)}{b(1-\L \mu(z))}\Big|^2
\end{align*}
because
\begin{align*}
\E\Big[\int_{0}^\zeta e^{-2\Real(z)x}b\, \dd x\Big]
=\E\Big[\int_{0}^\zeta e^{-\alpha x}b\, \dd x\Big]=1,
\end{align*}
and
\begin{equation*}
\mathit{I\!I}=\Big|\frac{1}{z(1-\L\mu(z))}\Big|^2\big(\L_\zeta(\alpha)-|\L_\zeta(z)|^2\big).
\end{equation*}
Observing that $\overline{z}=\alpha-z$ and $\overline{\L \mu(z)}=\L\mu(\alpha-z)$ whenever
$\Real(z)=\frac \alpha 2$, we further infer
\begin{align*}
\Big|\frac{z\L \mu(z)}{b}\Big|^2+\L_\zeta(\alpha)-|\L_\zeta(z)|^2
&=
|1-\L_\zeta(z)|^2+1-\frac{\alpha}{b}-|\L_\zeta(z)|^2	\\
&=1-\L_\zeta(z) + \overline{1-\L_\zeta(z)}-\tfrac{\alpha}{b}	\\
&=\tfrac1b \big(z(\L\mu(z)-1)+(\alpha-z)(\L\mu(\alpha-z)-1)\big),
\end{align*}
which in turn gives
\begin{align*}
\mathit{I}+\mathit{I\!I}
&=\Big|\frac{1}{z(1-\L \mu(z))}\Big|^2\Big(\Big|\frac{z\L \mu(z)}{b}\Big|^2+\L_\zeta(\alpha)-|\L_\zeta(z)|^2\Big)	\\
&=\frac{z(\L\mu(z)-1)+(\alpha-z)(\L\mu(\alpha-z)-1)}{bz(\L\mu(z)-1)(\alpha-z)(\L\mu(\alpha-z)-1)}	\\
&=\frac{1}{b(\alpha-z)(\L\mu(\alpha-z)-1)}+\frac{1}{bz(\L\mu(z)-1)}	\\
&=\frac2b \Real\Big(\frac{1}{z(\L\mu(z)-1)}\Big).
\end{align*}
We can now compute the variance as follows
\begin{align*}
\sigma^2
&=\lim_{R\to \infty}\frac{1}{b\pi}\Real\bigg(\int\limits_{\frac\alpha 2-\imag R}^{\frac\alpha 2+\imag R}
\frac{1}{z(\L\mu(z)-1)}\, |\dd z|\bigg)
=\frac{1}{b\pi} \lim_{R\to \infty}\Imag\bigg(\int\limits_{\frac\alpha 2-\imag R}^{\frac\alpha 2+\imag R}
\frac{1}{z(\L\mu(z)-1)}\, \dd z\bigg)\\
&=\frac{1}{b\pi} \lim_{R\to \infty}\Imag\bigg(\int\limits_{\frac\alpha 2-\imag R}^{\frac\alpha 2+\imag R}
	\frac{1}{z(\L\mu(z)-1)}+\frac{1}{z}\, \dd z\bigg)-\lim_{R\to \infty}\frac 1{b\pi}\int\limits_{-R}^R \frac{\frac{\alpha}{2}}{(\tfrac{\alpha}{2})^2+t^2}\, \dt\\
&=\frac{1}{b\pi}\Imag\bigg( \lim_{R\to \infty}\int\limits_{\frac\alpha 2-\imag R}^{\frac\alpha 2+\imag R}
\frac{\L\mu(z)}{z(\L\mu(z)-1)}\, \dd z\bigg)-\frac 1 b.
\end{align*}
To calculate the limit, we use the residue theorem. For $R>\alpha$,
\begin{align*}
\int\limits_{-\frac\pi 2}^{\frac\pi 2}\frac{\imag Re^{\imag \theta}\L\mu(\frac\alpha 2+Re^{\imag \theta})\dd\theta}{(\frac\alpha 2+Re^{\imag \theta})
(\L\mu(\frac\alpha 2+Re^{\imag \theta})-1)}-\int\limits_{\frac\alpha 2-\imag R}^{\frac\alpha 2+\imag R} \frac{\L\mu(z)\dd z}{z(\L\mu(z)-1)}
=2\pi\imag \, \underset{z=\alpha}{\Res} \frac{\L\mu(z)}{z(\L\mu(z)-1)}.
\end{align*}
It suffices to show that the integrand of the first integral decays to zero uniformly in $\theta$
as $R$ goes to infinity. In view of the inequality $|\L \mu(z)| \leq 2b|z|^{-1}$ and its consequence $|\L \mu(z)-1| \geq 1- 2b|z|^{-1}$
(both hold true for $\Real(z) \geq 0$) we conclude that
\begin{align*}
	\bigg|\frac{ R \L\mu(\frac\alpha 2+Re^{\imag \theta}) }{(\frac\alpha 2+Re^{\imag \theta})
		(\L\mu(\frac\alpha 2+Re^{\imag \theta})-1)}\bigg|\le \frac{R}{|\frac\alpha 2+Re^{\imag \theta}|}\cdot \frac{2b}{|\frac\alpha 2+Re^{\imag \theta}|-2b}=O(R^{-1})
\end{align*}
as $R\to\infty$ uniformly in $\theta  \in [-\frac\pi2,\frac\pi2]$.
Finally,
\begin{align*}
\sigma^2 = \frac{2-\alpha \beta}{\alpha b\beta }
\end{align*}
and thereupon 
\begin{align*}
e^{-\frac{\alpha}{2}t}\Big(\cZ^{\1_{[0,\zeta)}}_t-e^{\alpha t}\frac{W}{b\beta}\Big)\stablyto\sqrt{\frac{(2-\alpha\beta)W}{\alpha b\beta^2}}\cN.
\end{align*}
An application of a similar argument as in Remark \ref{Rem:application of stable convergence} enables us to conclude that the convergence mentioned above holds true conditionally given $\Surv$.
The distribution of $W$ conditionally given $\Surv$ is exponential with parameter $\alpha/b$.
Therefore, we have just reproved Henry's central limit theorem \cite{Henry:2017}.


\subsection{The conservative fragmentation model}	\label{sec:conservative fragmentation}

In this section we consider the conservative fragmentation model
as discussed in \cite{Janson+Neininger:2008}.
Let $b \geq 2$ be integer and $(V_1,V_2,\ldots,V_b)$ a vector of nonnegative random variables such that
$\sum_{j=1}^b V_j = 1$ a.\,s.
 For the sake of simplicity, we assume that $V_1,\ldots,V_b$ have Lebesgue densities except possible atoms at $0$.
(Our theory would allow to cover more general cases, too.)
Starting with an object of mass $x \geq 1$, we break it into pieces with masses $(V_1x,V_2x,\ldots,V_bx)$.
Continue recursively with each piece of mass $\geq 1$,
using new and independent copies of the random vector $(V_1,V_2,\ldots,V_b)$ each time.
 Once a fragment has mass $<1$, it is not further crumbled.
The process terminates a.\,s.\ after a finite number of steps, leaving a finite set of fragments of masses $<1$.

Denote by $n(x)$ the random number of fragmentation events,
i.e., the number of pieces of mass $\geq 1$ that appear during the process.
Further, let $n_e(x)$ be the final number of fragments, i.e., the number of pieces of mass $<1$ that appear.
A limit theorem for $n(x)$ has been proved in \cite{Janson+Neininger:2008},
where it was shown that the asymptotic behavior of $n(x)$ as $x$ goes to infinity
depends on the position of the roots of the function $z \mapsto \sum_{j \geq 1} \E[V_j^z]$.

Letting $\xi \defeq  \sum_{j=1}^b \1_{\{V_j > 0\}} \delta_{-\log V_j}$, we conclude that
the corresponding Malthusian parameter is 1, i.e., $\alpha=1$ and the limit of Nerman's martingale satisfies $W=1$ a.\,s. Further $\beta=\sum_{j=1}^b \mathbb{E}[V_j|\log V_j|]\in (0,\infty).$
Note also that $n(x)=N(\log x)$ corresponds to the number of individuals born up to and including time $\log x$
and similarly, we can represent $n_e(x)$ as a general branching process,
namely, $n_e(x)=\cZ^\varphi_{\log x}$, with
\begin{equation*}
\varphi(t) \defeq \sum_{j=1}^b \1_{\{V_j > 0\}} \1_{[0,-\log V_j)}(t)	\quad	\text{for } t \in \R.
\end{equation*}
Hence, our main result provides (precise) limit theorems for both $n$ and $n_e$.
For instance, in the case when all root from $\Lambda$ are simple, we  infer from Theorem \ref{Thm:main non-lattice}
(the constants $b_{\lambda,0}$ in the theorem can easily be seen to equal
$b_{\lambda,0} = -1/(\L\mu)'(\lambda)$, $\lambda \in \Lambda$
by Proposition \ref{Prop:determining the coefficients})
\begin{align*}
x^{-1/2}&\Big(n(x)+\sum_{\lambda\in\Lambda} \tfrac{W(\lambda)}{\lambda(\L\mu)'(\lambda)} x^\lambda\Big)
\distto \frac{\sigma}{\sqrt{\beta}}
\cN
\quad\text{if }\partial\Lambda \text{ is empty  and}	\\	
x^{-1/2}\big(\log x\big)^{-k+1/2}&\Big(n(x)+\sum_{\lambda\in\Lambda} \tfrac{W(\lambda)}{\lambda(\L\mu)'(\lambda)} x^\lambda\Big)
\distto \frac{\rho_{k-1}}{\sqrt{(2k-1)\beta}} 
\cN
\quad\text{if }\partial\Lambda \text{ is non-empty},
\end{align*}
where $k$ is the largest multiplicity  of a root on the critical line $\Real(z)=\frac\alpha2$, and $\rho_{k-1}$ is as in Theorem \ref{Thm:main}.

\section{Preliminaries for the proofs of the main results}	\label{sec:preliminaries}

In this section we gather facts from the literature, introduce some notation used throughout the paper
and perform some basic calculations.

\subsection{Change of measure and the connection to renewal theory}	\label{subsec:renewal theory}

 The existence of the Malthusian parameter (i.e., \eqref{eq:Malthusian alpha}) enables us
to use a change-of-measure argument as follows.
We define a random walk
$(S_n)_{n \in \N_0}$  with $S_0=0$ on some probability space
with underlying probability measure $\Probsf$ and increment distribution given by
\begin{align}	\label{eq:associated RW increment}
\Probsf(S_1 \in B) = \E\bigg[ \sum_{|u|=1} e^{-\alpha S(u)} \1_B(S(u))\bigg]
 = \int \limits_B e^{-\alpha x} \, \mu(\dx),	\qquad	B \in \B(\R).
\end{align}
With this definition, the many-to-one formula (see, e.g., \cite[Theorem 1.1]{Shi:2015}) holds:
\begin{align}	\label{eq:many-to-one1}
\Esf[f(S_1,\ldots,S_n)] &= \E\bigg[\sum_{|u|=n} e^{-\alpha S(u)} f(S(u|_1),\ldots,S(u)) \bigg]
\end{align}
for all Borel measurable $f:\R^n \to \R$ such that the expectation on the left- or right-hand side of \eqref{eq:many-to-one1}
is well-defined, possibly infinite.
In particular,  under \eqref{eq:beta}
\begin{align}	\label{eq:E[S_1]}
\Esf[S_1] &= \beta \in (0,\infty).
\end{align}
In other words, the increments of the random walk
$(S_n)_{n \in \N_0}$
have positive, finite mean. As a consequence, the associated renewal measure
\begin{equation*}
\U(\cdot) = \sum_{n \in \N_0} \Probsf(S_n \in \cdot)
\end{equation*}
is uniformly locally finite in the sense that
\begin{equation}	\label{eq:U unif loc finite}
\U([t,t+h]) \leq \U([0,h]) < \infty	\qquad	\text{for all } t,h \geq 0.
\end{equation}
Indeed, if $\tau \defeq \inf\{n \in \N_0: S_n \geq t\}$, then
\begin{equation*}
\sum_{n \geq 0} \1_{[t,t+h]}(S_n) = \sum_{n \geq 0} \1_{[t,t+h]}(S_{\tau+n}) \leq \sum_{n \geq 0} \1_{[0,h]}(S_{\tau+n}-S_\tau).
\end{equation*}
Now take expectations and use the strong Markov property at $\tau$ to infer \eqref{eq:U unif loc finite}.

By the many-to-one formula, (A\ref{ass:first moment}) implies
that the increments of the associated random walk $(S_n)_{n \in \N_0}$
have a finite exponential moment of order $\alpha-\vartheta > \alpha/2$
since
\begin{align}	
\Esf[e^{(\alpha-\vartheta) S_1}] &=  \E\bigg[\sum_{j=1}^N e^{-\vartheta X_j} \bigg] = \L\mu(\vartheta) < \infty.
\end{align}

\subsection{The expectation of the general branching process}	\label{subsec:expectation}

There is a connection between the renewal measure $\U$
and the expectation $m_t^\varphi = \E[\cZ_t^\varphi]$ of the general branching process
counted with characteristic $\varphi$ provided that $\varphi$ satisfies suitable assumptions.
For instance, if  $\varphi$ is nonnegative and
$t \mapsto \E[\varphi(t)] e^{-\alpha t}$ is a directly Riemann integrable function,
then we infer from the many-to-one formula
\begin{align}
m_t^\varphi e^{-\alpha t}
&\defeq \E[\cZ_t^\varphi] e^{-\alpha t}
= \sum_{n=0}^\infty \E\bigg[\sum_{|u|=n} e^{-\alpha S(u)} \varphi_u(t-S(u)) e^{-\alpha(t-S(u))} \bigg]		\notag	\\
&= \sum_{n=0}^\infty \E\bigg[\sum_{|u|=n} e^{-\alpha S(u)} \E[\varphi](t-S(u)) e^{-\alpha(t-S(u))}\bigg]		\notag	\\
&= \sum_{n=0}^\infty \Esf\big[\E[\varphi](t-S_n) e^{-\alpha(t-S_n)}\big]							\notag	\\
&= \int \E[\varphi](t-x) e^{-\alpha(t-x)} \, \U(\dx).	\label{eq:m_t^phi in terms of U}
\end{align}
By the direct Riemann integrability of $t \mapsto \E[\varphi(t)]e^{-\alpha t}$
and \eqref{eq:U unif loc finite}, the function $t \mapsto m_t^\varphi e^{-\alpha t}$ is
bounded and, moreover,
\begin{equation}
\lim_{\substack{t \to \infty\\ t \in \G}} e^{-\alpha t} m_t^\varphi
= \frac1\beta \int \E[\varphi](x) e^{-\alpha x} \, \ell(\dx)
 = \frac1\beta (\L\E[\varphi])(\alpha)
\end{equation}
by the key renewal theorem, see \cite[Theorem 4.2]{Athreya+McDonald+Ney:1978} in the non-lattice case
and \cite[Theorem 2.5.3]{Alsmeyer:1991} in the lattice (and non-lattice) case.
Recall that, in the lattice case, $\L\E[\varphi]$ denotes the `discrete'  bilateral Laplace transform of $\E[\varphi]$,
$\G = \Z$ and $\ell$ is the counting measure on $\Z$,
whereas in the non-lattice case, $\L\E[\varphi]$ is the `continuous'  bilateral Laplace transform,
$\G = \R$ and $\ell$ is the Lebesgue measure on $\R$.

We need a lemma in preparation for the proof of Proposition \ref{Prop:L1 existence of Z_t^varphi}.

\begin{lemma}	\label{Lem:centered characteristic->martingale}
Suppose that (A\ref{ass:Malthusian parameter}) holds and that $\chi$
is a centered characteristic, i.e., $\E[\chi(t)]=0$ for all $t \in \R$.
Fix $t \in \R$ and suppose that
\begin{align}	\label{eq:Esum Varvarphi(t-S(u))<infty}
\E\bigg[\sum_{u \in \I} \Var[\chi](t-S(u))\bigg] < \infty.
\end{align}
Let $(u_n)_{n \in \N}$ be an admissible ordering of $\I$ (see the paragraph before Proposition \ref{Prop:L1 existence of Z_t^varphi} for the definition).
Define
\begin{equation*}
M_n(t) \defeq \sum_{j=1}^n \chi_{u_j}(t-S(u_j))
\end{equation*}
for $n \in \N_0$.
Then $(M_n(t))_{n \in \N_0}$ is a centered martingale and bounded in $L^2$.
In particular,
\begin{align}	\label{eq:definition of Z^chi}
Z^{\chi}_t \defeq \sum_{u\in \I} \chi_{u}(t-S(u))
\end{align}	
converges unconditionally in $L^2$ and it is also the almost sure limit of $M_n(t)$ as $n \to \infty$.
Further, for any (deterministic) sequence $(\I_n)_{n \in \N_0}$ with $\I_n \uparrow \I$,
\begin{equation*}
M_{\I_n}(t) = \sum_{u \in \I_n} \chi_{u}(t-S(u)) \to \cZ_t^\chi
\quad	\text{in $L^2$ as } n \to \infty.
\end{equation*}
Moreover,
\begin{align}	\label{eq:Var[Z_t^chi]}
\Var[\cZ_t^\chi] = \E[(\cZ_t^\chi)^2] = \E[\cZ_t^{\chi^2}] < \infty.
\end{align}
Finally, (A\ref{ass:variance growth}) is sufficient for \eqref{eq:Esum Varvarphi(t-S(u))<infty} to hold
for every $t \in \R$.
\end{lemma}

\begin{proof}
Let $\mathcal{G}_n = \sigma(\pi_{u_j}: j \leq n)$,
where it should be recalled that $\pi_u$ is the projection onto the life space of individual $u$,
in particular, $(\xi_u,\zeta_u,\chi_u)$ is $\pi_u$-measurable.
Then $(M_n(t))_{n \in \N_0}$ is adapted with respect to $(\mathcal{G}_{n})_{n \in \N_0}$ as,
for any $u \in \I_n$, both $S(u)$ and $\chi_u$ are $\mathcal{G}_{n}$-measurable.
Moreover, \eqref{eq:Esum Varvarphi(t-S(u))<infty} implies that for any $u \in \I$,
$\E[|\chi_u(t-S(u))|]<\infty$. Hence $M_n(t)$ is integrable for any $n \in \N_0$.
The martingale property then follows since $S(u_{n+1})$ is $\mathcal{G}_n$-measurable
whereas $\chi_{u_{n+1}}$ is independent of  $\mathcal{G}_n$ and since $\E[\chi](x) = \E[\chi(x)] = 0$ for all $x \in \R$, so
\begin{equation*}
\E[\chi_{u_{n+1}}(t-S(u_{n+1})) |  \mathcal{G}_n] = \E[\chi](t-S(u_{n+1})) = 0	\quad	\text{almost surely.}
\end{equation*}
Next, we observe that, since the increments of $L^2$-martingales are uncorrelated,
\begin{align}
\E[M_n(t)^2]
&=\E\bigg[\sum_{j=1}^n \chi_{u_j}^2(t-S(u_j)) \bigg]
= \E\bigg[\sum_{j=1}^n \E[\chi^2](t-S(u_j)) \bigg]	\notag	\\
&\leq \E\bigg[\sum_{u \in \I} \E[\chi^2](t-S(u)) \bigg]
= \E\bigg[\sum_{u \in \I} \Var[\chi](t-S(u)) \bigg].	\label{eq:L^2 boundedness of the martingale}
\end{align}
By \eqref{eq:Esum Varvarphi(t-S(u))<infty},
the martingale $(M_n(t))_{n \in \N_0}$ is bounded in $L^2$ and thus converges in $L^2$ and almost surely.
We denote the limit by $\cZ_t^\chi$ and view it as the limit of the series
on the right-hand side of \eqref{eq:definition of Z^chi}.
This is justified by the following argument.
For any subset $\cJ \subseteq \I$, finite or infinite, since for any $u \in \cJ$ there is a unique $j \in \N$
with $u=u_j$ and again since martingale increments are uncorrelated, we have
\begin{equation*}
\E \bigg[\bigg|\sum_{u \in \cJ} \chi_u(t-S(u))\bigg|^2 \bigg] = \E\bigg[\sum_{u \in \cJ} \Var[\chi](t-S(u))\bigg].
\end{equation*}
From this and the Cauchy criterion,
on the one hand, we infer the unconditional convergence in $L^2$ of the series in \eqref{eq:definition of Z^chi},
thereby justifying to write $Z_t^\chi$ for the limit.
On the other hand, we conclude the convergence of $M_{\I_n}(t)$ to $\cZ_t^\chi$.
Since convergence in $L^2$ implies convergence in $L^1$,
$\cZ_t^\chi$ is centered.
Using this and again the convergence in $L^2$, we deduce
\begin{align*}
\Var[\cZ_t^\chi]
&= \E[(\cZ_t^\chi)^2] = \lim_{n \to \infty} \E[M_n(t)^2]	\\
&= \lim_{n \to \infty} \E\bigg[\sum_{u \in \I_n} \chi_u^2(t-S(u)) \bigg]
= \E\bigg[\sum_{u \in \I} \chi_u^2(t-S(u)) \bigg],
\end{align*}
i.e., \eqref{eq:Var[Z_t^chi]} holds.
Finally, (A\ref{ass:variance growth}) implies \eqref{eq:Esum Varvarphi(t-S(u))<infty}
since, for any $t \in \R$, by \eqref{eq:m_t^phi in terms of U} and the subsequent arguments,
\begin{align*}
\E\bigg[\sum_{u \in \I} \Var[\chi](t-S(u))\bigg]
&= \E\bigg[\sum_{u \in \I} \E[\chi^2](t-S(u)) \bigg]	\\
&= e^{\alpha t} \int \E[\chi^2](t-x) e^{-\alpha(t-x)} \, \U(\dx) \leq C e^{\alpha t}
\end{align*}
where as before $\U$ is the renewal measure of the associated random walk $(S_n)_{n \in \N_0}$
and $C > 0$ is some finite constant.
\end{proof}

We are now ready to prove Proposition \ref{Prop:L1 existence of Z_t^varphi}.

\begin{proof}[Proof of Proposition \ref{Prop:L1 existence of Z_t^varphi}]
By (A\ref{ass:mean growth}), $\E[\varphi](t)$ is finite for every $t \in \R$ and we may write
\begin{align*}
\varphi_u(t-S(u)) = \E[\varphi](t-S(u))+\big(\varphi_u(t-S(u))-\E[\varphi](t-S(u))\big)
\end{align*}
for every $u \in \I$.
It is therefore enough to check that both series
\begin{align}	\label{eq:two series}
\sum_{u \in \I} \E[\varphi](t-S(u))
\quad	\text{and}	\quad
\sum_{u \in \I} \big(\varphi_u(t-S(u))-\E[\varphi](t-S(u))\big)
\end{align}
converge almost surely over admissible orderings and unconditionally in $L^1$.
For the first series, note that by (A\ref{ass:mean growth})
the function $f(t) \defeq |\E[\varphi](t)| e^{-\alpha t}$ is directly Riemann integrable as well
and by  \eqref{eq:many-to-one1}, we have
\begin{align*}
\E\bigg[\sum_{u \in \I} |\E[\varphi](t-S(u))|\bigg]=\Esf\bigg[\sum_{n \geq 0}e^{\alpha S_n}|\E[\varphi](t-S_n)|\bigg]
=e^{\alpha t}\cdot f* \U(t),
\end{align*}
which is finite by \eqref{eq:U unif loc finite} and the direct Riemann integrability of $f$.
Hence, the series converges unconditionally in $L^1$ and absolutely almost surely.
The same argument as above gives
\begin{align*}
\E\bigg[\sum_{u \in \I} \Var[\varphi](t-S(u))\bigg] < \infty,
\end{align*}
i.e., $\chi(t) \defeq \varphi(t)-\E[\varphi](t)$ is a centered characteristic satisfying \eqref{eq:Esum Varvarphi(t-S(u))<infty}.
We may thus apply Lemma \ref{Lem:centered characteristic->martingale}
to conclude that the second series in \eqref{eq:two series} converges almost surely over admissible orderings of $\I$  and unconditionally in $L^2$.
\end{proof}

 We close this subsection with the proof of Proposition \ref{Prop:f^*}.

\begin{proof}[Proof of Proposition \ref{Prop:f^*}]

(a) Since $f$ is c\`adl\`ag, it is locally bounded and continuous Lebesgue-almost everywhere.
By (a slightly extended version of) \cite[Remark 3.10.4 on p.~236]{Resnick:1992}, this together with
\begin{align*}
\sum_{n \in \Z} \sup_{x \in [n,n+1]} |f(x)| \leq \int f^*(x) \, \dx < \infty
\end{align*}
ensures the direct Riemann integrability of $f$.

Conversely, if $f$ is directly Riemann integrable,
then it is locally bounded and continuous Lebesgue-almost everywhere.
Local boundedness of $f$ entails that of $f^*$.
Since $f^*$ is continuous on $\{x \in \R: f \text{ is continuous at } x-1 \text{ and } x+1\}$,
this implies that also $f^*$ is continuous Lebesgue-almost everywhere.
 Furthermore, for every $x \in \R$, we have
\begin{align*}
\sum_{n \in \Z} & \sup_{n \leq x < n+1} f^*(x)
= \sum_{n \in \Z} \sup_{n-1 \leq x < n+2} |f(x)|
\leq 3 \sum_{n \in \Z} \sup_{n \leq x < n+1} |f(x)| < \infty
\end{align*}
since $f$ is directly Riemann integrable.
Thus, again by \cite[Remark 3.10.4 on p.~236]{Resnick:1992}, $f^*$ is directly Riemann integrable.
	\smallskip

\noindent

We prove (b) and (c) at one go. To this end, let $p=1$ in the situation of (b)
and $p=2$ in the situation (c). Define $\phi(t) \defeq \varphi(t)^p$ for $t \in \R$.
Then we infer
\begin{equation}	\label{eq:int Ephi*e^-alphax dx finite}
\int \E[\phi^*](x) e^{-\alpha x} \, \dx < \infty,
\end{equation}
from \eqref{eq:int Evarphi*e^-alphax dx finite} or \eqref{eq:int Evarphi*^2e^-alphax dx finite}, respectively,
where we have used that $(\varphi^2)^* = (\varphi^*)^2$ in the situation of (c).
From \eqref{eq:int Ephi*e^-alphax dx finite} we deduce that $\E[\phi^*](x)<\infty$ for Lebesgue-almost all $x \in \R$
and hence
\begin{equation}	\label{eq:local majorant}
\E\Big[\sup_{|t-x| \leq \frac12} |\phi(t)|\Big] < \infty	\quad	\text{for all } x \in \R.
\end{equation}
In the case of (c), this implies the validity of (A\ref{ass:local ui of phi^2}).
In both cases, \eqref{eq:local majorant}
together with the dominated convergence theorem
imply that $\E[\phi]$
has c\`adl\`ag paths and thus also $f$ defined by $f(t) \defeq \E[\phi(t)]e^{-\alpha t}$.
Further, $\int f^*(x) \, \dx < \infty$ by \eqref{eq:int Evarphi*e^-alphax dx finite}
and \eqref{eq:int Evarphi*^2e^-alphax dx finite}, respectively,
since
$(\E[\phi](t)e^{-\alpha t})^*
\leq e^{\alpha}\E[\phi]^*(t)e^{-\alpha t} \leq e^{\alpha} \E[\phi^*](t)e^{-\alpha t}$.
Part (b) now follows from (a).
In the situation of (c), we deduce from (a) that $t \mapsto \E[\varphi^2](t) e^{-\alpha t}$ is directly Riemann integrable.
Also, $\E[\varphi]$ has c\`adl\`ag paths by \eqref{eq:local majorant} and the dominated convergence theorem.
Therefore, $\Var[\varphi](t) = \E[\varphi^2](t) - \E[\varphi(t)]^2$ is c\`adl\`ag
and, in particular, locally bounded and continuous
Lebesgue-almost everywhere. Since
\begin{align*}
\sum_{n \in \Z} \sup_{x \in [n,n+1]}  \Var[\varphi(x)] e^{-\alpha x}
&\leq
\sum_{n \in \Z} \sup_{x \in [n,n+1]} \E[\varphi(x)^2] e^{-\alpha x}
< \infty
\end{align*}
the direct Riemann integrability of $\Var[\varphi](t) e^{-\alpha t}$ follows from \cite[Remark 3.10.4]{Resnick:1992},
i.e., (A\ref{ass:variance growth}) holds.

\end{proof}

\subsection{Matrix notation}	\label{subsec:Matrix notation}
For any $s\in \R$ and $\gamma\in\C$ we define the following lower triangular $k \times k$ matrix
\begin{align}
\exp(\gamma,s,k)
\defeq e^{\gamma s}\times
\begin{pmatrix}
1&		0&		0&		\dots\qquad&	0\\
s&		1&		0&		\dots\qquad&	0\\
s^2&		2s&		1&		\dots\qquad&	0\\
\vdots&	\vdots&	\vdots&	\ddots\qquad&	\vdots\\
s^{k-1}\quad &	{k-1\choose 1} s^{k-2}\quad &	{k-1 \choose 2}s^{k-3}\quad &		\dots\qquad	&	1
\end{pmatrix}.
\end{align}
The $(i,j)^{\mathrm{th}}$ entry of the matrix is $e^{\gamma s} \binom{i-1}{j-1} s^{i-j}$, $i,j=1,\ldots,k$,
 where $\binom{i-1}{j-1} = 0$ for $j > i$ should be recalled.
Matrices of this form will be very useful since they simplify the notation
and allow us to deal with polynomial terms with relative ease.
 Indeed, for any $s,t \in \R$ and $\gamma\in\C$,
\begin{equation*}
\exp(\gamma, s,k) \cdot \exp(\gamma,t,k)=\exp(\gamma,s+t,k).
\end{equation*}
 This can be seen from elementary but tedious calculations.
 Alternatively, notice that
\begin{align*}
	\exp(\gamma,s,k) = \exp(sJ_{\gamma,k})
\end{align*}
where the matrix $J_{\gamma,k}$ is defined by
\begin{equation*}
	J_{\gamma,k}\defeq
	\begin{pmatrix}
		\gamma & &\\
		1      &\ddots &\mbox{0}\\
		&\ddots&\ddots&\\
		\mbox{0}&& k\!-\!1&\gamma
	\end{pmatrix}.
\end{equation*}
This leads to
\begin{align}
	\label{eq:derivative of exp}
	\frac{\dd}{\dx}\exp(\gamma,x,k)=J_{\gamma,k}\exp(\gamma,x,k).
\end{align}
With $\|\cdot\|$ denoting the operator norm and $\|\cdot\|_{\mathsf{HS}}$
denoting the Hilbert-Schmidt norm, the following (crude) bound holds
for every $\delta > 0$:
\begin{align}	\label{eq:matrix norm}
\|\exp(\gamma,s,k)\|
&\leq \|\exp(\gamma,s,k)\|_{\mathsf{HS}} \leq C' (1+|s|)^{k-1} e^{\Real(\gamma) s}
\leq Ce^{\Real(\gamma) s+ \delta |s|}
\end{align}
for some constant $C'>0$ depending on $k$ only and another constant $C>0$
depending on $k$ and $\delta>0$.
For a vector $x$, we write $x^\transp$ for its transpose.
Further, we write $\e_1, \e_2,\ldots$ for the canonical base vectors in Euclidean space.
Here, for ease of notation, we are slightly sloppy as we do not specify the dimension of that space
(formally, all Euclidean spaces may be embedded into an appropriate infinite-dimensional space such as $\ell^2$).
Then, for instance,
\begin{equation*}
\exp(\gamma,s,k) \cdot \e_1
= e^{\gamma s}
\begin{pmatrix}
1	\\	s	\\	s^2	\\	\vdots	\\	s^{k-1}
\end{pmatrix}.
\end{equation*}

Throughout the paper, for $\Real(\lambda) > \vartheta$, $n \in \N_0$ and $k \in \N$,
we denote by $Z_n(\lambda,k)$ the following random matrix
\begin{equation}	\label{eq:Z_n(lambda,k)}
Z_n(\lambda,k)\defeq\sum_{|u|=n}\exp(\lambda, -S(u), k).
\end{equation}
We set $Z_n(\lambda) \defeq Z_n(\lambda,1)$  for $\Real(\lambda) \geq \vartheta$.
In particular, $\mu(\theta) = \E[Z_1(\theta)]$
and (A\ref{ass:first moment}) becomes $\E[Z_1(\vartheta)] < \infty$.

\section{Nerman's martingales  as general branching processes}	\label{sec:Nerman's martingales as CMJ processes}

 Nerman's martingale and its complex counterparts are crucial for the paper
	as they constitute the building blocks for the asymptotic expansion of $\cZ^\varphi$.
	In the present section, we demonstrate how these martingales can be represented in terms
	of Crump-Mode-Jagers processes and which characteristics come into play.

 Suppose that (A\ref{ass:Malthusian parameter}) holds and that $\L\mu(\vartheta) < \infty$
	for some $0 < \vartheta < \alpha$. (Notice that the last condition is implied by (A\ref{ass:first moment}).)
 Further, let $\lambda \in \C$ with $\Real(\lambda) > \vartheta$ be a root of multiplicity $k=k(\lambda) \in \N$ of the mapping $z \mapsto \L\mu(z)-1$, i.e.,
\begin{align}
	&\L\mu(\lambda) = \E\bigg[\sum_{j=1}^Ne^{-\lambda X_j}\bigg] = 1,		\label{eq:lambda is a root}	\\
	&\L\mu^{(l)}(\lambda) = (-1)^l \E\bigg[\sum_{j=1}^N X_j^l e^{-\lambda X_j}\bigg] = 0	\quad	\text{for } l=1,\ldots,k(\lambda)-1,	\label{eq:lambda is at least kth root}	\\
	&\L\mu^{(k(\lambda))}(\lambda) \neq 0.	\label{eq:lambda is kth root}
\end{align}
Conditions \eqref{eq:lambda is a root} and \eqref{eq:lambda is at least kth root} are equivalent to
\begin{align}	\label{eq:lambda is a matrix root}
	\E[Z_1(\lambda,k)]=\E\bigg[\sum_{j=1}^N\exp(\lambda, -X_j,k)\bigg] = I_k
\end{align}
where $I_k$ is the $k \times k$ identity matrix.

Define the random matrix
\begin{equation*}
Y_u \defeq Z_{u,1}(\lambda,k)-I_k = \int \exp(\lambda,-x,k) \, \xi_u(\dx) - I_k.
\end{equation*}
Notice that $\E[Y_u]$ is the $k \times k$ zero matrix by \eqref{eq:lambda is a matrix root}.
Moreover, if $\Real(\lambda) \geq \frac\alpha 2$ and (A\ref{ass:second moment}) is satisfied then,
by the penultimate inequality in \eqref{eq:matrix norm},
we have $\E[\|Y_u \|^2] \leq C_{\lambda,k} < \infty$
for some constant $C_{\lambda,k}$ that depends only on $\lambda$ and $k$.

Now for $\lambda\in\C$ such that $\Real(\lambda)>\vartheta$ and \eqref{eq:lambda is a root} through \eqref{eq:lambda is kth root} holds
we define  matrix-valued characteristics $\phi_\lambda$ and $\chi_{\lambda,I}$,
which play a crucial role in the proof of the main theorem.
For $t \in \R$, we set
\begin{align}	\label{eq:phi_lambda}
\phi_{\lambda}(t) &\defeq \sum_{j=1}^N \1_{[0,X_j)}(t) \exp({\lambda,t-X_j},k),
\end{align}
and for any interval $I=[a,b)\cap \R$ with $-\infty \leq a<b<\infty$
\begin{align}	\label{eq:chi_lambda}
\chi_{\lambda,I}(t)	\defeq	\1_{I}(t) \exp({\lambda, t},k)Y_\varnothing
\quad\text{ and }\quad
\chi_\lambda(t)		\defeq	\chi_{\lambda,(-\infty,0)}(t).
\end{align}

By definition, characteristics take values in $\R^d$ for some $d \in \N$,
but here we use an obvious extension to $\C$ by splitting into real and imaginary part.
 Note also that  both $\phi_{\lambda}$ and $\chi_{\lambda,I}$ are $\sigma(\xi)$-measurable
and, in particular, the tuples $(\xi_u,\phi_{\lambda,u},\chi_{\lambda,u})$, $u\in\I$ are i.\,i.\,d.,
where $\chi_{\lambda,u} = \1_{(-\infty,0)}(t) \exp({\lambda, t},k)Y_u, t \in \R$.

\begin{lemma}	\label{Lem:phi^1 and phi^2}
	Suppose that (A\ref{ass:Malthusian parameter}) through (A\ref{ass:second moment}) hold.
	Let $\lambda\in\Lambda_{\ge}$, let $k$ denote the multiplicity of $\lambda$, and fix $x,y\in\R^k$.
	\begin{enumerate}[(a)]
		\item
		The characteristic $ x^\transp\phi_{\lambda} y $
		satisfies (A\ref{ass:mean growth}), (A\ref{ass:variance growth}) and (A\ref{ass:local ui of phi^2}).
		\item
		Let $\Real(\lambda)>\frac \alpha 2$,
		and $I=[a,b)\cap \R$ be an interval with $-\infty\le a<b<\infty$.
		Then the characteristic $x^\transp\chi_{\lambda,I} y$
		satisfies (A\ref{ass:mean growth}), (A\ref{ass:variance growth})
		and (A\ref{ass:local ui of phi^2}).
	\end{enumerate}
\end{lemma}
\begin{proof}
	Clearly, both characteristics $ x^\transp\phi_{\lambda} y$
	and $x^\transp\chi_\lambda y$ have c\`adl\`ag paths.
	Without loss of generality, we may assume that $|x|,|y| \leq 1$.
	In view of Proposition~\ref{Prop:f^*} it suffices to verify the
	integrability conditions \eqref{eq:int Evarphi*e^-alphax dx finite}
	and \eqref{eq:int Evarphi*^2e^-alphax dx finite}.
	Let us first assume that $\Real(\lambda)>\frac\alpha 2$.
	Then we take $\gamma\in(\frac \alpha 2,\Real(\lambda))$ and from \eqref{eq:matrix norm}
	we infer the existence of a constant $C$
	that depends only on $\lambda,\gamma$ and $k$ such that,
	for $s \leq 0$ and $|t-s| \leq 1$, we have
	\begin{equation*}
		\big| x^\transp\exp(\lambda ,s,k)y\big|
		\leq \|\exp(\lambda,s-t,k)\|_{\mathsf{HS}} \|\exp(\lambda ,t,k)\|_{\mathsf{HS}}
		\leq C e^{\gamma t}.
	\end{equation*}
	We can thus write for $t \in \R$
	\begin{align*}
		(x^\transp\phi_{\lambda} y )^*(t)&
		= \sup_{|s-t|\le 1}\big| x^\transp\phi_{\lambda}(s) y \big|\\
		&\leq \sup_{|s-t|\le 1}\sum_{j=1}^N \1_{[0,X_j)}(s) \big| x^\transp 
		\exp(\lambda ,s-X_j,k) y \big| 	\\
		&\leq C\sum_{j=1}^N \1_{[-1, X_j+1)}(t)  e^{\gamma (t-X_j)}.
	\end{align*}
	Hence, we have
	\begin{align*}
		\int \E\Big[ (x^\transp\phi_{\lambda} y )^*(t)\Big] e^{-\alpha t} \, \dt
		&\leq	
		C \int \limits_{-\infty}^{\infty}\E\bigg[\sum_{j=1}^N \1_{[-1, X_j+1)}(t) e^{\gamma (t-X_j)} \bigg]
		e^{-\alpha t} \, \dt	\\
		&\leq C \E\bigg[\sum_{j=1}^N  e^{-\gamma X_j}\bigg]
		\int \limits_{-1}^{\infty} e^{(\gamma -\alpha)t} \, \dt<\infty
	\end{align*}
	by (A\ref{ass:first moment}).
	Further,
	\begin{align*}
		\int & \E\Big[\Big( (x^\transp\phi_{\lambda} y )^*(t)\Big)^2\Big] e^{-\alpha t} \, \dt	\\
		&\leq	
		C^2 \int \limits_{-\infty}^{\infty}\E\bigg[\sum_{1 \leq i,j \leq N}\1_{[-1,X_i+1)}(t)  e^{\gamma (t-X_i)}
		\1_{[-1,X_j+1)}(t)  e^{\gamma (t-X_j)}  \bigg] e^{-\alpha t} \, \dt	\\
		&\leq C^2\E\bigg[\sum_{1 \leq i,j \leq N} \int_{-1}^{(X_i \wedge X_j)+1}e^{(2\gamma -\alpha)t} \, \dt\
		e^{-\gamma X_i} e^{-\gamma X_j} \bigg]\\
		&\leq  \frac{e^{2\gamma-\alpha}C^2}{2\gamma-\alpha}
		\E\bigg[\sum_{1 \leq i,j \leq N} e^{(2\gamma -\alpha)(X_i \wedge X_j)}
		e^{-\gamma X_i} e^{-\gamma X_j} \bigg]\\
		&\leq \frac{e^{2\gamma-\alpha}C^2}{2\gamma-\alpha}
		\E\bigg[\sum_{1 \leq i,j \leq N} e^{(\gamma-\frac{\alpha}{2})
			(X_i + X_j)}
		e^{-\gamma X_i} e^{-\gamma X_j} \bigg]	\\
		&= \frac{e^{2\gamma-\alpha}C^2}{2\gamma-\alpha}\E\bigg[\Big(\sum_{j=1}^N
		e^{-\frac{\alpha}{2} X_j}\Big)^2 \bigg]<\infty
	\end{align*}
	by (A\ref{ass:second moment}).
	Now	assume that $\Real(\lambda)=\frac\alpha 2$.
	Then for $s \leq 0$ and $|t-s| \leq 1$, we have, again by \eqref{eq:matrix norm},
	\begin{equation*}
		\big| x^\transp\exp(\lambda ,s,k)y\big|
		\leq C(1+|t|)^{k-1} e^{\frac{\alpha}{2} t}
	\end{equation*}
	for some $C$ depending on $\lambda$ and $k$
	(not necessarily the exact constant $C$ from \eqref{eq:matrix norm}, but a larger, finite one).
	This, in turn, gives, for arbitrary $t \in \R$,
	\begin{align*}
		(x^\transp\phi_{\lambda} y )^*(t)&
		=\sup_{|s-t| \leq 1}\big|x^\transp\phi_{\lambda}(s) y\big|\\
		&\leq \sup_{|s-t| \leq 1}\sum_{j=1}^N \1_{[0,X_j)}(s) \big| x^\transp\exp(\lambda ,s-X_j,k) y \big| 	\\
		&\leq C\sum_{j=1}^N \1_{[-1,X_j+1)}(t)  (1+|t-X_j|)^{k-1} e^{\frac{\alpha}{2} (t-X_j)},
	\end{align*}
	and, consequently,
	\begin{align*}
		\int  \E\big[ (x^\transp\phi_{\lambda} y )^*(t)\big]e^{-\alpha t} \, \dt		
		\leq \frac{C}{\sqrt{2}}\E\bigg[\sum_{j=1}^N(2+X_j)^{k-\frac{1}{2}}e^{-\frac{\alpha}{2} X_j}\bigg]
		\int \limits_{-1}^{\infty}e^{-\frac{\alpha}{2} t}\, \dt < \infty
	\end{align*}
	by (A\ref{ass:second moment}) and
	\begin{align*}
		\int & \E\big[\big( (x^\transp\phi_{\lambda} y)^*(t)\big)^2\big]e^{-\alpha t} \, \dt		\\
		&\leq
		C^2\int \limits_{-\infty}^{\infty}\E\bigg[\sum_{1 \leq i,j \leq N}\1_{[-1,X_i+1)}(t)  (1+|t-X_i|)^{k-1} e^{\frac{\alpha}{2} (t-X_i)}		\\
		& \hphantom{\leq C\int \limits_{-\infty}^{\infty}\E\bigg[\sum_{1 \leq i,j \leq N}}
		\cdot \1_{[-1,X_j+1)}(t)  (1+|t-X_j|)^{k-1} e^{\frac{\alpha}{2} (t-X_j)}  \bigg] e^{-\alpha t} \, \dt		\\
		&\leq
		C^2 \E\bigg[\sum_{1 \leq i,j \leq N} \int \limits_{-1}^{X_i\wedge X_j+1} \, \dt \, (2+X_i)^{k-1} e^{-\frac{\alpha}{2} X_i} (2+X_j)^{k-1}
		e^{-\frac{\alpha}{2} X_j} \bigg]	\\
		&\leq
		C^2 \E\bigg[\sum_{1 \leq i,j \leq N}(X_i\wedge X_j+2) (2+X_i)^{k-1} e^{-\frac{\alpha}{2} X_i} (2+X_j)^{k-1} e^{-\frac{\alpha}{2} X_j} \bigg]\\
		&\leq
		2^{2k-1} C^2 \E\bigg[\Big(\sum_{j=1}^N (1+\tfrac{X_j}{2})^{k-\frac 12}e^{-\frac\alpha 2 X_j}\Big)^2 \bigg] <\infty
	\end{align*}
	again by (A\ref{ass:second moment}), which finish the proof of (a).
	Regarding part (b), notice that (A\ref{ass:mean growth}) holds trivially
	as $\chi_{\lambda,I}$ is centered.
	Further, observe that
	\begin{align*}
		(x^\transp\chi_{\lambda,I} y)^*(t)
		&=\sup_{|s-t|\leq 1} \1_{[a,b)\cap \R}(s) \big| x^\transp\exp({\lambda, s},k) (Z_1(\lambda,k)-I_k)y\big|	\\
		&\leq C\1_{(-\infty,b+1)}(t) \|\exp({\lambda, t},k)\| \|Z_1(\lambda,k)-I_k\|	\\
		&= C^2  \1_{(-\infty,b+1)}(t) e^{\gamma t}\bigg(1+\sum_{j=1}^N e^{-\frac \alpha 2 X_j}\bigg)
	\end{align*}
	where, as before, $\gamma\in(\frac \alpha 2,\Real(\lambda))$.
	Thus
	\begin{align*}
		\int  \E\big[\big( (x^\transp\chi_{\lambda,I} y )^*(t)\big)^2\big]e^{-\alpha t} \, \dt
		\leq C^4\E\bigg[\Big(1+\sum_{j=1}^N e^{-\frac \alpha 2 X_j}\Big)^2\bigg]\int \limits_{-\infty}^{b+1}
		e^{(2\gamma-\alpha) t} \, \dt<\infty	
	\end{align*}
	by (A\ref{ass:second moment}), which completes the proof of (b).
\end{proof}

As a consequence of the above lemma we conclude that, under the assumptions  (A\ref{ass:Malthusian parameter}) -- (A\ref{ass:second moment}), for any $t\in \R$  $\cZ^{\phi_\lambda}_t$ for $\lambda\in\Lambda_\ge$ and $\cZ^{\chi_{\lambda,I}}_t$ for $\lambda\in\Lambda$ are well-defined.
The first one is so as an unconditional limit in $L^1$ by Proposition \ref{Prop:L1 existence of Z_t^varphi}
and the second as an unconditional limit in $L^2$
by (the first part of) Lemma \ref{Lem:centered characteristic->martingale} (by Lemma \ref{Lem:phi^1 and phi^2}, the characteristic $\chi_{\lambda,I}$ satisfies (A\ref{ass:variance growth}); according to the last part of Lemma \ref{Lem:centered characteristic->martingale}, (A\ref{ass:variance growth}) entails \eqref{eq:Esum Varvarphi(t-S(u))<infty}, the principal assumption of the first part of Lemma \ref{Lem:centered characteristic->martingale}). 
In particular,
\begin{align*}
	\cZ_0^{\chi_{\lambda}}=\sum_{u\in\I}\exp(\lambda,-S(u),k)Y_u
\end{align*}
converges unconditionally in $L^2$ and almost surely along admissible orderings of $\I$.

\subsection{Nerman's martingales with complex parameters}	\label{Exa:Nerman's martingales matrix-valued}

For $u\in \I$ we define $\mathcal G_u\defeq\sigma(\xi_v:v \preceq u)$ and, for $t \in \R$,
\begin{equation*}
\F^W_t	\defeq	\sigma(\{A\cap \{S(u)\leq t\}:  u\in\I \text{ and }A\in\mathcal G_u \}).
\end{equation*}
\begin{lemma}	\label{Lem:Nerman's martingales}
Suppose that (A\ref{ass:Malthusian parameter}) holds
and that $\L\mu(\vartheta)<\infty$ for some $0<\vartheta < \alpha$.
Let $\lambda\in\C$ with $\Real(\lambda) > \vartheta$ be a root of $z\mapsto\L\mu(z)-1$
with multiplicity $k$. If the characteristic $\phi_{\lambda}$ satisfies (A\ref{ass:mean growth}), (A\ref{ass:variance growth}) and (A\ref{ass:local ui of phi^2}),
then the following process
\begin{align*}
W_t(\lambda,k) \defeq \exp({\lambda },-t,k) \cdot  \cZ^{\phi_{\lambda}}_t,	\qquad	t \in \R
\end{align*}
is a (matrix-valued) martingale with respect to the filtration $(\F^W_t)_{t \geq 0}$.
Moreover, for any $t\in\R$, it holds
\begin{align}	\label{eq:alternative definition of Nerman's martingale}
W_t(\lambda,k)
=I_k\1_{[0,\infty)}(t) +\cZ^{\chi_{\lambda,[-t,1)}}_0
= \sum_{u \in \cC_t} \exp({\lambda, -S(u)},k)\quad \text{a.\,s.}
\end{align}
where
by definition (see \eqref{eq:coming generation}) $\cC_t = \{uj \in \familytree: S(u) \leq t < S(uj)\}$. In particular,
\begin{align*}
m^{\phi_{\lambda}}_t=\1_{[0,\infty)}(t)\exp(\lambda,t,k),	\quad 	t \in \R.
\end{align*}
\end{lemma}

\begin{remark}	\label{Rem:Nerman's martingales}
It is worth mentioning that for $1 \leq l \leq k$ and the matrix-valued characteristic
$\phi_{\lambda,l}$ obtained by taking the upper left $l \times l$ submatrix, i.e.,
\begin{align*}
\phi_{\lambda,l}(t) &\defeq \sum_{j=1}^N \1_{[0,X_j)}(t) \exp({\lambda,t-X_j},l),
\end{align*}
if $\phi_{\lambda,l}$ satisfies (A\ref{ass:mean growth}), (A\ref{ass:variance growth}) and (A\ref{ass:local ui of phi^2}),
then the proof below carries over and gives that $W_t(\lambda,l)$ is a matrix-valued martingale and
\begin{align*}
	 W_t(\lambda,l)
	= \sum_{u \in \cC_t}  \exp({\lambda, -S(u)},l)	\quad \text{a.\,s.}
\end{align*}
In particular, if the above conditions hold with $l=1$, then we obtain that $$W_t(\lambda,1)=\sum_{u \in \cC_t}  e^{-\lambda S(u)}$$
is a martingale.

\end{remark}

\begin{proof}[Proof of Lemma \ref{Lem:Nerman's martingales}]
(A\ref{ass:Malthusian parameter}), (A\ref{ass:mean growth}) and (A\ref{ass:variance growth}) entail that, by Proposition \ref{Prop:L1 existence of Z_t^varphi}, $\cZ^{\phi_{\lambda}}$ is well-defined as an unconditional limit in $L^1$ and that $W_t(\lambda,k)$ is integrable for any $t\in\R$.

We boldly write
\begin{align}
\exp&({\lambda },-t,k) \cZ^{\phi_{\lambda}}_t
=\sum_{u\in\I}\sum_{j=1}^{N_u} \1_{[0,X_{u,j})}(t-S(u))\exp({\lambda, -S(uj)},k)	\notag 	\\
&=\sum_{u\in\I}\sum_{j=1}^{N_u} \1_{\{S(u) \leq t\}} \1_{\{S(uj)>t\}} \exp({\lambda, -S(uj)},k)		\notag	\\
&=\sum_{u\in\I}\sum_{j=1}^{N_u} \big(\ind{S(u)\le t}-\ind{S(uj)\le t}\big) \exp({\lambda, -S(uj)},k)		\notag	\\
&=\sum_{u\in\I}\sum_{j=1}^{N_u} \ind{S(u)\le t}\exp({\lambda, -S(uj)},k)
-\sum_{|u| \geq 1}\1_{\{S(u) \leq t\}} \exp({\lambda, -S(u)},k)	\notag	\\
&= I_k\1_{[0,\infty)}(t) + \sum_{u\in\I} \1_{\{S(u) \leq t\}}
\exp({\lambda, -S(u)},k)\big(Z_1(\lambda,k)\circ\theta_u-I_k\big),
\label{eq:decomposition of Nerman's martingale}
\end{align}
where the rearrangements of the infinite series in the last two lines are justified by the fact that there are only finitely many non-zero terms almost surely.
Next, note that for any $t \in \R$,
\begin{align}	\label{eq:integrability}
\E\bigg[\sum_{u\in\I}\1_{\{S(u) \leq t\}} \big\|\exp({\lambda, -S(u)},k)\big\|
\bigg(\sum_{j=1}^{N_u} \big\|\exp({\lambda, -X_{u,j}},k)\big\|+1\bigg)\bigg]<\infty.
\end{align}
Indeed, by \eqref{eq:matrix norm},
the expectation in \eqref{eq:integrability} can be bounded by a finite, deterministic constant times
\begin{align*}
\E\bigg[\sum_{u\in\I}\ind{S(u)\le t} e^{-\vartheta S(u)} \big(Z_1(\vartheta)\circ\theta_u + 1\big)\bigg]
\leq \big(\L\mu(\vartheta)\!+\!1\big) \E\bigg[\sum_{u\in\I}\1_{\{S(u) \leq t\}} e^{-\vartheta S(u)} \bigg]
< \infty,
\end{align*}
where we have used the independence between $S(u)$ and $\xi_{u}$.
The finiteness of the last expectation follows from the many-to-one lemma
(Formula \eqref{eq:many-to-one1}), namely,
\begin{align*}
\E\bigg[\sum_{u\in\I}\1_{\{S(u) \leq t\}} e^{-\vartheta S(u)} \bigg]
&= \Esf\bigg[\sum_{n\ge 0} \1_{\{S_n \leq t\}} e^{(\alpha-\vartheta) S_n} \bigg]
\leq e^{(\alpha-\vartheta)t}\U([0,t])<\infty.
\end{align*}
Further, $(W_t(\lambda,k))_{t \geq 0}$ is adapted to the filtration $(\F^W_t)_{t \geq 0}$.
In order to show the martingale property note that, for $0 \leq s < t$,
\begin{align*}
W_t(\lambda,k)-W_s(\lambda,k)=\sum_{u\in\I}\ind{s<S(u)\le t}\exp({\lambda, -S(u)},k)\big(Z_1(\lambda,k)\circ\theta_u-I_k\big),
\end{align*}
and by \eqref{eq:integrability} it suffices to show that for any $u\in\I$
\begin{align}	\label{eq:martingale propery}
\E\big[\ind{s<S(u)\le t}\exp({\lambda, -S(u)},k)\big(Z_1(\lambda,k)\circ\theta_u-I_k\big)\big|\F^W_s\big]
= 0	\quad	\text{a.\,s.}
\end{align}
Let $u,v\in\I$ and note that the fact $S(u)>s$, $S(v)\leq s$ implies  $u \not\preceq v$.
In particular, for such $u$ and $v$, $\xi_u$ is independent of $\mathcal{G}_v$
and hence for any $A\in\mathcal{G}_v$
\begin{align*}
\E\big[\1_{\{s < S(u) \leq t\}} \exp({\lambda, -S(u)},k) \big(Z_1(\lambda,k) \circ \theta_u-I_k\big) \1_{A\cap\{S(v)\leq s\}}\big] = 0,
\end{align*}
where $\E[Z_1(\lambda,k)]=I_k$ was used.
The argument carries over if we take a finite intersection of sets of the type
$A \cap \{S(v) \leq s\}$, $A\in\mathcal G_v$ for different $v\in\I$.
The $\pi$-$\lambda$-theorem (or monotone class theorem) gives \eqref{eq:martingale propery}
and thus proves that $(W_t(\lambda,k))_{t \geq 0}$ is a martingale.
	
 It remains to prove \eqref{eq:alternative definition of Nerman's martingale}.
The first identity of this equation is \eqref{eq:decomposition of Nerman's martingale}.
Further, from the calculation leading towards \eqref{eq:decomposition of Nerman's martingale},
we have
\begin{align*}
W_t(\lambda,k)
&=\sum_{u\in\I}\sum_{j=1}^{N_u} \1_{\{S(u) \leq t\}} \1_{\{S(uj)>t\}} \exp({\lambda, -S(uj)},k)\\
&=\sum_{u\in\I}\sum_{j=1}^{N_u} \1_{\cC_t}(uj)\exp(\lambda, -S(uj),k)
=\sum_{u\in\cC_t}\exp(\lambda, -S(u),k).
\end{align*}

\end{proof}

 The random matrix $W_t(\lambda,k)$ has the following form
\begin{align}
	W_t(\lambda,k) =\begin{pmatrix}
		W^{(0)}_t(\lambda)&		0&		0&		\dots&	0\\
		W^{(1)}_t(\lambda)&		W^{(0)}_t(\lambda)&		0&		\dots&	0\\
		W^{(2)}_t(\lambda)&		2W^{(1)}_t(\lambda)&		W^{(0)}_t(\lambda)&		\dots&	0\\
		\vdots&	\vdots&	\vdots&	\ddots&	\vdots\\
		W^{(k-1)}_t(\lambda)&		{k-1\choose 1} W^{(k-2)}_t(\lambda)&	{k -1\choose 2}W^{(k-3)}_t(\lambda)&		\dots&	W^{(0)}_t(\lambda)
	\end{pmatrix},
\end{align}
where $W_t^{(j)}(\lambda) = \sum_{u \in \cC_t}(-S(u))^j e^{-\lambda S(u)}$ as in \eqref{eq:W_t^(j)(lambda)}.

\subsection{Convergence of Nerman's martingales}	\label{subsec:convergence of martingales}

The following lemma implies Theorem \ref{Thm:martingale convergence}.

\begin{lemma}	\label{Lem:Nerman's martingale is L^2 bounded}
Suppose that (A\ref{ass:Malthusian parameter}) through (A\ref{ass:second moment}) hold
and let $\lambda$ be a solution to \eqref{eq:roots} with multiplicity $k$ and $\Real(\lambda)>\frac\alpha2$.
Then the process $W_t(\lambda,k)$ is an $L^2$-bounded martingale
 with limit given by
\begin{equation*}
W(\lambda,k)\defeq I_k+\sum_{u\in\I}\exp(\lambda,-S(u),k)Y_u = I_k + \cZ^{\chi_{\lambda,(-\infty,1)}}_0,
\end{equation*}
where the series above converges unconditionally in $L^2$.
In particular, for every $0 \leq j \leq k\!-\!1$,
the martingale $(W_r^{(j)}(\lambda))_{t \geq 0}$  converges a.\,s.\ and in $L^2$.
\end{lemma}
\begin{proof}
Fix $x,y\in\R^k$. It suffices to show the corresponding result for the martingale $ x^\transp W_t(\lambda,k)y$.
By Lemma \ref{Lem:phi^1 and phi^2} the centered characteristics
$ x^\transp\chi_{\lambda,(-\infty,1)} y$, $ x^\transp\chi_{\lambda,(-\infty,-t)} y$ and $ x^\transp\chi_{\lambda,[-t,1)} y$
satisfy (A\ref{ass:variance growth}) and (A\ref{ass:local ui of phi^2}).
In particular, by Lemma \ref{Lem:centered characteristic->martingale},
the general branching processes counted with these characteristics are well-defined as unconditional limits in $L^2$,
and, for $t\geq0$,
\begin{align*}
\sum_{u\in\I}& x^\transp\exp(\lambda,-S(u),k)Y_u y
=\cZ^{ x^\transp\chi_{\lambda,(-\infty,1)}y}_0
=\cZ^{ x^\transp\chi_{\lambda,(-\infty,-t)}y}_0
+\cZ^{ x^\transp\chi_{\lambda,[-t,1)}y}_0	\\
&=\sum_{u\in\I}\ind{S(u)>t} x^\transp\exp(\lambda,-S(u),k)Y_uy
+\sum_{u\in\I}\ind{S(u)\leq t} x^\transp\exp(\lambda,-S(u),k)Y_uy.
\end{align*}
Taking into account that, for any $s\in\R$, $t\geq-1$,
\begin{align*}
\Var[ x^\transp\chi_{\lambda,(-\infty,-t)} y] (s) \leq \Var[ x^\transp\chi_{\lambda,(-\infty,1)} y](s),
\end{align*}
and applying the identity \eqref{eq:Var[Z_t^chi]} (by splitting the characteristics into real and imaginary part)
the dominated convergence theorem  yields
\begin{align*}
\lim_{t\to\infty} \E\Big[\Big|\cZ^{ x^\transp\chi_{\lambda,(-\infty,-t)} y}_0\Big|^2\Big]=0.
\end{align*}
In particular, in view of  \eqref{eq:alternative definition of Nerman's martingale}, we infer that
\begin{align*}
 x^\transp W_t(\lambda,k)y= x^\transp y+\cZ^{ x^\transp\chi_{\lambda,(-\infty,1)}y}_0
-\cZ^{ x^\transp\chi_{\lambda,(-\infty,-t)}y}_0
\end{align*}
converges  in $L^2$ as $t \to \infty$.
\end{proof}

\subsection{Limits of Nerman's martingales as general branching processes}	\label{ex:phi^2}
Suppose now that the martingale $(W_t(\lambda,k))_{ t \geq 0}$ is uniformly integrable.
Then it converges in $L^1$ as $t \to \infty$ to some random matrix $W(\lambda,k)$ of the form
\begin{align}
	W(\lambda,k) =\begin{pmatrix}
		W^{(0)}(\lambda)&		0&		0&		\dots&	0\\
		W^{(1)}(\lambda)&		W^{(0)}(\lambda)&		0&		\dots&	0\\
		W^{(2)}(\lambda)&		2W^{(1)}(\lambda)&		W^{(0)}(\lambda)&		\dots&	0\\
		\vdots&	\vdots&	\vdots&	\ddots&	\vdots\\
		W^{(k-1)}(\lambda)&		{k-1\choose 1} W^{(k-2)}(\lambda)&	{k-1 \choose 2}W^{(k-3)}(\lambda)&		\dots&	W^{(0)}(\lambda)
	\end{pmatrix}.
\end{align}
By uniform integrability, $\E[W(\lambda,k)]=I_k$.

\begin{lemma}	\label{Lem:chi_lambda}
Suppose that (A\ref{ass:Malthusian parameter}) through (A\ref{ass:second moment}) hold
and let $\lambda$ be a solution to \eqref{eq:roots} with multiplicity $k$ and $\Real(\lambda)>\frac\alpha2$.
Then, for $t \geq 0$,
\begin{equation*}
\exp(\lambda,-t,k)\cZ^{\chi_\lambda}_t=W(\lambda,k)-W_t(\lambda,k).
\end{equation*}
In particular,  for $t\ge0$,
\begin{align}	\label{eq:phi+chi}
\cZ^{\phi_\lambda+\chi_\lambda}_t=\exp(\lambda,t,k)W(\lambda,k)	.
\end{align}
\end{lemma}
\begin{proof}
Note that for any $s\in \R$, $t \geq 0$, we have
\begin{align*}
\exp(\lambda,-t,k)\chi_{\lambda}(s)=\chi_{\lambda,(-\infty,-t)}(s-t).
\end{align*}
In particular,
\begin{align*}
\exp(\lambda,-t,k)\cZ^{\chi_\lambda}_t=\cZ^{\chi_{\lambda,(-\infty,-t)}}_0,
\end{align*}
which equals $W(\lambda,k)-W_t(\lambda,k)$ by Lemma \ref{Lem:Nerman's martingale is L^2 bounded}.
This together with Lemma \ref{Lem:Nerman's martingales}  implies
\begin{equation*}
\cZ^{\phi_\lambda+\chi_\lambda}_t=\exp(\lambda,t,k)W(\lambda,k),
\end{equation*}
for any $t\ge0$.
\end{proof}

\section{Proofs of the main results}
\label{sec:proofs}

In this section we provide a proof of our main Theorem~\ref{Thm:main}.
We begin by outlining the main ideas and steps.

The basic step is to decompose a given general branching process $\cZ_t^\varphi$
\begin{align*}
	\cZ_t^\varphi = H_\Lambda(t) + H_{\partial \Lambda}(t) + \cZ_t^{\varrho-\phi_{\partial\Lambda}} + \cZ_t^\chi.
\end{align*}
into $H_\Lambda(t)$ and $H_{\partial \Lambda}(t)$, see \eqref{eq:H_Lambda} and \eqref{eq:H_dLambda},
the leading terms in the expansion,
plus two general branching processes $\cZ_t^{\varrho-\phi_{\partial\Lambda}}$ and $\cZ_t^\chi$,
the first one with mean roughly of the order $o(e^{\frac\alpha2t})$ as $t \to \pm\infty$
and the second one with centered characteristic, i.e., $\E[\chi(t)]=0$ for all $t\in \R$.

General branching processes with centered characteristics are studied in Section \ref{subsec:centered characteristics}.
Theorem~\ref{Thm:CLT with centered phi} provides the fluctuations of $\cZ_t^\chi$
for a centered characteristic $\chi$.
There are two different cases of interest.
First, when $\chi$ satisfies (A\ref{ass:variance growth})
and second when $\int_0^t \Var[\chi(x)]e^{-\alpha x} \, \dx \sim c t^{\theta}$ for some $\theta \geq 0$ and $c>0$.
In both cases, $\cZ_t^\chi$, appropriately rescaled, is asymptotically normal.
The main tools to prove this are the martingale central limit theorem and Nerman's strong law of large numbers
for supercritical general branching processes.
The second case requires the additional auxiliary Lemma~\ref{Lem:regular varying variance}.

Characteristics such that the corresponding general branching process has mean function roughly of the order
$o(e^{\frac\alpha2 t})$ as $t \to \pm\infty$ are treated in Section \ref{subsec:slowly growing mean}.
Theorem \ref{Thm:slowly growing mean} of this section yields asymptotic normality for such processes.

Section \ref{subsec:deterministic characteristics} provides a connection between the cases
studied in Sections \ref{subsec:centered characteristics} and \ref{subsec:slowly growing mean}.
Roughly speaking, Lemma~\ref{Lem:reduction to centered phi} enables us to rewrite the process $\cZ_t^f$
for a deterministic characteristic $f$ in the form
\begin{equation*}
	\cZ^f_t=\cZ^{\chi_f}_t+m^{f}_t,
\end{equation*}
for an appropriately chosen centered characteristic $\chi_f$.
This enables us to reduce the case of general branching processes with mean function roughly of the order $o(e^{\frac\alpha2 t})$ as $t \to \pm\infty$
to the case of centered characteristics.

In Section \ref{subsec:Proof of main theorem}, we put all the pieces together and prove the main Theorem \ref{Thm:main}.

We investigate the asymptotic behavior of the general branching process $\cZ_t^\varphi$ counted with characteristic $\varphi$
as $t \to \infty$ in several steps.
In the first step, we prove convergence of Nerman's martingales at complex parameters.

\subsection{Centered characteristics}	\label{subsec:centered characteristics}

In this section we study the fluctuations of $\cZ^\chi_t$ as $t \to \infty$ for centered characteristics,
that is, for characteristics $\chi$ satisfying $\E[\chi(t)]=0$ for all $t\in\R$.
Theorem \ref{Thm:CLT with centered phi} below plays a key role in the proof of our main result Theorem \ref{Thm:main}.
Before we state it, we give a preparatory lemma.

\begin{lemma}	\label{Lem:regular varying variance}
Suppose that (A\ref{ass:Malthusian parameter}) through (A\ref{ass:second moment}) hold.
Let $\theta \geq 0$ and $f: [0,\infty) \to [0,\infty)$ be a continuous function with
$f(x) = O(x^\theta)$ as $x \to \infty$
such that $x^{-\theta} f(x)$ is uniformly continuous on $[1,\infty)$ and the limit
\begin{equation}	\label{eq:power integral}
\lim_{t \to \infty, \, t\in\G} \frac1{t^{\theta+1}} \int_{[0,t]} f(x) \, \ell(\dx) \eqdef c \in (0,\infty)
\end{equation}
exists. Then, for $\varphi(t) \defeq e^{\alpha t} f(t) \1_{[0,\infty)}(t)$,
 we have $\sup_{t \geq 1} e^{-\alpha t} t^{-\theta-1} \E[\cZ^\varphi_t] < \infty$ and
\begin{align}	\label{eq:cZ on critical line}
\frac{e^{-\alpha t}}{t^{\theta+1}} \cZ^\varphi_t \to \frac{cW}{\beta}		\quad \text{as } t \to \infty,\;t\in \G
\quad	\text{a.\,s.}
\end{align}
\end{lemma}
\begin{proof}
For any $t \geq 0$, we have $e^{-\alpha t} N((t-1,t]) = e^{-\alpha t} \cZ_t^{\1_{[0,1)}}$.
Taking expectations and using \eqref{eq:m_t^phi in terms of U} gives
\begin{align*}
e^{-\alpha t} \E[N((t-1,t])]
&= e^{-\alpha t} \E[\cZ_t^{\1_{[0,1)}}]
= \int \1_{[0,1)}(t-x) e^{-\alpha(t-x)} \, \U(\dx),
\end{align*}
which converges to a finite constant as $t \to \infty$, $t \in \G$ by the key renewal theorem and (A\ref{ass:Malthusian parameter}).
Hence,
\begin{align*}
\frac{e^{-\alpha t}}{t^{\theta+1}}\E[\cZ^\varphi_t]
&\leq C_\theta \bigg(\frac{\E[N(\{0\})]}{t} + \frac{e^{-\alpha t}}{t^{\theta+1}} \sum_{n=0}^{\lfloor t\rfloor} e^{\alpha (t-n)} (t-n)^\theta\E[N((n,n+1])] \bigg)
\end{align*}
is bounded  for $t \geq 1$.
It remains to show \eqref{eq:cZ on critical line}.
To this end, first assume that $|t^{-\theta} f(t)| \leq C_\theta < \infty$ for all $t > 0$.
In particular, $f(0)=0$ if $\theta > 0$.
First notice that for any fixed $r>0$  and $t \geq r$,
\begin{align*}
\frac{e^{-\alpha t}}{t^{\theta+1}} \sum_{\substack{u \in \I: \\ S(u) \leq r}} e^{\alpha (t-S(u))} f(t\!-\!S(u))
\leq \frac{C_\theta}{t} \cdot N([0,r]) \to 0
\end{align*}
almost surely as $t \to \infty$.
Hence, almost surely, the limiting behavior of $e^{-\alpha t}t^{-\theta-1} \cZ^\varphi_t$ as $t \to \infty$, $t \in \G$,
is the same as that of
\begin{align}	\label{eq:lim with r<S(u)<=t}
e^{-\alpha t} t^{-\theta-1} \sum_{\substack{u \in \I: \\ r < S(u) \leq t}} e^{\alpha (t\!-\!S(u))} f(t-S(u)).
\end{align}
Now first consider the lattice case and notice that by \cite[Corollary 3.1(b)]{Meiners:2010},
for given $\varepsilon > 0$, with probability $1$ we may choose  (a random) $r \in \N$ so large that
\begin{equation*}
(1-\varepsilon) e^{\alpha k} \frac{W}{\beta} \leq N(\{k\}) \leq (1+\varepsilon) e^{\alpha k} \frac{W}{\beta}
\end{equation*}
for all $k \in \N$, $k \geq r$.
Then, for $t \in \N$ with $t>r$,
\begin{align*}
e^{-\alpha t} t^{-\theta-1} \!\!\! \sum_{\substack{u \in \I: \\ r \leq S(u) \leq t}} \!\!\! e^{\alpha (t\!-\!S(u))} f(t-S(u))
&=
\frac1{t^{\theta+1}} \sum_{k=r}^t f(t-k) e^{-\alpha k} N(\{k\})	\\
&\leq
\frac{(1+\varepsilon) W}{\beta t^{\theta+1}}\sum_{k=0}^{t-r} f(k)
\to (1+\varepsilon) \frac{cW}{\beta}	\quad	\text{as } t \to \infty
\end{align*}
by \eqref{eq:power integral}. The corresponding lower bound can be obtained analogously.
Now \eqref{eq:cZ on critical line} follows by letting $\varepsilon \to 0$.

Next, we turn to the non-lattice case
and fix small $\varepsilon,\delta > 0$.
By \cite[Corollary 3.1(a)]{Meiners:2010},
 with probability $1$, we may choose  (a random) $r \in \delta \N$, $r \geq 1$ so large that
\begin{equation*}
(1-\varepsilon) e^{\alpha t} \frac{e^{\alpha \delta} - 1}{\alpha} \frac{W}{\beta}
\leq N((t,t+\delta])
\leq (1+\varepsilon) e^{\alpha t} \frac{e^{\alpha \delta} - 1}{\alpha} \frac{W}{\beta}
\end{equation*}
for all $t \geq r-\delta$. For $t \geq r$,
define $I_k^{\delta} \defeq [k\delta,(k+1)\delta)$ for $k = 0,\ldots,t_\delta-1$
where $t_\delta \defeq \lfloor \frac{t-r}{\delta} \rfloor$, and $I_{t_\delta}^\delta \defeq [t_\delta \delta, t-r)$.
Notice that $t-t_\delta \delta \geq r$.
Hence, almost surely,
\begin{align}
\limsup_{t \to \infty}
\frac{e^{-\alpha t}}{t^{\theta+1}} \!\!\! & \sum_{\substack{u \in \I: \\ r < S(u) \leq t}} \!\!\! e^{\alpha (t\!-\!S(u))} f(t\!-\!S(u))	\notag	\\
&\leq
\limsup_{t \to \infty} \frac{e^{-\alpha t}}{t^{\theta + 1}}
\sum_{k=0}^{t_\delta} \!\!\! \sum_{\substack{u \in \I: \\ t-S(u) \in I_{k}^\delta}} \!\!\! e^{\alpha  \sup I_k^\delta} f(t\!-\!S(u))	\notag	\\
&\leq
\limsup_{t \to \infty} \frac{e^{-\alpha t}}{t^{\theta + 1}}
\sum_{k=0}^{t_\delta} e^{\alpha  \sup I_k^\delta} N(t-I_k^\delta) \sup_{x \in I_k^\delta} f(x)	\notag	\\
&\leq
(1+\varepsilon) \frac{e^{\alpha \delta} - 1}{\alpha \delta} \frac{W}{\beta} \limsup_{t \to \infty}
\frac{1}{t^{\theta + 1}}
\sum_{k=0}^{t_\delta} \delta \sup_{x \in I_k^\delta} f(x).		\label{eq:g^delta upper bound}
\end{align}
Write $w(\delta) \defeq \sup_{x,y \geq 1,\, |x-y| \leq \delta} |x^{-\theta} f(x)- y^{-\theta} f(y)|$ for the modulus of continuity of $x^{-\theta} f(x)$ on $[1,\infty)$.
By uniform continuity, $w(\delta) \to 0$ as $\delta \to 0$.
We now estimate the $\limsup$ in \eqref{eq:g^delta upper bound}:
\begin{align*}
&\limsup_{t \to \infty} \frac{1}{t^{\theta + 1}}
\sum_{k=0}^{t_\delta} \delta \sup_{x \in I_k^\delta} f(x)	\\
&~\leq
\limsup_{t \to \infty} \bigg(\frac{1}{t^{\theta + 1}} \!\! \int \limits_{0}^{t-r +\delta} \! f(x) \, \dx
+\frac{1}{t^{\theta + 1}} \!\! \int \limits_{0}^{t-r +\delta} \!
((x+\delta)^\theta (x^{-\theta}f(x) + w(\delta)) - x^\theta x^{-\theta}f(x)) \, \dx  \bigg)	\\
&= c + \limsup_{t \to \infty} \bigg(\frac{C_\theta}{t^{\theta + 1}} \!\! \int \limits_{0}^{t-r +\delta} \! \big((x+\delta)^\theta - x^\theta\big) \, \dx
+ \frac{w(\delta)}{t^{\theta + 1}} \!\! \int \limits_{0}^{t-r +\delta} \! (x+\delta)^\theta \dx  \bigg)	\\
&\leq c+\frac{w(\delta)}{\theta+1}.
\end{align*}
Using this in \eqref{eq:g^delta upper bound} gives
\begin{align*}
\limsup_{t \to \infty}
\frac{e^{-\alpha t}}{t^{\theta+1}} \!\!\! & \sum_{\substack{u \in \I: \\ r \leq S(u) \leq t}} \!\!\! e^{\alpha (t\!-\!S(u))} f(t\!-\!S(u))	\\
&\leq
(1+\varepsilon) \frac{e^{\alpha \delta} - 1}{\alpha \delta} \frac{W}{\beta} \Big(c+\frac{w(\delta)}{\theta+1}\Big)
\end{align*}
almost surely. Letting $\varepsilon, \delta \to 0$ yields the upper bound of \eqref{eq:cZ on critical line}.
The lower bound can be obtained analogously.

\noindent
For the general case, we split $f=f_1+f_2$ with
\begin{equation*}
f_2(x) =	\begin{cases}
		f(1) x^\theta	&	\text{for } 0 \leq x \leq 1,	\\
		f(x)					&	\text{for } x \geq 1
		\end{cases}
\end{equation*}
and $f_1 \defeq f - f_2$. Then $f_1,f_2$ are continuous and the previous part of the proof applies to $f_2$.
Further, the limit in \eqref{eq:power integral} is the same if $f$ is replaced by $f_2$.
Define $\varphi_i(t) \defeq e^{\alpha t} f_i(t) \1_{[0,\infty)}(t)$ for $i=1,2$ so that $\varphi=\varphi_1+\varphi_2$.
We conclude
\begin{align*}	
\frac{e^{-\alpha t}}{t^{\theta+1}} \cZ^{\varphi_2}_t \to \frac{cW}{\beta}		\quad \text{as } t \to \infty,\;t\in \G
\quad	\text{a.\,s.}
\end{align*}
On the other hand, as $\varphi_1$ is bounded and supported on $[0,1)$,
we have $e^{-\alpha t}\cZ^{\varphi_1}_t$ converges a.\,s.\ to an a.\,s.\ finite limit
by \cite[Theorem 5.4]{Nerman:1981} and \cite[Theorem 3.2]{Gatzouras:2000},
which finishes the proof.
\end{proof}

\begin{remark}	\label{Rem:regular varying variance}
Notice that in the proof of Lemma \ref{Lem:regular varying variance},
we actually do not use the full power of assumptions (A\ref{ass:first moment}) and (A\ref{ass:second moment}).
Indeed, we only need the assumptions regarding $\xi$ that allow us to apply
\cite[Corollary 3.1]{Meiners:2010}.
What is more, we could replace the application
of \cite[Corollary 3.1(a)]{Meiners:2010} in the non-lattice case by an application of
\cite[Theorem 5.4]{Nerman:1981}
and the application of \cite[Corollary 3.1(b)]{Meiners:2010} in the lattice case
by an application of \cite[Theorem 3.2]{Gatzouras:2000} to get the assertion of the lemma
under the even weaker assumptions of \cite[Theorem 5.4]{Nerman:1981} and \cite[Theorem 3.2]{Gatzouras:2000},
respectively.
\end{remark}

The following theorem gives the central limit theorem in the case of a centered characteristic $\chi$.
 Recall that $\G=\Z$ in the lattice case and $\G=\R$ in the non-lattice case
and that $\F = \sigma(\pi_u: u \in \I)$ where $\pi_u$ is the projection onto the life space of individual $u$
(in particular, $(\xi_u,\zeta_u,\chi_u)$ is $\sigma(\pi_u)$-measurable).

\begin{theorem}	\label{Thm:CLT with centered phi}
 Suppose that (A\ref{ass:Malthusian parameter}) through (A\ref{ass:second moment}) hold.
Let $\chi$ be a real-valued, centered characteristic,
and let $\cN$ be a standard normal random variable independent of $\F$.
\begin{enumerate}[(i)]
	\item
		Suppose that (A\ref{ass:variance growth}) holds for the characteristic $\chi$. Then
		\begin{equation}    \label{eq:1st summand convergenceII}
		e^{-\frac\alpha2 t}\cZ^{\chi}_t
		\stablyto \bigg(\frac{W}\beta \int \limits_{\G} \E[\chi^2](x) e^{-\alpha x} \, \ell(\dx)\bigg)^{\!\!1/2} \cN
		\quad \text{as }	t\to\infty,\,t\in\G.
		\end{equation}
	\item
		Suppose that there are $\theta \geq 0$
		and a function $f$  not vanishing identically on $\G$ and
		satisfying the conditions of Lemma \ref{Lem:regular varying variance},
		$\E[\chi^2(t)] = e^{\alpha t} f(t) \1_{[0,\infty)}(t)$ and
		\begin{align} 	\label{eq:pointwise Lindeberg's condition}
		\E\big[\chi^2(t)\ind{\chi^2(t)> \varepsilon e^{\alpha t} t^{\theta+1}} \big]=o(t^\theta e^{\alpha t})	\quad \text{as }t\to\infty
		\end{align}
		for every $\varepsilon > 0$.
		Then
		\begin{equation}    \label{eq:1st summand convergenceXIII}
		\bigg(e^{\alpha t}\int \limits_{[0,t]} \E[\chi^2(x)] e^{-\alpha x} \, \ell(\dx)\bigg)^{\!\! -1/2}{\cZ_t^{\chi}} \stablyto
		\big(\tfrac{W}{\beta}\big)^{\! 1/2}\cN \text{ as } t\to\infty,\, t\in\G.
		\end{equation}
\end{enumerate}
\end{theorem}

\begin{proof}[Proof of Theorem \ref{Thm:CLT with centered phi}]
 Consider an admissible ordering $v_1,v_2,\ldots$ of $\I$ and put
$\I_n \defeq \{v_1,\ldots,v_n\}$ and $\mathcal{G}_n \defeq \sigma(\pi_{v_j}: j=1,\ldots, n)$.
Now we set
\begin{equation*}
a_t \defeq
\begin{cases}
\frac{1}{\beta} \int_{ \G} \E[\chi^2](x) e^{-\alpha x} \, \ell(\dx) \cdot e^{\alpha t}		&	\quad	\text{in case (i)}, \\[0.7em]
\frac{1}{\beta} \int_{0}^t \E[\chi^2](x) e^{-\alpha x} \, \ell(\dx) \cdot e^{\alpha t}			&	\quad	\text{in case (ii)}
\end{cases}
\end{equation*}
for all $t \in \G$.
 If $\|\chi\|_{L^2(\dProb \otimes e^{-\alpha t}\ell(\dt))} = \int \E[\chi^2(x)] e^{-\alpha x} \, \ell(\dx)=0$ in case (i),
then the assertion is trivial.
Hence, we exclude this case
and may thus assume that $a_t>0$ for all sufficiently large $t \in \G$.
The latter is automatic in case (ii) in view of the assumption that $f$ does not vanish identically on $\G$ and is uniformly continuous.
For  $t$ with $a_t>0$, we define
\begin{equation*}
M_n(t) \defeq a_t^{-1/2} \sum_{u \in \I_n} \chi_u(t-S(u)).
\end{equation*}
 Then $(M_n(t),\mathcal{G}_n)_{n \in \N_0}$ is a centered,
martingale and bounded in $L^2$
by Lemma \ref{Lem:centered characteristic->martingale}
We write $M(t)$ for its limit (almost sure and in $L^2$).
Let $(t_n)_{n \in \N}$ be an increasing sequence in $\G$ that diverges to infinity.
Then there exists an increasing sequence $(k_n)_{n \in \N}$
such that $\E[(M(t_n)-M_{k_n}(t_n))^2] \leq 2^{-n}$ for every $n \in \N$
and, therefore, $M(t_n)-M_{k_n}(t_n)$ converges to $0$ almost surely as $n \to \infty$.
In view of Slutsky's theorem \cite[Theorem 8.6.1]{Resnick:2014},
in order to prove the convergence in distribution of
$M(t_n) = M(t_n)-M_{k_n}(t_n) + M_{k_n}(t_n)$ as $n \to \infty$,
it suffices to prove convergence in distribution of $M_{k_n}(t_n)$ as $n \to \infty$.
For the latter,
we rely on the martingale central limit theorem \cite[Corollary 3.1 on p.~58]{Hall+Heyde:1980}.
To apply the cited theorem,
it suffices to verify that
\begin{align}
&a_{t_n}^{-1} \sum_{j=1}^{k_n}\E\Big[ \chi_{v_j}^2(t_n-S(v_j))\Big| \mathcal{G}_{j-1}\Big]\Probto W	\quad \text{as } n \to \infty	\label{eq:mgale CLT1}	\\
&a_{t_n}^{-1} \sum_{j=1}^{k_n}\E\Big[ \chi_{v_j}^2(t_n-S(v_j))\ind{|\chi_{v_j}(t_n-S(v_j))|>\varepsilon a_{t_n}^{1/2}}\Big| \mathcal{G}_{j-1}\Big]
\Probto 0	\quad \text{as } n \to \infty	   \label{eq:mgale CLT2}
\end{align}
for every $\varepsilon>0$.
To prove \eqref{eq:mgale CLT1} observe that
\begin{align*}
a_{t_n}^{-1}\E\bigg[\sum_{j=k_n+1}^{\infty}\E\Big[\chi_{v_j}^2(t_n-S(v_j))\Big| \mathcal{G}_{j-1}\Big]\bigg] = \E\big[(M(t_n)-M_{k_n}(t_n))^2\big] \leq 2^{-n}
\end{align*}
and hence \eqref{eq:mgale CLT1} is equivalent to
\begin{align}	\label{eq:3.1}
a_{t_n}^{-1}\sum_{j=1}^{\infty}\E\Big[\chi_{v_j}^2(t_n-S(v_j))\Big| \mathcal{G}_{j-1}\Big]
=a_{t_n}^{-1}\sum_{u\in\I}\E[\chi^2](t_n-S(u))
\Probto W.
\end{align}
In case (i), \eqref{eq:3.1} is equivalent to
\begin{align*}
&e^{-\alpha t_n}\sum_{u\in\I}\E[\chi^2](t_n-S(u))
= e^{-\alpha t_n} \cZ_{t_n}^{\E[\chi^2]}
~\Probto\frac {W}{\beta} \int e^{-\alpha x} \E[\chi^2](x) \, \ell(\dx),
\end{align*}
which follows from \cite[Theorem 6.1]{Jagers+Nerman:1984} in the non-lattice case.
The lattice case is analogous.
Lemma~\ref{Lem:regular varying variance} gives \eqref{eq:3.1} in case (ii).

Now we show \eqref{eq:mgale CLT2}. Let $v_2(t,s) \defeq \E[\chi^2(t) \ind{|\chi(t)|>s}]$ for $t \in \R$ and $s \geq 0$.
In case (i), for any $\varepsilon > 0$,
\begin{align*}
\limsup_{n \to \infty} e^{-\alpha t_n}
&\sum_{j=1}^{k_n}\E\big[\chi_{v_j}^2(t_n-S(v_j))\ind{|\chi_{v_j}(t_n-S(v_j))|>\varepsilon e^{\alpha t_n/2}}
\, \big| \,  \mathcal{G}_{j-1}\big]	\\
&\leq \limsup_{n \to \infty} e^{-\alpha t_n}\sum_{j=1}^{\infty}v_2(t_n-S(v_j),\varepsilon e^{\alpha t_n/2})\\
&\leq \liminf_{s \to \infty} \limsup_{n \to \infty} e^{-\alpha t_n} \cZ_{t_n}^{v_2(\cdot,s)}\\
&=\liminf_{s \to \infty} \frac {W}{\beta} \int v_2(x,s)  e^{-\alpha x} \, \ell(\dx) = 0	\qquad	\text{a.\,s.}
\end{align*}
by \cite[Theorem 6.1]{Jagers+Nerman:1984} in the non-lattice case and the dominated convergence theorem.
The lattice case is analogous.

\noindent
We turn to case (ii) and fix $\varepsilon > 0$.
We infer from \eqref{eq:pointwise Lindeberg's condition} that for any $\varepsilon,\delta>0$ there is a $T \geq 0$
such that, for all $t \geq T$,
\begin{equation*}
v_2(t, \varepsilon e^{\alpha t/2} t^{\frac{\theta+1}{2}}) = \E\Big[\chi(t)^2\ind{|\chi(t)|>\varepsilon {e^{\alpha t/2} t^{\frac{\theta+1}{2}}}} \Big]
\leq \delta e^{\alpha t}t^\theta.
\end{equation*}
Therefore, with $\|\E[\chi^2]\|_{[0,T]} \defeq \sup_{x \in [0,T]} \E[\chi^2](x)$,
\begin{align*}
&\limsup_{n\to\infty} \frac{e^{-\alpha t_n}}{t_n^{\theta+1}}
\sum_{j=1}^{k_n}\E\Big[\chi_{v_j}^2(t_n\!-\!S(v_j))\ind{|\chi_{v_j}(t_n-S(v_i))|>\varepsilon {e^{\alpha t_n/2} t_n^{\frac{\theta+1}{2}}}} \, \Big|\,  \mathcal{G}_{j-1}\Big]\\
&~\leq \limsup_{n\to\infty} \frac{e^{-\alpha t_n}}{t_n^{\theta+1}} \sum_{u\in \I} v_2(t_n\!-\!S(u),\varepsilon {e^{\alpha t_n/2} t_n^{\frac{\theta+1}{2}}}) \\
&~= \limsup_{n\to\infty} \frac{e^{-\alpha t_n}}{t_n^{\theta+1}}
\bigg( \!\!\!\! \sum_{\substack{u \in \I: \\S(u) \leq t_n \!-\! T}} \!\!\!\!\!\!\!\! v_2(t_n\!-\!S(u),\varepsilon {e^{\alpha t_n/2} t_n^{\frac{\theta+1}{2}}}) + \|\E[\chi^2]\|_{[0,T]} \cdot N((t_n-T,t_n])\bigg) \\
&~\leq \limsup_{n \to\infty} \frac{e^{-\alpha t_n}}{t_n^{\theta+1}}\bigg(\delta  \!\!\!\! \sum_{\substack{u \in \I: \\S(u) \leq t_n \!-\! T}} \!\!\!\!\!\!\!\! e^{\alpha (t_n\!-\!S(u))}(t_n-S(u))^\theta + \|\E[\chi^2]\|_{[0,T]} \cdot N((t_n-T, t_n]) \bigg)	\\
&~\leq \frac{\delta W}{\beta (\theta+1)}	\qquad	\text{a.\,s.}	
\end{align*}
by Lemma \ref{Lem:regular varying variance} with $f(t)=t^\theta$, $t \geq 0$
and the fact that $e^{-\alpha t_n}N((t_n-T,t_n])$ converges a.\,s.\ by \cite[Theorem 5.4]{Nerman:1981} in the non-lattice case
and by \cite[Theorem 3.2]{Gatzouras:2000} in the lattice case.
Since $\delta > 0$ was arbitrary, we conclude that the limit is zero and, therefore, \eqref{eq:mgale CLT2} holds true
in both cases.

It remains to justify that the convergence is stable
and that limiting random variable $\cN$ is independent of $\F$.
Although, this is not stated explicitly in \cite[Theorem 3.2]{Hall+Heyde:1980},
it follows from the proof of the preceding Lemma  3.1 of \cite{Hall+Heyde:1980}, cf.~Eq.~(3.15) there,
that is, for any $E \in \F$, we have
\begin{align*}
\E\Big[e^{\imag \theta M_{k_n}(t_n)}\1_E\Big] \to \E\Big[e^{-W \frac{\theta^2}2 }\1_E\Big]
\end{align*}
for every $\theta \in \R$.
The latter is equivalent, by a standard approximation argument, to say that for any $\F$-measurable random variable $Y$
\begin{align*}
\E\Big[e^{\imag \theta M_{k_n}(t_n)}e^{\imag \eta Y }\Big]
\to \E\Big[e^{-W\frac{\theta^2}2 }e^{\imag \eta Y }\Big]
=\E\Big[e^{\imag \theta \sqrt W \cN}e^{\imag \eta Y }\Big]
\end{align*}
for a standard normal variable $\cN$ independent of $(W,Y)$.
This also implies the stable convergence by \cite[Proposition 1]{Aldous+Eagleson:1978}.

\end{proof}

\subsection{Deterministic characteristics}	\label{subsec:deterministic characteristics}

Let $f$ be a deterministic characteristic, i.e.,
a c\`adl\`ag function $f:\R \to \R$.
We investigate the behavior of $\cZ^f_t$ as $t \to \infty$
by means of an auxiliary centered random characteristic $\chi_f$ defined by
\begin{equation}	\label{eq:chi_f}
\chi_f(t)	\defeq	f*\xi*V(t)-f*\mu*V(t) =  m^f * \xi(t) - m^f * \mu(t),
\end{equation}
where $V(\cdot) = \sum_{n=0}^\infty \mu^{*n}(\cdot) = \E[\sum_{u \in \I} \delta_{S(u)}(\cdot)]$ and $*$ denotes
 Lebesgue-Stieltjes convolution.
For instance, for every $t \in \R$,
\begin{equation*}
f*V(t) = \int f(t-x) \, V(\dx) = \E\bigg[\sum_{u \in \I} f(t-S(u))\bigg] = m_t^f
\end{equation*}
if the integrals are well-defined.
However, the latter is not guaranteed a priori.
The following lemma provides a sufficient condition along with an important connection
between $\cZ^f_t$ and $\cZ^{\chi_f}_t$.

\begin{lemma}	\label{Lem:reduction to centered phi}
Assume that (A\ref{ass:Malthusian parameter}) holds.
Let $f:\R \to \R$ be a deterministic c\`adl\`ag function such that $t \mapsto f(t)e^{-\alpha t}$ is directly Riemann integrable.
\begin{enumerate}[(a)]
	\item
		The characteristic $\chi_f$ given by \eqref{eq:chi_f} is well-defined and has almost surely c\`adl\`ag paths.
	\item	
		If (A\ref{ass:second moment}) holds and
		the function $t\mapsto m_t^f e^{-\frac\alpha2 t} (1+t^2)$ is bounded,
		then the characteristic $\chi_f$ satisfies (A\ref{ass:variance growth}) and (A\ref{ass:local ui of phi^2}).
	\item
		If (A\ref{ass:second moment}) holds and
		the function $t\mapsto m_t^f e^{-\frac\alpha2 t} (1+t^2)$ is bounded, then for any $t \in \R$, $\E[\chi_f(t)]=0$ and
		\begin{equation}	\label{eq:Z_t^chi_f}
		\cZ^{f}_t-m^{f}_t=\cZ^{\chi_f}_t	\qquad	\text{a.\,s.\quad for all } t \in \R.
		\end{equation}
	\item
		If $f$ is supported on $[0,\infty)$, i.e., if $f(t) = 0$ for all $t<0$, then \eqref{eq:Z_t^chi_f} also holds.
\end{enumerate}
\end{lemma}

\begin{proof}
A function $g:R \to \R$ is directly Riemann integrable if and only if $g_+$ and $g_-$, the positive and negative part of $g$, respectively, are.
Hence, if $t \mapsto f(t) e^{-\alpha t}$ is directly Riemann integrable, then so is $t \mapsto |f(t)| e^{-\alpha t}$.

(a)
In order to see that $\chi_f$ is well-defined, it suffices to check that $|f|*\mu*V(t)$ is finite for all $t \in \R$.
Indeed, as $\mu * V \leq \delta_0+\mu*V = V$, putting $g(t) \defeq |f(t)| e^{-\alpha t}$ we obtain
\begin{align}
|f|*\mu*V(t) \leq |f|*V(t) &= \sum_{n=0}^\infty \E\bigg[ \sum_{|u|=n} |f|(t-S(u))\bigg]	\notag	\\
&= e^{\alpha t} \sum_{n=0}^\infty \E\bigg[ \sum_{|u|=n} e^{-\alpha S(u)} g(t-S(u))\bigg]
= e^{\alpha t} g*\U(t) < \infty,	\label{eq:|f|*mu*V}
\end{align}
where we have used the many-to-one formula \eqref{eq:many-to-one1} in the next-to-last step
and the direct Riemann integrability of $g$ in combination with \eqref{eq:U unif loc finite} in the last.
To prove that $\chi_f$ has c\`adl\`ag paths almost surely,
it suffices to show that $f^**\mu*V$ is finite.
This is justified by the fact that $f$ has c\`adl\`ag paths together with the dominated convergence theorem.
Since $f^**\mu*V \leq f^** V$, we have to check that the latter is finite.
Further, $e^{-\alpha t} f^*(t) \leq e^\alpha g^*(t)$.
Therefore, by a calculation analogous to \eqref{eq:|f|*mu*V},
it is enough to show that $g^**\U$ is finite,
which, in view of \eqref{eq:U unif loc finite},
is true if $g^*$ is directly Riemann integrable.
This however follows from the converse part of Proposition \ref{Prop:f^*}(a)
since $g$ is directly Riemann integrable.	\smallskip

\noindent

(b) By Proposition \ref{Prop:f^*}(c), it suffices to show that
\begin{equation}	\label{eq:int Echi_f*^2e^-alphax dx finite}
\int \E\big[(\chi_f^*)^2 \big](x) e^{-\alpha x} \, \dx < \infty,
\end{equation}
To this end, note that, since $\chi_f= m^f * \xi - m^f * \mu$,
we have, for any $x \in \R$,
\begin{align*}
\E\big[|\chi^*_f(x)|^2\big]
\leq 2\E\Big[\sup_{|t-x| \leq 1} |m^f \!* \xi(t)|^2\Big] + 2 \!\!\! \sup_{|t-x| \leq 1} |m^f \!* \mu(t)|^2
\leq 4\E\Big[\sup_{|t-x| \leq 1} |m^f \!* \xi(t)|^2\Big],
\end{align*}
where we have used Jensen's inequality.
For $|t-x| \leq 1$, we obtain
\begin{align*}
1+(x-X_j)^2
&= 1+(t-X_j + x- t)^2
\leq 1 + 2(t-X_j)^2 + 2(x-t)^2 \leq 3 + 2(t-X_j)^2	\\
&\leq 3 (1+(t-X_j)^2)
\end{align*}
for $j=1,\ldots,N$, and, therefore, with $C \defeq 3e^{\frac{\alpha}{2}} \sup_{t \in \R} e^{-\frac\alpha2 t} (1+t^2) |m_t^f|$,
\begin{align*}
|m^f \!* \xi(t)|^2
=\bigg| \sum_{j=1}^{N} m_{t-X_j}^f \bigg|^2
\leq C^2
\sum_{1 \leq i,j \leq N} \frac{e^{\frac{\alpha}{2}(x-X_i)}}{1+(x-X_i)^2}\frac{e^{\frac{\alpha}{2}(x-X_j)}}{1+(x-X_j)^2}.
\end{align*}
Thus, since
\begin{align*}
\int\frac{1}{1+(x-X_i)^2}\frac{1}{1+(x-X_j)^2} \, \dx \leq \int \frac{\dx}{1+(x-X_i)^2} = \pi,
\end{align*}
we conclude
\begin{align*}
\int \E\big[|\chi^*_f(x)|^2\big]e^{-\alpha x} \, \dx
\leq 4\pi C^2 \E\bigg[\sum_{1 \leq i,j \leq N} e^{-\frac\alpha 2X_i}e^{-\frac\alpha 2X_j}\bigg] < \infty
\end{align*}
from assumption
(A\ref{ass:second moment}).	\smallskip

\noindent

(c) By part (a), $ |f|* \mu * V(t)$ is finite for all $t \in \R$ and, hence,
$\E[\chi_f](t) = \E[f*\xi*V(t)] - f*\mu*V(t) = f*\mu*V(t)-f*\mu*V(t) = 0$ for all $t \in \R$.
Further, since $f$ satisfies (A\ref{ass:mean growth}) by assumption
and trivially also (A\ref{ass:variance growth}),
Proposition \ref{Prop:L1 existence of Z_t^varphi} implies that $\cZ^{f}_t$
converges in $L^1$ for every $t \in \R$.
The characteristic $\chi_f$ on the other hand satisfies (A\ref{ass:variance growth}) by part (b)
and trivially also (A\ref{ass:mean growth}) because it is centered.
Thus, Proposition \ref{Prop:L1 existence of Z_t^varphi} yields that also
the series defining $\cZ^{\chi_f}_t$ converges unconditionally in $L^1$ for all $t \in \R$.

In particular, for any $n \in \N$, the infinite series
\begin{equation*}
\sum_{0 \leq |u| \leq n} \chi_{f,u}(t-S(u))
\end{equation*}
also converges unconditionally in $L^1$ and so is well-defined and converges to $\cZ^{\chi_f}_t$ as $n\to \infty$  in $L^1$.
Moreover, due to the fact that $f$ is deterministic, $\chi_f$ is $\xi$-measurable.
For $u \in \I$, we have $\chi_{f,u}(t) = f*\xi_u*V(t)-f*\mu*V(t)$.
Using this and $V = \delta_0 + \mu*V$, we infer
\begin{align*}
\sum_{0 \leq |u| \leq n} \chi_{ f,u}(t-S(u))
& ~= \sum_{0 \leq |u| \leq n} \big(f*\xi_u*V(t-S(u)) - f*\mu*V(t-S(u))\big)	\\
&~= \sum_{1 \leq |u| \leq n+1} \! f*V(t-S(u)) - \sum_{0 \leq |u| \leq n} f*\mu*V(t-S(u))	\\
&~= \sum_{1 \leq |u| \leq n} \! f(t-S(u)) - f*\mu*V(t) + \sum_{|u|= n+1} f*V(t-S(u))		\\
&= \sum_{0 \leq |u| \leq n} \! f(t-S(u)) - f*V(t) + \sum_{|u|=n+1} f*V(t-S(u)).
\end{align*}
In the last line, $f*V(t) = m_t^f$.
The manipulations in the above chain of equalities are justified
by
the fact that
\begin{equation*}
\E\bigg[\sum_{0 \leq |u| \leq n+1} |f*V(t-S(u))| \bigg] < \infty.
\end{equation*}
The finiteness of the above expectation follows from
\begin{align*}
\E\bigg[\sum_{|u|=k} \! |f*V(t-S(u))| \bigg]
&\le \E\bigg[\sum_{|u|=k} \! |f|*V(t-S(u)) \bigg]
=\E\bigg[\sum_{|u|\ge k} \! |f|(t-S(u)) \bigg]\\
&\le \E\bigg[\sum_{u\in \I} \! |f|(t-S(u)) \bigg]=m^{|f|}_t<\infty.
\end{align*}
The dominated convergence theorem  yields
$$
\sum_{|u|=n+1} f*V(t-S(u))\to 0\quad \text{and} \sum_{0 \leq |u| \leq n} \! f(t-S(u))\to \cZ^{f}_t
$$
as $n \to \infty$
in $L^1$.

\smallskip

\noindent

(d) The last calculation carries over if $f(t) = 0$ for all $t<0$.
Indeed, the latter condition implies $\chi_f(t) = 0$ for all $t<0$
and hence with probability one,
all sums have only finitely many non-vanishing terms
(since only finitely many individuals are born before any fixed time almost surely
by \cite[Theorem 6.2.3]{Jagers:1975}).
\end{proof}

Lemma \ref{Lem:reduction to centered phi} has the following corollary.

\begin{corollary}	\label{Cor:reduction to centered phi}
If the assumptions of Lemma \ref{Lem:reduction to centered phi}(b) are satisfied,
then, for every $t \in \R$,
\begin{equation*}
\Var[\cZ^{f}_t] = m^{\chi_f^2}_t < \infty.
\end{equation*}
\end{corollary}

\begin{proof}
By Lemma \ref{Lem:reduction to centered phi}, we have
$\cZ^{f}_t = \cZ^{\chi_f}_t + m^{f}_t$ where $\cZ^{\chi_f}_t$ is centered
and
\begin{equation*}
e^{-\alpha t} \Var[\chi_f](t) = e^{-\alpha t} \E[\chi_f^2](t)
\end{equation*}
is directly Riemann integrable.
Using $\Var[\cZ^{f}_t] = \Var[\cZ^{\chi_f}_t + m^{f}_t] = \Var[\cZ^{\chi_f}_t] = \E[(\cZ^{\chi_f}_t)^2]$
we infer with the help of \eqref{eq:Var[Z_t^chi]} that
\begin{equation*}
\E[(\cZ^{\chi_f}_t)^2] = \E[\cZ^{\chi_f^2}_t]	\quad	\text{for all } t \in \R.
\end{equation*}
\end{proof}

\subsection{Slowly growing mean process with signed characteristics}	\label{subsec:slowly growing mean}

We now treat the case where $m_t^\varphi$ grows relatively slowly as $|t| \to \infty$.
Later, we shall reduce the general case to this one.

\begin{theorem}	\label{Thm:slowly growing mean}
Suppose that (A\ref{ass:Malthusian parameter})  through (A\ref{ass:second moment}) hold,
 the random characteristic $\varphi$ satisfies  (A\ref{ass:mean growth}),
(A\ref{ass:variance growth}) and that
the function $t\mapsto e^{-\frac\alpha2 t} (1+t^2) m_t^\varphi$ is bounded.
If $\G=\R$, assume in addition that (A\ref{ass:local ui of phi^2}) holds.
Then, with $\cN$ a standard normal random variable independent of  $\F$,
\begin{equation*}
e^{-\frac\alpha2 t} \cZ^\varphi_t \stablyto \sigma_\varphi\sqrt{\tfrac{W}{\beta}}\cN,
\end{equation*}
with
\begin{equation*}
\sigma^2_\varphi \defeq \int \Var[\varphi(x)+\xi*m^{\varphi}(x)] e^{-\alpha x} \, \ell(\dx).
\end{equation*}
Moreover, $\sigma^2_\varphi=0$ if and only
if the mean $m^\varphi$ is a version of the process $\cZ^\varphi$.

\end{theorem}
\begin{proof}
Clearly, $m^{\E[\varphi]}=m^\varphi$ where $\E[\varphi]$ denotes the function $t \mapsto \E[\varphi(t)]$.
In view of Lemma \ref{Lem:reduction to centered phi}, we can write
\begin{align}
\cZ_t^\varphi
&= \cZ^{\varphi-\E[\varphi]}_t + \cZ^{\E[\varphi]}_t
= \cZ^{\varphi-\E[\varphi]}_t+\cZ^{\chi_{\E[\varphi]}}_t+m^\varphi_t
= \cZ^{\varphi-\E[\varphi]+\chi_{\E[\varphi]}}_t+m^\varphi_t		\notag	\\
& = \cZ_t^{\varphi + m^\varphi*\xi - \E[\varphi + m^\varphi*\xi]}+m^\varphi_t	\quad	\text{a.\,s.}
\label{eq:Z_t^varphi centering trick}
\end{align}
By assumption, $e^{-\alpha t/2} m^\varphi_t \to 0$ as $t \to \infty$.
Hence, it suffices to show that $e^{-\frac\alpha2 t} \cZ^{\varphi-\E[\varphi]+\chi_{\E[\varphi]}}_t$
converges in distribution to the claimed distribution.
Since the characteristic
\begin{equation}	\label{eq:reduced characteristic}
t \mapsto \varphi(t)-\E[\varphi(t)]+\chi_{\E[\varphi]}(t)
\end{equation}
is centered, it is reasonable to apply Theorem \ref{Thm:CLT with centered phi}(i).
To this end, we need to check that (A\ref{ass:variance growth})
holds for the characteristic in \eqref{eq:reduced characteristic},
i.e., that the function
\begin{align}
t &\mapsto e^{-\alpha t} \Var[\varphi(t)-\E[\varphi(t)]+\chi_{\E[\varphi]}(t)]		\notag	\\
&= e^{-\alpha t} \Var[\varphi(t)+\chi_{\E[\varphi]}(t)]	\text{ is directly Riemann integrable.}		\label{eq:Var phi+Ephi^* dRi}
\end{align}
From Lemma \ref{Lem:reduction to centered phi}
we conclude that $\chi_{\E[\varphi]}(t)$ satisfies (A\ref{ass:variance growth}) and (A\ref{ass:local ui of phi^2}).
This is also true for $\varphi$. Hence, \eqref{eq:Var phi+Ephi^* dRi} holds by Remark \ref{Rem:dRi}.

Finally, if $\sigma_\varphi = 0$, then $\varphi(x)+m^{\varphi}*\xi(x)$
is equal to its expectation a.\,s.\ for $\ell$-almost every $x \in \G$,
i.e., $\varphi(x)+m^{\varphi}*\xi(x) - (\E[\varphi](x) + m^{\varphi}*\mu(x)) = 0$ a.\,s.\ for $\ell$-almost every $x \in \G$.
On the other hand, Lemma \ref{Lem:reduction to centered phi}(a) implies that
$\chi_{\E[\varphi]}=m^{\varphi}*\xi - m^{\varphi}*\mu$  has c\`adl\`ag paths a.\,s.\ and by Remark 	\ref{Rem:dRi} the same holds true for the characteristic $\E[\varphi]$, which in turn implies that, except on a $\Prob$-null set, $\varphi(x)+m^{\varphi}*\xi(x) - (\E[\varphi](x) + m^{\varphi}*\mu(x)) = 0$ for every $x \in \G$.
Consequently, by \eqref{eq:Z_t^varphi centering trick}, for every fixed  $t \in \G$,
\begin{align*}
\cZ_t^\varphi
= \cZ_t^{\varphi + m^\varphi*\xi - \E[\varphi + m^\varphi*\xi]}+m^\varphi_t = m^\varphi_t
\quad	\text{a.\,s.},
\end{align*}
i.e., for every fixed $t \in \G$, $\cZ_t^\varphi$ is a.\,s.\ deterministic.
\end{proof}

\subsection{Proof of Theorem \ref{Thm:main}}	\label{subsec:Proof of main theorem}

In the proof of Theorem \ref{Thm:main}, we use the following fact.

\begin{lemma}	\label{lem:auxilarity lemma for the bounary}
	Let $\eta_1,\ldots,\eta_m$ be distinct real numbers. In the lattice case, we additionally assume that $\eta_i\in(-\pi,\pi]$. Consider a collection of centered, square-integrable random variables $(Y_{j,l})_{1 \leq j \leq m, 0 \leq l \leq n}$.
Then for
\begin{equation*}
\chi(t)
\defeq \1_{[0,\infty)}(t)e^{\frac\alpha2 t} \sum_{l=0}^n\sum_{j=1}^m t^{l} e^{\imag \eta_j t} Y_{j,l}
\end{equation*}
it holds that
\begin{equation}	\label{eq:convergence of variance}
\frac{1}{t^{2n+1}} \int_{[0,t]}
\Var\big[\chi(x) \big] e^{-\alpha x}\, \ell(\dx) \to \frac{1}{2n+1} \sum_{j=1}^m \Var[Y_{j,n}]
\quad \text {as } t\to\infty, \, t \in \G
\end{equation}
and, for any $\varepsilon>0$,
\begin{align}	\label{eq:Lindeberg's condition}
\E\big[|\chi(t)|^2 \ind{|\chi(t)|^2> \varepsilon t^{2n+1} e^{\alpha t} } \big]
=o(t^{2n} e^{\alpha t})	\quad \text{as }	t\to\infty, \, t \in \G.
\end{align}
In other words, $\chi(t)$ fulfills the assumption of Theorem~\ref{Thm:CLT with centered phi}(ii)
with $\theta = 2n$.
\end{lemma}
\begin{proof}
Expanding the variance gives
\begin{align}	\label{eq:expanding the variance}	\nonumber
\Var\bigg[\sum_{l=0}^n\sum_{j=1}^m x^{l} e^{\imag \eta_j x} Y_{j,l} \bigg]
&= x^{2n}\sum_{j=1}^m \Var[Y_{j,n}]\\
&\hphantom{=} ~+ x^{2n}\sum_{j \not = k} e^{\imag (\eta_j - \eta_k) x} \Cov[Y_{j,n},Y_{k,n}]+O(x^{2n-1})
\end{align}
as $x \to \infty$.
Further notice that, for $\eta \in \R$, with $|\eta|<2\pi$ in the lattice case,
\begin{align}	\label{eq:trigonometric power integral}
\frac{1}{t^{{2n}+1}} \int_{[0,t]} x^{2n} e^{\imag \eta x} \, \ell(\dx)
\to	
\begin{cases}
\frac1{2n+1},	&	\text{if } \eta = 0,	\\
0,			&	\text{if } \eta \not = 0.
\end{cases}
\end{align}
This follows from the fundamental theorem of calculus in the non-lattice case and integration by parts
if $\eta\not =0$,
whereas in the lattice case, it follows from Faulhaber's formula if $\eta=0$
and from summation by parts if $\eta \not = 0$.
Relation \eqref{eq:convergence of variance} now follows from
\eqref{eq:expanding the variance} and \eqref{eq:trigonometric power integral}.

In order to prove that \eqref{eq:Lindeberg's condition} holds, for $t\ge0$ we set
$\chi_j(t)\defeq e^{\frac\alpha2 t} \sum_{l=0}^nt^{l} e^{\imag \eta_j t} Y_{ j,l}$.
Since for any complex numbers $c_1,\dots,c_m$ and $y > 0$
\begin{equation*}
|c_1+\cdots+c_m|^2 \ind{|c_1+\cdots+c_m|>y} \leq m^2\big(|c_1|^2\ind{|c_1|>y/m}+\cdots+|c_m|^2\ind{|c_m|>y/m}\big),
\end{equation*}
it suffices to prove  \eqref{eq:Lindeberg's condition} for $\chi_j$ instead of $\chi$.
By Markov's inequality, we have
\begin{equation*}
\Prob(|\chi_j(t)|^2>\varepsilon t^{2n+1} e^{\alpha t})
\leq \frac{(n+1)^2}{\varepsilon t^{2n+1}} \sum_{l=0}^nt^{2l}  \E[|Y_{j,l}|^2] \to 0
\quad	\text{as } t \to \infty
\end{equation*}
as the sum is of the order $t^{2n}$ as $t \to \infty$.
Consequently,
\begin{align*}
&\E\big[|\chi_j(t)|^2 \ind{|\chi_j(t)|^2>\varepsilon t^{2n+1} e^{\alpha t}}\big]	\\
&~\leq 2t^{2n} e^{\alpha t} \E\big[|Y_{j,n}|^2\ind{|\chi_j(t)|^2>\varepsilon t^{2n+1} e^{\alpha t}}\big] + O(t^{2n-1} e^{\alpha t})
= o(t^{2n} e^{\alpha t})
\end{align*}
as $t \to \infty$. This proves \eqref{eq:Lindeberg's condition}.		
\end{proof}

We now turn to the proof of Theorem \ref{Thm:main}.

\begin{proof}[Proof of Theorem \ref{Thm:main}]
Suppose that $\varphi$ is a random characteristic satisfying
\begin{align*}	\tag{\ref{eq:expansion of mean}}
m^{\varphi}_t = \1_{[0,\infty)}(t) \sum_{\lambda\in\Lambda_{\geq}} \sum_{l=0}^{k(\lambda)-1} a_{\lambda,l} t^l e^{\lambda t} + r(t),	\quad t\in\G
\end{align*}
for some constants $a_{\lambda,l} \in \R$ and a function $r$ such that
$|r(t)| \leq C e^{\frac\alpha2 t}/(1+t^2)$ for a finite constant $C > 0$.
For any $\lambda\in\Lambda_{\geq}$,
we put
\begin{equation*}
\vec{a}_\lambda \defeq \sum_{l=1}^{k(\lambda)} a_{\lambda,l-1} \e_{l}
\end{equation*}
and consider the following characteristic
\begin{align*}
\psi_\Lambda(t)=\sum_{\lambda\in \Lambda} \vec{a}_\lambda^\transp(\phi_{\lambda}(t)+\chi_\lambda(t)) \e_1
\end{align*}
for $\phi_{\lambda}$ and $\chi_\lambda$ defined in \eqref{eq:phi_lambda} and \eqref{eq:chi_lambda}, respectively.
Then, by  Lemma \ref{Lem:chi_lambda}, for $t \geq 0$,
\begin{align}	\label{eq:Z_t^psi_Lambda}
\cZ_t^{\psi_\Lambda}
&=\sum_{\lambda\in \Lambda} \vec{a}_\lambda^\transp\cZ_t^{\phi_{\lambda}+\chi_\lambda} \e_1
= \sum_{\lambda\in \Lambda} \vec{a}_\lambda^\transp\exp(\lambda,t,k(\lambda)) W(\lambda,k(\lambda)) \e_1
= H_\Lambda(t),
\end{align}
 where the definition of $H_\Lambda$ should be recalled from \eqref{eq:H_Lambda}.
Further, by Lemma \ref{Lem:Nerman's martingales}
\begin{align}
m_t^{\psi_\Lambda}
&=\sum_{\lambda\in \Lambda}\vec{a}_\lambda^\transp
\E[\cZ_t^{\phi_{\lambda}+\chi_\lambda}] \e_1	
=\sum_{\lambda\in \Lambda}\vec{a}_\lambda^\transp\E[\cZ_t^{\phi_{\lambda}}] \e_1\notag	\\
&= \1_{[0,\infty)}(t)\sum_{\lambda\in \Lambda}\vec{a}_\lambda^\transp
\exp(\lambda,t,k(\lambda))  \e_1=\1_{[0,\infty)}(t)\E[H_\Lambda(t)].	\label{eq:m^psi_Lambda}
\end{align}
Similarly, putting
\begin{align*}
\psi_{\partial \Lambda}(t)
&=\sum_{\lambda\in\partial\Lambda}\vec{a}_\lambda^\transp\E[\phi_{\lambda}(t)] \e_1,
\end{align*}
again by Lemma \ref{Lem:Nerman's martingales}, we obtain,
for any $t \in \R$,
\begin{align}
m^{\psi_{\partial \Lambda}}_t
&= \E[\cZ_t^{\psi_{\partial \Lambda}}]
= \sum_{\lambda\in\partial\Lambda} \vec{a}_\lambda^\transp\E[\cZ_t^{\phi_{\lambda}}] \e_1	\notag	\\
&=
\1_{[0,\infty)}(t) \sum_{\lambda\in\partial\Lambda} \vec{a}_\lambda^\transp\exp{(\lambda,t,k(\lambda))} \e_1
= \1_{[0,\infty)}(t) H_{\partial \Lambda}(t),	\label{eq:m^psi_dLambda}
\end{align}
 where the definition of $H_{\partial\Lambda}$ should be recalled from \eqref{eq:H_dLambda}.
 In view of Remark \ref{Rem:conjugate}, we have $\overline{\vec{a}_\lambda}=\vec{a}_{\overline\lambda}$,
$\overline{\phi_{\lambda}}=\phi_{\overline\lambda}$ and $ \overline{\chi_{\lambda}}=\chi_{\overline\lambda}$.
We thus conclude that both characteristics $\psi_\Lambda$ and $\psi_{\partial\Lambda}$ are in fact real-valued.

Now write
\begin{align}
	\label{eq:first decomposition of Z^phi}
\cZ_t^\varphi=\cZ_t^{\psi_\Lambda}+\cZ_t^{\psi_{\partial \Lambda}}+\cZ_t^{\varrho}
= H_\Lambda(t)+\cZ_t^{\psi_{\partial \Lambda}}+\cZ_t^{\varrho},
\end{align}
where $\varrho \defeq \varphi - \psi_\Lambda - \psi_{\partial \Lambda}$.
Next, since $\psi_{\partial \Lambda}$ is deterministic, we may consider
the associated centered characteristic $\chi_{\psi_{\partial \Lambda}}$
defined in \eqref{eq:chi_f}, namely,
\begin{align*}
\chi_{\psi_{\partial \Lambda}}(t)
&= \xi * m^{\psi_{\partial \Lambda}}_t - \mu * m^{\psi_{\partial \Lambda}}_t	\\
&= \sum_{\lambda\in\partial\Lambda} \sum_{j=1}^N  \vec{a}_\lambda^\transp\exp{(\lambda,t-X_j,k(\lambda))} \e_1
\1_{[0,\infty)}(t-X_j)-\mu*m^{\psi_{\partial \Lambda}}_t.
\end{align*}
 If we set now
\begin{align*}
\psi_{\lambda}(t)
&\defeq \1_{[0,\infty)}(t) \sum_{j=1}^N  \vec{a}_\lambda^\transp\exp(\lambda,t-X_j,k(\lambda)) \e_1	\\
&\hphantom{\defeq}~
-\1_{[0,\infty)}(t) \E\bigg[\sum_{j=1}^N  \vec{a}_\lambda^\transp\exp(\lambda,t-X_j,k(\lambda)) \e_1 \bigg],\\
	 \chi(t)& \defeq	\sum_{\lambda\in\partial\Lambda}\psi_{\lambda}(t)
\end{align*}
and
\begin{align*}
\phi_{\partial\Lambda}(t)
&\defeq
\sum_{\lambda\in\partial\Lambda}
\sum_{j=1}^N  \vec{a}_\lambda^\transp\exp(\lambda,t-X_j,k(\lambda)) \e_1 \1_{[0,X_j)}(t)	\\
&\hphantom{\defeq}~
-\sum_{\lambda\in\partial\Lambda} \E\bigg[\sum_{j=1}^N  \vec{a}_\lambda^\transp\exp(\lambda,t-X_j,k(\lambda))
\e_1 \1_{[0,X_j)}(t) \bigg]	\\
&=\sum_{\lambda\in\partial\Lambda} \vec{a}_\lambda^\transp \phi_{\lambda}(t) \e_1-  \sum_{\lambda\in\partial\Lambda} \E\big[  \vec{a}_\lambda^\transp \phi_{\lambda}(t)\e_1 \big].
\end{align*}
  we  get the following decomposition
\begin{align}
	\label{eq:decomposition of psi_partial_Lambda}
\chi_{\psi_{\partial \Lambda}}(t)=\chi(t)-\phi_{\partial\Lambda}(t).
\end{align}
The fact that all the expectations above are finite and thus the characteristics are well-defined follows from (A\ref{ass:first moment}) and (A\ref{ass:Lambda finite}).
Note also that, for every $\lambda\in\partial\Lambda$,
\begin{align*}
\sum_{i=1}^N  \vec{a}_\lambda^\transp\exp(\lambda,t-X_i,k(\lambda)) \e_1
&= \sum_{i=1}^N e^{\lambda(t-X_i)} \sum_{l=0}^{k(\lambda)-1}a_{\lambda,l}(t-X_i)^l	\\
&=
\sum_{i=1}^N e^{\lambda(t-X_i)} \sum_{l=0}^{k(\lambda)-1} a_{\lambda,l} \sum_{j=0}^l \binom{l}{j} t^j (-X_i)^{l-j}	\\
&=
\sum_{j=0}^{k(\lambda)-1} \sum_{l=j}^{k(\lambda)-1}
\sum_{i=1}^N e^{-\lambda X_i} a_{\lambda,l} \binom{l}{j}(-X_i)^{l-j} t^j e^{\lambda t}	\\
&= \sum_{j=0}^{k(\lambda)-1} R_{\lambda,j} t^j e^{\lambda t}
\end{align*}
 by \eqref{eq:definition of R_lambda}.
With this at hand, we infer
\begin{align}	\label{eq:psi_lambda}
\psi_{\lambda}(t) = \1_{[0,\infty)}(t) \sum_{l=0}^{k(\lambda)-1} (R_{\lambda,l}-\E[R_{\lambda,l}])t^le^{\lambda t},
\quad	t \in \R.
\end{align}
Using  Lemma \ref{Lem:reduction to centered phi}(d) with $f=\psi_{\partial \Lambda}$  (note here that such $f$ fulfills the assumptions as for any $\lambda\in\partial\Lambda$ the characteristic $\phi_\lambda$ vanishes on $(-\infty,0)$ and satisfies (A\ref{ass:mean growth}) by Lemma \ref{Lem:phi^1 and phi^2}), we get that
$\cZ_t^{\psi_{\partial \Lambda}}-m_t^{\psi_{\partial \Lambda}} = \cZ_t^{\chi_{\psi_{\partial \Lambda}}}$
and, hence, $\cZ_t^{\psi_{\partial \Lambda}}= \cZ_t^{\chi_{\psi_{\partial \Lambda}}} + m_t^{\psi_{\partial \Lambda}}
= \cZ_t^{\chi_{\psi_{\partial \Lambda}}} + H_{\partial \Lambda}(t)$  for $t \geq 0$.
Therefore, from \eqref{eq:first decomposition of Z^phi} and \eqref{eq:decomposition of psi_partial_Lambda} we obtain the following decomposition,
\begin{align}	\label{eq:final decomposition Z_t^varphi}
\cZ_t^\varphi = H_\Lambda(t) + H_{\partial \Lambda}(t) + \cZ_t^{\varrho-\phi_{\partial\Lambda}} + \cZ_t^\chi,
\quad	t \geq 0.
\end{align}
It suffices to prove
the limit theorem for $\cZ_t^{\varrho-\phi_{\partial\Lambda}}$ and $\cZ_t^\chi$.
To this end, we invoke Theorem~\ref{Thm:slowly growing mean} for the first process
and Theorem \ref{Thm:CLT with centered phi}(ii) for the second  as both characteristics are real-valued
as a consequence of Remark \ref{Rem:conjugate}.
We begin with $\cZ_t^{\varrho-\phi_{\partial\Lambda}}$ and first notice that,
in view of Lemma \ref{Lem:phi^1 and phi^2} and Remark \ref{Rem:dRi},
the characteristic $\varrho-\phi_{\partial\Lambda}$ has c\`adl\`ag paths
and satisfies (A\ref{ass:mean growth}), (A\ref{ass:variance growth}) and (A\ref{ass:local ui of phi^2}).
Moreover,  using the fact that $\phi_{\partial \Lambda}$ is centered in combination with  \eqref{eq:m^psi_Lambda}, \eqref{eq:m^psi_dLambda} and
 \eqref{eq:expansion of mean}
we infer
\begin{align}	\label{eq:m^rho-phi_dLambda_t=r(t)}
m^{\varrho-\phi_{\partial\Lambda}}_t
&=m^{\varrho}_t
= m^{\varphi - \psi_\Lambda - \psi_{\partial \Lambda}}_t
= m_t^\varphi - m_t^{\psi_\Lambda} - m_t^{\psi_{\partial \Lambda}} \\
\nonumber
& = m_t^\varphi -\1_{[0,\infty)}(t)\big(\E \big[H_\Lambda(t)\big] +H_{\partial \Lambda}(t)\big)
=r(t).
\end{align}
We may thus apply Theorem~\ref{Thm:slowly growing mean} to conclude that
\begin{equation}	\label{eq:convergence of Z^rho-dLambda}
e^{-\frac\alpha2t} \cZ^{\varrho-\phi_{\partial\Lambda}}_t
\stablyto \sigma\sqrt{\tfrac{W}{\beta}}\cN,
\end{equation}
where
\begin{align}	\label{eq:variance}
\sigma^2 &\defeq
\int v(x) e^{-\alpha x} \, \ell(\dx),
\end{align}
and $v(t) \defeq \Var\big[\varrho(t)-\phi_{\partial\Lambda}(t)+r*\xi(t)\big]$.
The function $v$ can be further simplified in the following way
\begin{align*}
v(t)
&= \Var\big[\varphi(t) - \psi_\Lambda(t) - \psi_{\partial \Lambda}(t)-\phi_{\partial\Lambda}(t)+r*\xi(t)\big]\\
&=\Var\bigg[\varphi(t)-\sum_{\lambda\in \Lambda} \vec{a}_\lambda^\transp(\phi_{\lambda}(t)+\chi_\lambda(t)) \e_1-\sum_{\lambda\in \partial\Lambda} \vec{a}_\lambda^\transp\phi_{\lambda}(t)\e_1+r*\xi(t)\bigg]\\
&=\Var\bigg[\varphi(t) +\sum_{j=1}^N\bigg( -\sum_{\lambda\in \Lambda} \vec{a}_\lambda^\transp \1_{(-\infty,X_j)}(t) \exp({\lambda,t-X_j},k(\lambda)) \e_1\\
&\hphantom{=\Var\bigg[} -\sum_{\lambda\in \partial\Lambda} \vec{a}_\lambda \1_{[0,X_j)}(t) \exp({\lambda,t-X_j},k(\lambda)) \e_1+r(t-X_j)\bigg)\bigg]\\
&=\Var\bigg[\varphi(t) +\sum_{j=1}^N\bigg( -\sum_{\lambda\in \Lambda}\vec{a}_\lambda^\transp  \exp({\lambda,t-X_j},k(\lambda)) \e_1\\
&\hphantom{=\Var\bigg[} -\1_{[0,\infty)}(t)\sum_{\lambda\in \partial\Lambda} \vec{a}_\lambda^\transp  \exp({\lambda,t-X_j},k(\lambda)) \e_1+m^\varphi(t-X_j)\bigg)\bigg].
\end{align*}
This proves the theorem under the assumption that
$\rho_l=0$ for all $l \geq 0$. Indeed, in this case
$\cZ_t^\chi=0$  and $ \sum_{j=1}^N  \vec{a}_\lambda^\transp  \exp({\lambda, t-X_j},k(\lambda)) \e_1$ is  a.\,s.\ constant for any $\lambda\in \partial\Lambda, t\in\G$ and
\eqref{eq:sigma} follows.

Now, combining \eqref{eq:final decomposition Z_t^varphi}
and \eqref{eq:convergence of Z^rho-dLambda}, we arrive at \eqref{eq:limit theorem} if $\sigma^2 > 0$.
However, if $\sigma^2=0$, then for all $t \in \G$,
$\cZ_t^{\varrho-\phi_{\partial\Lambda}}$ equals its expectation,
which is $r(t)$ a.\,s., as shown by \eqref{eq:m^rho-phi_dLambda_t=r(t)}. This establishes (i).

It remains to prove the theorem in the case where $\rho_l>0$ for some $l \geq 0$.
First notice that by \eqref{eq:convergence of Z^rho-dLambda},
the already established central limit theorem for $\cZ_t^{\varrho-\phi_{\partial\Lambda}}$,
we have $\cZ_t^{\varrho-\phi_{\partial\Lambda}}=o(t^{\frac12}e^{\frac\alpha2 t})$
as $t \to \infty$ in probability. Let $n \in \N_0$ be maximal with $\rho_n > 0$.
We show that the characteristic $\chi$ satisfies the assumptions of Theorem \ref{Thm:CLT with centered phi}(ii)
with $\theta=2n$.
Observe that, for any $\lambda \in \partial \Lambda$, $l \leq k(\lambda)-1$ and some constant $C_{\lambda,\vec{a}_\lambda}$ depending on $\lambda,l$ and $\vec{a}_\lambda$,
\begin{align*}
|R_{\lambda,l}| \leq C_{\lambda,\vec{a}_\lambda} \sum_{j=1}^N (1+X_j^{k(\lambda)-1}) e^{-\frac\alpha2 X_j}.
\end{align*}
In view of assumption (A\ref{ass:second moment}) the random variable $R_{\lambda,l}$ is square integrable.
Setting $R_{\lambda,l}\defeq 0$ for $l \geq k(\lambda)$,
we 
write
\begin{align*}
\chi(t)=\1_{[0,\infty)}(t) e^{\frac\alpha2 t} \cdot \sum_{\lambda\in\partial\Lambda}
\sum_{l=0}^{n} (R_{\lambda,l}-\E[R_{\lambda,l}])t^le^{\imag\,\Imag(\lambda) t}.
\end{align*}
An application on Lemma~\ref{lem:auxilarity lemma for the bounary} gives
\begin{align*}
\frac1{t^{2n+1}}
\int_0^t \E[\chi^2(x)] e^{-\alpha x} \, \ell(\dx)
\to\frac{\sum_{\lambda\in\partial\Lambda}\Var[R_{\lambda,n}]}{2n+1}
= \frac{\rho^2_{n}}{2n+1}
\end{align*}
and
\begin{align*}
\E\big[|\chi(t)|^2 \ind{|\chi(t)|^2> \varepsilon t^{2n+1}  e^{\alpha t}} \big]=o(t^{2n} e^{\alpha t})
\quad	\text{as } t \to \infty,\; t \in \G.
\end{align*}
Finally, by 
Theorem~\ref{Thm:CLT with centered phi}(ii), 
\begin{align*}
\Big(\frac{\rho^2_{n}t^{2n+1}}{2n+1}e^{\alpha t}\Big)^{\!-\frac12}\cZ_t^\chi \stablyto \sqrt{\tfrac{W}{\beta} }\cN,
\end{align*}
which finishes the proof.
\end{proof}

\begin{remark}
The proofs of Theorems \ref{Thm:main} and \ref{Thm:slowly growing mean} reveal that for any characteristic $\varphi$
satisfying the assumptions (A\ref{ass:mean growth}) through (A\ref{ass:local ui of phi^2}), there exists a decomposition  $\varphi=\varphi_1+\varphi_2+\varphi_3$, where each term also satisfies (A\ref{ass:mean growth}) through (A\ref{ass:local ui of phi^2}). Furthermore, for the corresponding Crump-Mode-Jagers processes, it is established that $\cZ^{\varphi_1}_t=H(t)$ for $t\ge0$, $\cZ^{\varphi_2}_t$ is centered, and $\cZ^{\varphi_3}_t$ is a deterministic function equal to $r(t)$. In particular, for the characteristic $\tilde \varphi \defeq\varphi_1+\varphi_3$, one obtains that for $t\ge0$, $\cZ^{\tilde \varphi}_t=H(t)+r(t)$ almost surely, indicating the lack of Gaussian fluctuations as $t$ goes to infinity.
\end{remark}

\begin{proof}[Proof of Corollary \ref{Cor:fd convergence}]
First observe that linear combinations as well as the translations $\varphi(\cdot)\mapsto\varphi(\cdot-s)$
preserve the conditions  (A\ref{ass:mean growth}),  (A\ref{ass:variance growth}) and (A\ref{ass:local ui of phi^2}).
Moreover, for the characteristic $\psi(t) \defeq \varphi(t-s)$
the mean function $m^\psi_t$ has 
expansion
\eqref{eq:expansion of mean} with coefficients given by vectors
$\big({\exp(\lambda,-s,k(\lambda))}^\transp\vec{a}_\lambda\big)$.
According to the Cram\'er--Wold device the  convergence in distribution of
\begin{equation*}
t^{-\frac{d}2}e^{-\frac\alpha2 t}
\big(\cZ^\varphi_{t-s_1} -H(t-s_1),\dots,\cZ^\varphi_{t-s_n} -H(t-s_n)\big)
\end{equation*}
is equivalent to the convergence in distribution of
\begin{equation*}
t^{-\frac{d}2}e^{-\frac\alpha2 t} \sum_{j=1}^n c_j\big(\cZ^\varphi_{t-s_j} -H(t-s_j)\big),
\end{equation*}
for all choices $c_1,\dots,c_n \in \R$,
and the latter convergence follows from Theorem \ref{Thm:main}.
The covariance can be obtained by the polarization identity applied to the variance
and the fact that $m^\psi_{t}=m^\varphi_{t-s}, \ h^\psi(t)=h^\varphi(t-s)$.
\end{proof}

\section{Asymptotic expansion of the mean}	\label{sec:asymptotic expansion of the mean}

In this section we are concerned with
the asymptotic expansion of the mean $m_t^{ \varphi} = \E[\cZ^{ \varphi}_t]$
of a supercritical general branching process $(\cZ_t^{ \varphi})_{t \geq 0}$ as $t \to \infty$.	
Throughout the section, we assume that (A\ref{ass:Malthusian parameter}) and
(A\ref{ass:first moment}) hold.
We fix some notation throughout the section.
By $\theta$ we denote a parameter from $(0,\frac\alpha2)$ such that
 \begin{align}
 	\label{eq:choice of theta}
 	\L\mu(\theta)<\infty \text{ and } \L\mu(z) \neq 1 \quad \text{whenever } \theta \leq \Real(z) < \tfrac\alpha2.
\end{align}
Such a $\theta$ exists in all particular cases considered in this section
and may sometimes be enlarged in order to ensure the validity of additional conditions.
We also fix
 $\gamma > \frac\alpha2$ such that $\gamma < \min\{\Real(\lambda):\lambda\in\Lambda\}$.

It's worth noting that the mean $m^\varphi_t$ depends on the underlying point process $\xi$ only through its intensity measure $\mu$. Therefore, when analyzing $m^\varphi_t$, without loss of generality, we can assume that, additionally to (A\ref{ass:Malthusian parameter}) and
(A\ref{ass:first moment}), the condition (A\ref{ass:second moment}) holds true. Otherwise, we can always replace $\xi$ with the Poisson point process whose intensity measure is $\mu$.

In the non-lattice case we work with
the corresponding  bilateral  Laplace transforms whereas in the lattice case,
we use generating functions.

\subsection{The lattice case}	\label{subsec:the lattice case}

In the present subsection,
we assume that $\mu$ is concentrated on the lattice $\Z$  (and not on a smaller lattice).
We set
\begin{align}	\label{eq:generating_mu}
\Gfct\mu(z) &\defeq \sum_{k=0}^\infty\mu(\{k\}) z^k = \int z^x \, \mu(\dx)
\end{align}
for all $z\in \C$ for which the series is absolutely convergent. In particular, $\L\mu(z) = \Gfct\mu (e^{-z})$.
Note that, due to assumption (A\ref{ass:first moment}), $\Gfct\mu(e^{-\vartheta}) < \infty$ and hence
the power series \eqref{eq:generating_mu}
defines a holomorphic function on $\{|z|<e^{-\vartheta}\}$.
Further, by slightly increasing the value of $\vartheta$ if necessary, we may assume
without loss of generality that there are only finitely many solutions of the equation $\Gfct\mu(z)=1$
in the disc $\{|z|<e^{-\vartheta}\}$.

\begin{lemma}	\label{Lem:lattice asymptotic of the mean}
 Assume that (A\ref{ass:Malthusian parameter}) and (A\ref{ass:first moment}) hold.
Let $\theta\in(\vartheta,\frac\alpha2)$ be such that
there are no solutions to $\Gfct\mu(z)=1$ in $\{z:e^{-\alpha/2}< |z| \leq e^{-\theta }\}$.
Then there are constants  $b_{\lambda,l}$, $\lambda \in \Lambda_{\geq},\,l=0,\ldots,k(\lambda)-1$ such that,
for any characteristic $\varphi$ with
\begin{equation}
	\label{eq:integrability condition in the lattice case}
\sum_{n \in \Z} |\E[\varphi(n)]|(e^{-\theta n}+e^{-\alpha n})<\infty,
\end{equation}
it holds that, for $t \in \Z$,
\begin{align}	\label{eq:m_t^varphi lattice case}
m^{\varphi}_t
&=	\begin{cases}
	\sum \limits_{\lambda\in\Lambda_{\geq}} \sum \limits_{l=0}^{k(\lambda)-1}b_{\lambda,l}
	\sum \limits_{n \in \Z} \E[\varphi(n)] (t-n)^le^{\lambda(t-n)} +O(e^{\theta t})	&	\text{as } t \to \infty	\\
	O(e^{\gamma t})	&	\text{as } t \to -\infty.
	\end{cases}
\end{align}
\end{lemma}

\begin{remark}
Defining $\vec{b}_{\lambda} \defeq (b_{\lambda,l-1})_{l=1,\ldots,k(\lambda)}$,
we may write \eqref{eq:m_t^varphi lattice case} more compactly in the form
\begin{align*}
m_t^\varphi = \sum_{\lambda\in\Lambda_{\geq}}\ind{t\ge 0\text{ or }\lambda\in\Lambda}
\sum_{n \in \Z} \E[\varphi(n)] {\vec{b}_\lambda} ^{ \transp}\exp(\lambda, t-n,k(\lambda)) \e_1+O(e^{\theta t}\wedge e^{\gamma t})
\end{align*}
as $t \to \pm \infty$, $t \in \Z$.
\end{remark}

\begin{proof}[Proof of Lemma \ref{Lem:lattice asymptotic of the mean}]
For $r>0$, let $B_r = \{|z| < r\}$ and $\partial B_r = \{|z|=r\}$.
Now fix $r<e^{-\alpha}$.
As $\Gfct(\mu^{*l})= (\Gfct\mu)^{l}$, for any $l\in\N$
and since $\Gfct\mu$ is holomorphic on $B_{e^{-\vartheta}}$,
we infer from Cauchy's integral formula that
\begin{align*}
\mu^{*l}(\{n\})=\frac1{2\pi\imag}\int \limits_{\partial B_r} \frac{(\Gfct\mu)^l(z)}{z^{n+1}} \, \dz.
\end{align*}
In particular,
\begin{align*}
\E[N(\{n\})]
&= \sum_{l=0}^\infty \mu^{*l}(\{n\})
= \sum_{l=0}^\infty \frac1{2\pi\imag} \int \limits_{\partial B_r} \! \frac{(\Gfct\mu)^l(z)}{z^{n+1}}\dz
= \frac1{2\pi\imag}\int \limits_{\partial B_r} \! \frac{\dz}{(1-\Gfct\mu(z))z^{n+1}}
\end{align*}
where the last equality follows by Fubini's theorem.
For $\lambda\in\Lambda_{\geq}$, let
\begin{equation*}
\sum_{j=-k(\lambda)}^{-1}{b_j(\lambda)}{(z-e^{-\lambda})^{j}}
\end{equation*}
be the principle part of the Laurent expansion of the meromorphic function $(1-\Gfct\mu(z))^{-1}$ around $e^{-\lambda}$.
Then the function
\begin{align*}
H(z) \defeq \frac{1}{1-\Gfct\mu(z)}-\sum_{\lambda\in \Lambda_{\geq}} \sum_{j=-k(\lambda)}^{-1}{b_j(\lambda)}{(z-e^{-\lambda})^{j}}
\end{align*}
is holomorphic on $B_{e^{-\theta}}$.
On the other hand, for any $d \in \N_0$,
\begin{align*}
\frac1{2\pi\imag}\int \limits_{\partial B_r} \frac{(z-e^{-\lambda})^{-d}}{z^{n+1}} \, \dz
= (-e^{\lambda})^d e^{\lambda n}{n+d-1 \choose d-1}
\end{align*}
by the residue theorem. Therefore,
\begin{align*}
G(z) \defeq \sum_{\lambda\in \Lambda_{\geq}} \sum_{j=-k(\lambda)}^{-1}{b_j(\lambda)}{(z-e^{-\lambda})^{j}}
\end{align*}
satisfies
\begin{align*}
\frac1{2\pi\imag} \int \limits_{\partial B_r} \frac{G(z)}{z^{n+1}} \, \dz &= \sum_{\lambda\in\Lambda_{\geq}}p_\lambda(n)e^{\lambda n}
\end{align*}
where $p_\lambda$, for $\lambda \in \Lambda_{\geq}$, is a polynomial with complex coefficients of degree $k(\lambda)-1$.
From the analyticity of $H$, we infer
\begin{equation*}
\bigg|\int \limits_{\partial B_{e^{-\theta}}} \frac {H(z)}{z^{n+1}} \, \dz\bigg| = O(e^{\theta n})	\quad	\text{as } n \to \infty,
\end{equation*}
which in turn gives
\begin{align*}
\E [N(\{n\})]
&= \frac1{2\pi\imag} \int \limits_{\partial B_r} \frac{G(z) + H(z)}{z^{n+1}} \, \dz	\\
&= \sum_{\lambda\in \Lambda_{\geq}} \vec{b}_\lambda^{\;\transp} \exp(\lambda,n,k(\lambda)) \e_1 + O(e^{\theta n})
\quad	\text{as } n \to \infty
\end{align*}
for some $\vec{b}_\lambda = \sum_{l=1}^{k(\lambda)}b_{\lambda,l-1} \e_{l} \in \R^{k(\lambda)}$.
In other words, there exists a constant $C > 0$ such that, for any $n \in \Z$,
\begin{align}	\label{eq:aux1}
\bigg|\E [N(\{n\})] - \sum_{\lambda\in\Lambda_{\geq}}  \1_{\{n \geq 0 \text{ or } \lambda \in \Lambda\}} \vec{b}_\lambda^{\;\transp}\exp(\lambda,n,k(\lambda))
\e_1\bigg|
\leq C(e^{\theta n}\wedge e^{\gamma n}).
\end{align}
Now we are ready to investigate
the asymptotic behavior of $m_t^\varphi$ as $t \to \pm \infty$, $t \in \Z$.
Since $m_t^\varphi = m_t^{\E[\varphi]}$, we assume without loss of generality
that $\varphi = f$ is a deterministic function satisfying
\begin{equation*}
\sum_{n \in \Z} |f(n)|(e^{-\theta n}+e^{-\alpha n})<\infty,
\end{equation*}
Then, for $t \in \Z$, we have $m_t^f=\sum_{n \in \Z} f(n)\E [N(\{t-n\})]$.
We  write
\begin{align}
&\bigg|m^f_t-\sum_{\lambda\in \Lambda_{\geq}} \ind{t \ge 0\text{ or }\lambda\in\Lambda} \sum_{n \in \Z} f(n)\vec{b}_\lambda^{\;\transp}\exp(\lambda,t-n,k(\lambda)) \e_1\bigg|	\notag	\\
&\leq \sum_{n \in \Z}|f(n)|\bigg|\E[N(\{t-n\})] - \sum_{\lambda\in \Lambda_{\geq}}\ind{t\ge n\text{ or }\lambda\in \Lambda }\vec{b}_\lambda^{\;\transp}\exp(\lambda,t-n,k(\lambda)) \e_1\bigg|	\notag	\\
&\hphantom{\leq}+\sum_{n \in \Z}|f(n)|\bigg| \sum_{\lambda\in \partial\Lambda}|\ind{t\ge 0}-\ind{t \geq n }| \vec{b}_\lambda^{\;\transp}\exp(\lambda,t-n,k(\lambda)) \e_1\bigg|.	\label{eq:m_t^f-P(t) decomposed}
\end{align}
We use \eqref{eq:aux1} to estimate the first sum on the right-hand side of \eqref{eq:m_t^f-P(t) decomposed} by
\begin{align*}
C \sum_{n \in \Z} |f(n)|(e^{\theta (t-n)}\wedge e^{\gamma (t-n)})
&\leq C \bigg(\sum_{n \in \Z} |f(n)| e^{\theta (t-n)}\bigg) \wedge \bigg(\sum_{n \in \Z} |f(n)| e^{\gamma (t-n)}\bigg)\\
&\leq (e^{\theta t}\wedge e^{\gamma t}) \, C \sum_{n \in \Z} |f(n)|(e^{- \theta n} + e^{-\alpha n}).
\end{align*}
On the other hand,
we use \eqref{eq:matrix norm} to conclude that for any $0<\epsilon<\alpha/2-\theta$ and $\lambda\in\partial\Lambda$
there is a constant $C_\epsilon \geq 0$ such that
$\|\exp(\lambda,n,k(\lambda))\| \leq C_\epsilon e^{\frac\alpha2 n+\epsilon|n|}$.
Hence the second sum on the right-hand side of \eqref{eq:m_t^f-P(t) decomposed} is bounded by
\begin{align*}
C_\epsilon \sum_{\lambda\in \partial \Lambda} |\vec{b}_\lambda|
\sum_{n \in \Z} |f(n)| |\ind{t \geq 0} - \ind{t \geq n} |e^{\frac\alpha2(t-n)+\epsilon |t-n|}.
\end{align*}
The latter sum can be estimated as follows
\begin{align*}
&\sum_{n \in \Z} |f(n)| |\ind{t\ge 0}-\ind{t\ge n}|e^{\frac\alpha2(t-n)+\epsilon |t-n|}\\
&=\1_{\N_0}(t) \sum_{n > t} |f(n)| e^{(\frac\alpha2-\epsilon )(t-n)}
+\1_{\Z \setminus \N_0}(t) \sum_{n \leq t} |f(n)| e^{(\frac\alpha2+\epsilon )(t-n)}\\
&\leq \1_{\N_0}(t) \sum_{n \in \Z} |f(n)| e^{\theta(t-n)}
+ \1_{\Z \setminus \N_0}(t) \sum_{n \leq t} |f(n)| e^{\alpha(t-n)}
= O(e^{\theta t}\wedge e^{\alpha t}),
\end{align*}
as $t \to \pm \infty$.
\end{proof}

\subsection{The non-lattice case} 	\label{subsec:the non-lattice case}

We again work under the conditions
(A\ref{ass:Malthusian parameter}) and (A\ref{ass:first moment})
as in Section \ref{subsec:the lattice case},
but now we assume that $\mu$ is non-lattice.

Similar to the lattice case, first we study the behavior of $\E[N(t)]$.
This was already done in \cite[Theorem 3.1]{Janson+Neininger:2008} in the special case
where $\L\mu(z)-1$ has only simple roots.
However, the proof given in the cited source can be adapted to the more general setting here.
In order to make this paper self-contained and for the reader's convenience, we include the proof.

\begin{lemma}	\label{Lem:asymptotics of E[N(t)] non-lattice}	
Suppose that, besides (A\ref{ass:Malthusian parameter}) and (A\ref{ass:first moment}),
the following condition holds:
\begin{equation}	\label{eq:Lmu becomes small}
\limsup_{\eta \to \infty} |\L\mu(\tfrac\alpha2-\delta+\imag \eta)| < 1
\end{equation}
for some $\delta \in (0,\tfrac\alpha2-\vartheta]$.
Then $\Lambda_\geq$ is finite. In fact, the function $\L\mu$ takes the value $1$
only at finitely many points in the strip $\tfrac\alpha2-\delta < \Real(z) < \alpha$.
Then, for any root $\lambda\in\Lambda_{\geq}$ of multiplicity $k(\lambda) \in \N$,
there exist constants $c_{\lambda,l}$, $l=0,\ldots,k(\lambda)-1$ such that,
for any $\theta \in(\tfrac\alpha2-\delta,\frac\alpha2)$  satisfying \eqref{eq:choice of theta} it holds that
\begin{align}	\label{eq:asymptotics of E[N(t)] non-lattice}
\E[N(t)]
=\sum_{\lambda\in\Lambda_{\geq}}e^{\lambda t}\sum_{l=0}^{ k(\lambda)-1}c_{\lambda,l}{t^l}+O(e^{\theta t})
\quad	\text{as }	t \to \infty.
\end{align}
\end{lemma}

\begin{remark}	\label{Rem:asymptotics of E[N(t)] non-lattice}
Note that \eqref{eq:asymptotics of E[N(t)] non-lattice} can be rewritten in the form
\begin{equation*}
\E[N(t)] = \sum_{\lambda \in \Lambda_\geq}\vec{c}_\lambda^{\ \transp}\exp(\lambda, t,k(\lambda)) \e_1+O(e^{\theta t})
\end{equation*}
as $t \to \infty$ with $\vec{c}_\lambda \defeq \sum_{l=1}^{k(\lambda)} c_{\lambda,l-1} \e_{l}$.
\end{remark}

\begin{remark}	\label{Rem:Riemann-Lebesgue}
Suppose that (A\ref{ass:Malthusian parameter}) and (A\ref{ass:first moment}) hold
 and that the intensity measure $\mu$ has a density with respect to the Lebesgue measure.
Then one can check using the Riemann-Lebesgue lemma that \eqref{eq:Lmu becomes small} holds  for any $\delta \in (0,\tfrac\alpha2-\vartheta]$.
Hence, in this case, Lemma \ref{Lem:asymptotics of E[N(t)] non-lattice} applies.
\end{remark}

\begin{proof}[Proof of Lemma \ref{Lem:asymptotics of E[N(t)] non-lattice}]
First, condition \eqref{eq:Lmu becomes small} implies that
\begin{equation*}	
\limsup_{\eta \to \infty} |\L\mu(\theta+\imag \eta)| < 1
\end{equation*}
for all $\theta \geq \tfrac\alpha2-\delta$ and that there are only finitely many roots
of the equation $\L\mu(z)=1$ in the strip $\tfrac\alpha2-\delta \leq \Real(z) < \alpha$,
see Lemmas 2.1 and 2.3 in \cite{Janson+Neininger:2008}  (note that although the setup in  \cite{Janson+Neininger:2008} is slightly different,
the proofs carry over without changes).

Now let $f=\1_{[0,\infty)}$ and recall
that $N(t) = \cZ_t^f$, hence $V(t) \defeq \E[N(t)] = m_t^f$ for $t \in \R$.
In analogy to the derivation of \cite[Eq.\ (3.11)]{Janson+Neininger:2008},
we use the recursive structure of $\cZ_t^f$ to obtain a renewal equation for $V(t)$ as follows. We start with
 \begin{equation*}
\cZ_t^f = f(t) + \sum_{j=1}^N \cZ_{j,t-X_j}^f
\end{equation*}
where $\cZ^f_{1,t},\cZ^f_{2,t}\ldots$ are independent copies of $\cZ^f_{t}$.
Taking expectations, then conditioning with respect to $\xi$,
the reproduction point process of the ancestor, we infer
\begin{equation}	\label{eq:renewal equation for V}
m_t^f = f(t) + \int m_{t-x}^f \, \mu(\dx) = f(t) + \mu*m_t^f ,	\quad	t \in \R.
\end{equation}
Our subsequent proof relies on a smoothing technique.
So let $
\rho \defeq \1_{[0,1]}$.
For any $\varepsilon>0$, we set
\begin{equation*}	\textstyle
\rho_{\varepsilon}(t) \defeq \frac1\varepsilon\rho(\frac{t}{\varepsilon}) = \frac1\varepsilon\1_{[0,\varepsilon]}(t),
\quad	t \in \R.
\end{equation*}
Then for $f_{\varepsilon} \defeq f*\rho_{\varepsilon}$ (Lebesgue convolution), we have
\begin{equation*}
f_{\varepsilon}(t) \leq f(t) \leq f_{\varepsilon}(t+\varepsilon)
\end{equation*}
for all $t \in \R$,
which in turn gives
\begin{equation}
	\label{eq:inequalities between m's}
m^{f_{\varepsilon}}_t \leq m^{f}_t \leq m^{f_{\varepsilon}}_{t+\varepsilon}.
\end{equation}
Also, one can check that $t \mapsto m^{f_{\varepsilon}}_t$ is a continuous function.
First, we find the asymptotic expansion of this function and then, we let $\varepsilon$ tend to $0$
in a controlled way while letting $t \to \infty$ to deduce the asymptotic behavior of $V(t) = m_t^f$ from that of $m^{f_{\varepsilon}}_t$.
From the renewal equation \eqref{eq:renewal equation for V} we conclude that for $\Real(z)>\alpha$ it holds
\begin{equation*}
\L m^{f_{\varepsilon}}(z) = \L(\rho_\varepsilon * m_t^f)(z) = \L{f_{\varepsilon}}(z)+\L\mu(z) \L m^{f_{\varepsilon}}(z),
\end{equation*}
hence,
\begin{equation*}
\L m^{f_{\varepsilon}}(z) = \frac{\L{f_{\varepsilon}}(z)}{1-\L\mu(z)}	\quad	\text{for } \Real(z)>\alpha.
\end{equation*}
The function
\begin{equation*}
\frac{\L{f_{\varepsilon}}(z)}{1-\L\mu(z)} = \frac{\L{\rho_{\varepsilon}}(z) \L f(z)}{1-\L\mu(z)}
= \frac{1-e^{-\varepsilon z}}{\varepsilon z^2(1-\L\mu(z))}
\end{equation*}
defines a meromorphic extension   of $\L m^{f_{\varepsilon}}$ on $\Real (z) > \vartheta$.
This function decays as $|\Imag(z)|^{-2}$ as $\Imag(z) \to \pm\infty$ and $\Real(z)$ is constant,
hence, it is integrable along vertical lines.
Thus, for any $ \tau>\alpha$, the Laplace inversion formula (see, for instance, \cite[Theorem 7.3 on p.~66]{Widder:1941}) gives
\begin{align*}
m^{f_{\varepsilon}}_t=\frac{m^{f_{\varepsilon}}_{t+}+m^{f_{\varepsilon}}_{t-}}{2}
= \frac1{2\pi\imag}\int \limits_{ \tau-\imag\infty}^{ \tau+\imag\infty}e^{tz}\L m^{f_{\varepsilon}}(z) \, \dz,
\quad	t >0.
\end{align*}
To simplify notation, we assume without loss of generality that
$\vartheta = \frac\alpha2 - \delta$ and that $\L\mu$ is holomorphic on a neighborhood of $\Real(z) \geq \vartheta$.
Then, for large enough $R$, an application of the residue theorem gives
\begin{align*}
\int \limits_{ \tau-\imag R}^{ \tau+\imag R} e^{tz}\L m^{f_{\varepsilon}}(z) \, \dz
= 2\pi\imag\sum_{\lambda \in \Lambda_{\geq}}\Res_{z=\lambda}\big(e^{tz}\L m^{f_{\varepsilon}}(z)\big)
+\int \limits_{\vartheta-\imag R}^{\vartheta+\imag R}e^{tz}\L m^{f_{\varepsilon}}(z) \, \dz\\
 +\int \limits_{\vartheta+\imag R}^{ \tau+\imag R}e^{tz}\L m^{f_{\varepsilon}}(z) \, \dz
 -\int \limits_{\vartheta-\imag R}^{ \tau-\imag R}e^{tz}\L m^{f_{\varepsilon}}(z) \, \dz.
\end{align*}
Here,
\begin{align*}
\bigg| \int \limits_{\vartheta+\imag R}^{ \tau+\imag R} e^{tz}\L m^{f_{\varepsilon}}(z) \, \dz\bigg|
& \leq e^{t \tau}\int \limits_{\vartheta}^{ \tau}\Big|\frac{1-e^{-\varepsilon (x+\imag R)}}
{\varepsilon (x+\imag R)^2(1-\L\mu(x+\imag R))}\Big| \, \dx		\\
&\leq
Ce^{t \tau}\int \limits_{\vartheta}^{ \tau}\Big|\frac{1}{\varepsilon (x+\imag R)^2}\Big| \, \dx
\underset{R\to \infty}{\longrightarrow} 0
\end{align*}
for some constant $C$ that depends only on $\mu$.
Here we used the fact that, by Lemma 2.1 of \cite{Janson+Neininger:2008}, $\L\mu(x+\imag R)$ for $x \geq \vartheta$
and $R \geq R_0$ is uniformly bounded away from $1$ for some sufficiently large $R_0>0$.
The same bound holds for the second horizontal integral.
Therefore, by letting $R$ tend to infinity we conclude
\begin{align}	\label{eq:1}
 m^{f_{\varepsilon}}_t
= \sum_{\lambda \in \Lambda_\geq} \Res_{z=\lambda}\big(e^{tz}\L m^{f_{\varepsilon}}(z)\big)
+ \frac1{2\pi\imag} \int \limits_{\vartheta-\imag\infty}^{\vartheta+\imag\infty}
e^{tz} \L m^{f_{\varepsilon}}(z) \, \dz.
\end{align}
Next, denoting by $\{b_j(\lambda)\}_{j \in \Z}$ the coefficients in the Laurent expansion
of the function $(1-\L\mu(z))^{-1}$ at $z=\lambda \in \Lambda_\geq$
(hence, in particular, $b_j(\lambda)=0$ for $j<-k(\lambda)$), we have
\begin{align}
\Res_{z=\lambda}\big(e^{tz}\L m^{f_{\varepsilon}}(z)\big)
&= \Res_{z=\lambda}\Big(e^{tz}\frac{\L{f_{\varepsilon}}(z)}{1-\L\mu(z)}\Big)	\notag	\\
&= e^{\lambda t} \!\!\! \sum_{\substack{n,l \geq 0 \\ n+l < k(\lambda)}} \!\!
\frac{t^l}{l!}\frac{(\L f_{\varepsilon})^{(n)}(\lambda)}{n!}b_{-1-n-l}(\lambda)	\notag	\\
&= e^{\lambda t} \!\!\! \sum_{\substack{n,l \geq 0 \\ n+l < k(\lambda)}} \!\!
\frac{t^l}{l!}\frac{\int (-x)^nf_\varepsilon(x)e^{-\lambda x} \, \dx}{n!} b_{-1-n-l}(\lambda)	\notag	\\
&= e^{\lambda t} \!\!\! \sum_{\substack{n,l \geq 0 \\ n+l < k(\lambda)}} \!\!
\frac{t^l}{l!}\frac{\int (-x)^nf(x)e^{-\lambda x}\dx}{n!}b_{-1-n-l}(\lambda)
+\varepsilon O(e^{ \alpha t})	\notag	\\
&\eqdef e^{\lambda t} \sum_{0 \leq l < k(\lambda)} c_{\lambda,l}t^l
+\varepsilon O(e^{ \alpha t}),		\label{eq:residue formula for Laplace transform}
\end{align}
where the implicit constant depends only on $\lambda$, not on $\varepsilon$.
It remains to estimate the second term in \eqref{eq:1}.
For $\varepsilon\le \vartheta^{-1}$, using that $|1-e^{-z}| \leq |z| \wedge 2$ for $\Real(z) \geq 0$,
we infer
\begin{align}
	\nonumber
\bigg|\int \limits_{\vartheta-\imag\infty}^{\vartheta+\imag\infty}
e^{tz} \L m^{f_{\varepsilon}}(z) \, \dz \bigg|
&\leq
e^{\vartheta t} \int \limits_{\vartheta-\imag\infty}^{\vartheta+\imag\infty}
\bigg|\frac{1-e^{-\varepsilon z}}{\varepsilon z^2(1-\L\mu(z))}\bigg| \, |\dz|	\\
&\leq \nonumber
C e^{\vartheta t} \int \limits_{\vartheta-\imag\infty}^{\vartheta+\imag\infty}
\frac{|\varepsilon z|^{-1}\wedge 1}{|z|} \, |\dz|	\\
&= \nonumber
Ce^{\vartheta t} \int \limits_{\varepsilon\vartheta-\imag\infty}^{\varepsilon\vartheta+\imag\infty}
(|z|^{-1}\wedge|z|^{-2}) \, |\dz|	\\
&\leq \nonumber
Ce^{\vartheta t} \int \limits_{-\infty}^{\infty} x^{-1}\wedge x^{-2} \wedge(\varepsilon \vartheta)^{-1} \, \dx	\\\label{eq:estimation of the integral}
&\leq
C'e^{\vartheta t}(|\log \varepsilon|+1)
\end{align}
for some constant $C'$ that
depends neither on $t$ nor on $\varepsilon$.
Using \eqref{eq:residue formula for Laplace transform} with $t+\varepsilon$ instead of $t$,
we conclude that
\begin{align*}
\big|\Res_{z=\lambda}\big(e^{(t+\varepsilon)z}\L m^{f_\varepsilon}(z)\big)
-\Res_{z=\lambda}\big(e^{tz}\L m^{f_{\varepsilon}}(z)\big)\big|
= \varepsilon O(e^{\alpha t}),
\end{align*}
where we used $k(\lambda)=1$ for $\lambda = \alpha$, and thereupon, by \eqref{eq:1} and \eqref{eq:estimation of the integral},
\begin{equation*}
m^{f_{\varepsilon}}_{t+\varepsilon}-m^{f_{\varepsilon}}_t= O(\varepsilon e^{\alpha t} +|\log \varepsilon|e^{\vartheta t}).
\end{equation*}
Setting now $\varepsilon\defeq e^{-\alpha t}$,  by \eqref{eq:inequalities between m's},  we infer
\begin{equation*}
m_t^{f}
= \sum_{\lambda\in\Lambda_\geq} e^{\lambda t} \sum_{0 \leq l \leq k(\lambda)-1}
c_{\lambda,l}t^l + O(te^{\vartheta t}),
\end{equation*}
which completes the proof of the lemma.
\end{proof}

Now we are ready to provide the asymptotic expansion for the expectation function
of a general branching process counted with a random characteristic $\varphi$.

\begin{lemma}	\label{Lem:nonlattice asymptotic of the mean}
Suppose that, besides (A\ref{ass:Malthusian parameter}) and (A\ref{ass:first moment}),
condition \eqref{eq:Lmu becomes small} holds.
Then $\Lambda_{\ge}$ is finite and
there are constants $b_{\lambda,l}$, $\lambda \in \Lambda_\geq$, $0 \leq l < k(\lambda)$
such that, for any $\vartheta < \theta< \frac\alpha2$
fulfilling \eqref{eq:choice of theta} and any random characteristic $\varphi$ satisfying \eqref{eq:VEphi integrable}, we have
\begin{align}	\label{eq:nonlattice asymptotic of the mean}
m^{\varphi}_t
&=	\begin{cases}
	\sum \limits_{\lambda\in\Lambda_{\geq}} \sum \limits_{l=0}^{k(\lambda)-1}b_{\lambda,l}
	\int(t-x)^le^{\lambda(t-x)} \E[\varphi(x)] \, \dx + O(e^{\theta t})	&	\text{as } t \to \infty	\\
	O(e^{\gamma t})									&	\text{as } t \to -\infty.
	\end{cases}
\end{align}
\end{lemma}
\begin{remark}
	\label{Rem:compact expansion of the mean}
If we set $\vec{b}_\lambda \defeq (b_{\lambda,l})_{l=0,\ldots,k(\lambda)-1}$, then formula
\eqref{eq:nonlattice asymptotic of the mean} can be written in the form
\begin{equation*}
m_t^\varphi
=\sum_{\lambda\in\Lambda_{\geq}} \ind{t \geq 0\text{ or } \lambda\in\Lambda}
\int  \vec{b}_\lambda^{\;\transp}\exp(\lambda, t-x,k(\lambda)) \e_1  \E[\varphi(x)] \, \dx +O(e^{\theta t}\wedge e^{\gamma t})
\end{equation*}
as $t \to \pm\infty$.
\end{remark}

\begin{proof}
Without loss of generality we assume that the characteristic $\varphi=f$ is a deterministic function.
By Lemma \ref{Lem:asymptotics of E[N(t)] non-lattice}
there  exist constants $c_{\lambda,l}$, $l=0,\ldots,k(\lambda)-1$, $\theta \in (\vartheta,\frac\alpha2)$ and a constant $C$ such that, for any $t\in\R$,
\begin{equation}
	\label{eq:expansion of N}
\bigg|\E[N(t)]-\sum_{\lambda\in\Lambda_{\geq}} \ind{t \geq 0\text{ or } \lambda\in\Lambda} \vec{c}_\lambda^{\ \transp} \exp(\lambda, t,k(\lambda)) \e_1\bigg|
\leq C(e^{\theta t}\wedge e^{\gamma t})
\end{equation}
and hence for the characteristic $f(t) = \1_{[x,\infty)}(t) = \1_{[0,\infty)}(t-x)$, we find
\begin{equation*}
\bigg|m^{f}_t-\sum_{\lambda\in\Lambda_{\geq}} \ind{t-x \geq 0\text{ or }\lambda\in\Lambda}\vec c_\lambda^{\ \transp}\exp(\lambda, t-x,k(\lambda)) \e_1\bigg|
\leq C(e^{\theta (t-x)}\wedge e^{\gamma (t-x)}).
\end{equation*}
Suppose now that $f \geq 0$ is a c\`adl\`ag, nondecreasing function with
\begin{align}	\label{eq:integrability of f}
\int f(x)(e^{-\alpha x}+e^{-\vartheta x}) \, \dx < \infty.
\end{align}
Then $f$ is the measure-generating function of a locally finite measure $\nu$ on the Borel sets of $\R$, namely, for any $y \in \R$,
\begin{equation*}
f(y) = \nu((-\infty,y]) = \int \1_{[x,\infty)}(y) \, \nu(\dx).
\end{equation*}
For any $t\in\R$, by an application of Fubini's theorem, we infer
\begin{equation*}
m^f_t = \int \E[N(t-x)] \, \nu(\dx).
\end{equation*}
By \eqref{eq:derivative of exp} we have $\frac{\dd}{\dx}\exp(\lambda,x,k)=J_{\lambda,k}\exp(\lambda,x,k)$.
We show that \eqref{eq:nonlattice asymptotic of the mean} holds with
$\vec{b}_\lambda\defeq J_{\lambda,k}^\transp\vec{c}_\lambda$, $\lambda \in \Lambda_{\geq}$.
To this end, first notice that,
as $f$ fulfills \eqref{eq:integrability of f}, another application of Fubini's theorem yields
\begin{align*}
\int \exp(\lambda, -x,k) \, \nu(\dx)
= \int J_{\lambda,k} \exp(\lambda, -x,k) f(x) \, \dx.
\end{align*}
We now write
\begin{align*}
&\bigg|m^f_t -\int  \Big(\sum_{\lambda\in\Lambda_{\geq}}\ind{t \geq 0\text{ or }\lambda\in\Lambda} \vec{b}_\lambda^{\;\transp} \exp(\lambda, t-x,k(\lambda))f(x)\e_1 \Big) \, \dx\bigg|	\\
&~= \bigg|m^f_t -\int \Big(\sum_{\lambda\in\Lambda_{\geq}}\ind{t\ge 0\text{ or }\lambda\in\Lambda} \vec{c}_\lambda^{\ \transp}J_{\lambda,k(\lambda)}\exp(\lambda, t-x,k(\lambda))f(x)\e_1\Big) \, \dx\bigg|	\\
&~=
\bigg|\int \Big(\E[N(t-x)]- \sum_{\lambda\in\Lambda_{\geq}}\ind{t\ge 0\text{ or }\lambda\in\Lambda} \vec{ c}_\lambda^{\ \transp}\exp(\lambda, t-x,k(\lambda))\e_1\Big) \, \nu(\dx)\bigg|\\
&~\leq
\int \bigg|\E[N(t-x)]- \sum_{\lambda\in\Lambda_{\geq}}\ind{t-x\ge 0\text{ or }\lambda\in\Lambda} \vec{ c}_\lambda^{\ \transp}\exp(\lambda, t-x,k(\lambda))\e_1 \bigg| \, \nu(\dx)\\
&\hphantom{~\leq}~+ \sum_{\lambda\in\partial\Lambda}\int \big| \ind{t \ge 0} - \ind{t-x \geq 0}\big|\big| \vec{ c}_\lambda^{\ \transp}\exp(\lambda, t-x,k(\lambda))\e_1 \big| \, \nu(\dx).
\end{align*}
For the first term, by	\eqref{eq:expansion of N}, we have
\begin{align*}
&\int\Big|\E[N(t-x)] - \sum_{\lambda\in\Lambda_{\geq}} \ind{t-x \geq 0\text{ or }\lambda\in\Lambda}\vec{c}_\lambda^{\ \transp}\exp(\lambda, t-x,k(\lambda))\e_1\Big| \, \nu(\dx)	\\
&~\leq	C\int (e^{\theta (t-x)}\wedge e^{\gamma (t-x)}) \, \nu(\dx)	\\
&~\leq	C\bigg(\int e^{\theta (t-x)}\nu(\dx)\wedge\int e^{\gamma (t-x)} \, \nu(\dx)\bigg)	\\
&~\leq	C(e^{\theta t}\wedge e^{\gamma t}) \bigg(\theta  \int f(x)e^{-\theta x} \, \dx + \gamma  \int f(x)e^{-\gamma x} \, \dx\bigg).
\end{align*}
Next, for $\lambda\in\partial\Lambda$, we estimate
\begin{align*}
&\int \big| \ind{t \geq 0} - \ind{t-x \geq 0} \big| \big|\vec{ c}_\lambda^{\ \transp}\exp(\lambda, t-x,k(\lambda)) \e_1 \big| \, \nu(\dx)	\\
&~= \1_{(-\infty,0)}(t) \int \limits_{(-\infty,t]}\big|\vec{ c}_\lambda^{\ \transp}\exp(\lambda, t-x,k(\lambda))\e_1\big| \, \nu(\dx)	\\
&\hphantom{~=}~
+ \1_{[0,\infty)}(t) \int \limits_{(t,\infty)} \big|\vec{ c}_\lambda^{\ \transp}\exp(\lambda, t-x,k(\lambda))\e_1\big| \, \nu(\dx)	\\
&~\leq
C \1_{(-\infty,0)}(t) \int e^{\alpha(t-x)} \, \nu(\dx)
+ C \1_{[0,\infty)}(t) \int e^{\vartheta(t-x)} \, \nu(\dx)	\\
&~\leq C' \big(e^{\alpha t}\wedge e^{\vartheta t}\big),
\end{align*}
where we have used \eqref{eq:matrix norm} in the penultimate step.
This completes the proof of the theorem for non-decreasing $f \geq 0$.

Now let $f$ be an arbitrary c\`adl\`ag function
satisfying the integrability condition \eqref{eq:VEphi integrable} (with $f$ in place of $\E[\varphi]$).
Define
\begin{equation*}
f_\pm(x) \defeq \sup\bigg\{\sum_{j=1}^n (f(x_j)-f(x_{j-1}))^\pm: -\infty < x_0 < \ldots < x_n \leq x,\ n \in \N \bigg\}
\end{equation*}
for $x \in \R$.
Clearly, $f_+,f_-: \R \to \R$ are nondecreasing with $f_\pm \geq 0$.
It is known that $f = f_+ - f_-$.
(This is the Jordan decomposition of $f$ on $\R$.)
It is further known that $f_+$ and $f_-$ are c\`adl\`ag since $f$ is.
Further, $\mathrm{V}\!f(x) = f_+(x) + f_-(x)$
and hence \eqref{eq:VEphi integrable} implies that both, $f_+$ and $f_-$ satisfy \eqref{eq:integrability of f}.
The previous part of the proof thereby applies to $f_+$ and $f_-$ and, by linearity,
extends to $f$.
\end{proof}

\begin{remark}	\label{Rem:expansion of mean}
In the situations of Lemmas \ref{Lem:lattice asymptotic of the mean} and \ref{Lem:nonlattice asymptotic of the mean},
$m^{\varphi}_t$ has a representation as in \eqref{eq:expansion of mean}.
Indeed, in both cases, $m_t^\varphi$ can be written as
\begin{align*}
m_t^\varphi
&=\1_{[0,\infty)}(t) \sum_{\lambda\in\Lambda_{\ge}}\int_{\G} \vec{b}_\lambda^{\;\transp}\exp(\lambda,t-x,k(\lambda))\e_1\E[ \varphi](x)\, \ell(\dx)+O(e^{\theta t}\wedge e^{\gamma t})	\\
&=\1_{[0,\infty)}(t) \sum_{\lambda\in\Lambda_{\ge}} \vec{a}_\lambda^{\;\transp}\exp(\lambda,t,k(\lambda))\e_1
+O(e^{\theta t}\wedge e^{\gamma t}),
\end{align*}
where $\vec a_\lambda \defeq \int_\G\exp(\lambda,-x,k(\lambda))^\transp \vec b_\lambda\E[\varphi](x) \, \ell(\dx)$.
Consequently (cf.~\eqref{eq:Z_t^psi_Lambda},  \eqref{eq:H_dLambda} and \eqref{eq:m^psi_dLambda}), we have
\begin{align*}
H_{\Lambda}(t)
&= \sum_{\lambda\in \Lambda}\vec{a}_\lambda^{\;\transp}\exp(\lambda,t,k(\lambda)) W(\lambda, k(\lambda)) \e_1\\
&=\sum_{\lambda\in \Lambda}\vec{b}_\lambda^{\;\transp}\int_{\G}\exp(\lambda,t-x,k(\lambda)) W(\lambda, k(\lambda))\e_1\E[\varphi](x)\ell(\dx) 	\\
&=\sum_{\lambda\in \Lambda}\sum_{l=0}^{k(\lambda)-1}b_{\lambda,l}\int_{\G}e^{\lambda(t-x)}\sum_{j=0}^l {l \choose j} (t-x)^{l-j}W^{(j)} (\lambda)\E[\varphi](x)\ell(\dx)\\
&=\sum_{\lambda\in \Lambda}e^{\lambda t}\sum_{l=0}^{k(\lambda)-1}b_{\lambda,l}\sum_{j=0}^l {l \choose j}W^{(j)} (\lambda)\int_{\G}(t-x)^{l-j} \E[\varphi](x)e^{-\lambda x} \ell(\dx)
\end{align*}
and, similarly,
\begin{align*}
H_{\partial\Lambda}(t)
&=	\sum_{\lambda\in \partial\Lambda}\vec{a}_\lambda^{\;\transp}\exp(\lambda,t,k(\lambda)) \e_1\\
&=	\sum_{\lambda\in\partial \Lambda}\vec{b}_\lambda^{\;\transp}\int_{\G}\exp(\lambda,t-x,k(\lambda)) \e_1\E[\varphi](x) \, \ell(\dx) \\
&=	\sum_{\lambda\in \partial\Lambda}e^{\lambda t}\sum_{l=0}^{k(\lambda)-1}b_{\lambda,l}\int_{\G} (t-x)^{l} \E[\varphi](x)e^{-\lambda x}\, \ell(\dx).
\end{align*}
\end{remark}

\subsection{ Proofs of Theorems \ref{Thm:no roots}, \ref{Thm:main non-lattice} and \ref{Thm:main lattice}}

\begin{proof}[Proof of Theorem \ref{Thm:no roots}]
Theorem \ref{Thm:no roots} is a consequence of Theorem \ref{Thm:main}.
To see this, we first check that the assumptions of Theorem \ref{Thm:no roots} imply
those of Theorem \ref{Thm:main}.
In a second step, we show how the conclusion of Theorem \ref{Thm:no roots}
follows from that of Theorem \ref{Thm:main}.

So suppose the assumptions of Theorem \ref{Thm:no roots}, in particular,
(A\ref{ass:Malthusian parameter}) through (A\ref{ass:second moment}), hold.
Also (A\ref{ass:Lambda finite}) holds because $\Lambda_\geq = \{\alpha\}$ and,
by assumption, there are no roots of the equation $\L\mu(z)=1$ in the strip $\vartheta < \Real(z) < \alpha$.
Let $\varphi$ be a characteristic
satisfying (A\ref{ass:variance growth}), (A\ref{ass:local ui of phi^2}) and \eqref{eq:VEphi integrable}.
Notice that \eqref{eq:VEphi integrable} implies (A\ref{ass:mean growth}).
To see this, decompose $\E[\varphi]=f_1-f_2$ for two non-negative, non-decreasing functions
$f_1$ and $f_2$ such that $\mathrm{V}\!\E[\varphi](x)=f_1(x)+f_2(x)$.
Then
\begin{align*}
\int \Big(f_i(x)e^{-\alpha x}\Big)^{\!*} \, \dx
\leq \int f_i(x+1)e^{-\alpha (x-1)} \, \dx
\leq e^{2\alpha}\int \mathrm{V}\!\E[\varphi](x)e^{-\alpha x} \, \dx
<\infty.
\end{align*}
We conclude from Proposition \ref{Prop:f^*} that both
$x\mapsto f_1(x)e^{-\alpha x}$ and $x\mapsto f_2(x)e^{-\alpha x}$
are directly Riemann integrable and hence so is their difference, i.e., (A\ref{ass:mean growth}) is satisfied. We have to check that $m_t^\varphi$ has a representation of the form \eqref{eq:expansion of mean}.
This follows from Lemma \ref{Lem:nonlattice asymptotic of the mean}
once we have shown that \eqref{eq:Lmu becomes small} holds  (cf.\ Remark \ref{Rem:expansion of mean}).
However, the latter follows from the existence of a Lebesgue density for $\mu$ and the Riemann-Lebesgue lemma  (cf.\ Remark \ref{Rem:Riemann-Lebesgue}).

Since $\Lambda_\geq = \{\alpha\}$ and $\L\mu'(\alpha) = -\beta \not = 0$,
i.e., $k(\alpha)=1$, Lemma \ref{Lem:nonlattice asymptotic of the mean} gives
\begin{align*}
m^{\varphi}_t
&=	\begin{cases}
	e^{\alpha t} b_{\alpha,0} \int \E[\varphi(x)] e^{-\alpha x} \, \dx + O(e^{\theta t})	&	\text{as } t \to \infty	\\
	O(e^{\gamma t})										&	\text{as } t \to -\infty
	\end{cases}
\end{align*}
for some constant $b_{\alpha,0} \in \R$.
From Nerman's law of large numbers \cite[Theorem 6.1, see the proof on p.\;246]{Jagers+Nerman:1984}
or alternatively Proposition \ref{Prop:determining the coefficients},
we know that $b_{\alpha,0} = \beta^{-1}$,
so
\begin{equation*}
m^{\varphi}_t = \1_{[0,\infty)}(t) \beta^{-1} \L(\E[\varphi])(\alpha)e^{\alpha t}  + O(e^{\theta t} \wedge e^{\gamma t})
\end{equation*}
as $t \to \pm\infty$, i.e., $m^{\varphi}_t$ indeed has a representation of the form \eqref{eq:expansion of mean}.
Hence, Theorem \ref{Thm:main} applies.
Let $a_\alpha \defeq \beta^{-1} \L(\E[\varphi])(\alpha)$.
Then $H(t) = e^{\alpha t} a_\alpha W$.
Notice that $n = -1$ and thus $\rho_{-1} = 0$ in Theorem \ref{Thm:main}.
With $\sigma \geq 0$ as in Theorem \ref{Thm:main}, we now infer that in both cases,
$\sigma = 0$ and $\sigma > 0$, that
\begin{align*}
e^{-\frac\alpha2 t} \big(\cZ_t^\varphi - a_\alpha e^{\alpha t} W \big)
= e^{-\frac\alpha2 t} \big(\cZ_t^\varphi - H(t)\big)
\distto \sigma  \sqrt{\tfrac{W}{\beta}}
\cN
\quad	\text{as } t \to \infty
\end{align*}
for a standard normal random variable $\cN$ independent of $W$.
\end{proof}

\begin{proof}[Proof of Theorem \ref{Thm:main non-lattice}]
As before we shall prove that the assumptions of Theorem~\ref{Thm:main} are fulfilled.
As in the proof of Theorem \ref{Thm:no roots}, we conclude that \eqref{eq:VEphi integrable} implies (A\ref{ass:mean growth}).
Lemma \ref{Lem:nonlattice asymptotic of the mean} yields (A\ref{ass:Lambda finite})
and that $m_t^\varphi$ has an expansion of the form \eqref{eq:expansion of mean}
with $r(t)=O(e^{\theta t} \wedge e^{\gamma t})$ as $t \to \infty$, see also Remark \ref{Rem:expansion of mean}.
The assertion now follows from Theorem \ref{Thm:main} and Remark \ref{Rem:expansion of mean}.
Note, that $H_{\partial\Lambda}(t)=O(t^{k-1}e^{\frac \alpha 2 t})=o( t^{k-\frac12}e^{\frac \alpha 2 t})$ by the definition of $H_{\partial\Lambda}$ (cf.\ \eqref{eq:H_dLambda})
and hence this term can be neglected in the limit theorem.
\end{proof}

\begin{proof}[Proof of Theorem \ref{Thm:main lattice}]
Again, we first check that the assumptions of Theorem \ref{Thm:main lattice}
imply those of Theorem \ref{Thm:main}.
So suppose that the assumptions of Theorem \ref{Thm:main lattice},
in particular, (A\ref{ass:Malthusian parameter}) through (A\ref{ass:second moment}), hold.
Regarding (A\ref{ass:Lambda finite}), i.e., the finiteness of $\Lambda_\geq$,
notice that since $\mu$ is lattice, we have
\begin{equation*}
\Lambda_\geq = \{\lambda \in \C: \tfrac{\alpha}2 \leq \Real(\lambda) \leq \alpha,\;-\pi < \Imag(\lambda) \leq \pi,\;\L\mu(\lambda)=1\}.
\end{equation*}
By (A\ref{ass:first moment}), $\L\mu$ is holomorphic on $\{z \in \C: \Real(z) > \vartheta\}$
and non-constant by (A\ref{ass:Malthusian parameter}). Hence, $\L\mu(z)=1$ can hold for only finitely many $z$
in the compact box $\tfrac{\alpha}2 \leq \Real(z) \leq \alpha$, $|\Imag(z)| \leq \pi$,
that is, $\Lambda_\geq$ is finite. Now suppose that $\varphi$ is a characteristic satisfying $\sum_{n \in \Z} |\E[\varphi(n)]|(e^{-\theta n}+e^{-\alpha n})<\infty$
and (A\ref{ass:variance growth}).
Then $\varphi$ satisfies (A\ref{ass:mean growth}) and (A\ref{ass:local ui of phi^2}).  Moreover, by Lemma \ref{Lem:lattice asymptotic of the mean} and Remark \ref{Rem:expansion of mean} we conclude \eqref{eq:expansion of mean} hence, Theorem \ref{Thm:main} applies.
As before, the assertion follows from Theorem \ref{Thm:main}
in combination with Remark \ref{Rem:expansion of mean}.  By the same argument as in the proof of Theorem \ref{Thm:main non-lattice},
the term $H_{\partial\Lambda}$ can be neglected.
\end{proof}

\subsection{Determining the coefficients}

Note that although the constants $\vec c_\lambda$ and $\vec b_\lambda$	
are not given explicitly it is not hard to follow the proofs
and provide explicit expressions for them.
However, even for small $k(\lambda)$, this approach may lead to tedious calculations,
not to mention that there can also be several roots in the relevant strip.
It seems that a more efficient way to determine the constants $b^\lambda_l$
is an application of Lemma \ref{Lem:lattice asymptotic of the mean} or \ref{Lem:nonlattice asymptotic of the mean},
respectively, to a characteristic for which we explicitly know the asymptotic behavior of the
expectation of the associated general branching process.

\begin{proposition}	\label{Prop:determining the coefficients}
Let $ \lambda\in\Lambda_{\ge}$ be a root of $\L\mu(z)=1$ of multiplicity $k$
(in the lattice case we also assume that $\Imag(\lambda) \in (-\pi,\pi]$).
Then the vector $\vec{b}_\lambda$ appearing in Lemma \ref{Lem:lattice asymptotic of the mean}
or \ref{Lem:nonlattice asymptotic of the mean}, respectively,
is given by $M^\lambda \vec{b}_\lambda=\e_{k}$,
where $M^\lambda$ is the $k \times k$ upper triangular matrix such that for $j \geq i$
\begin{equation*}
(M^\lambda)_{i,j}
\defeq -\frac{(j-1)!(k-1)!}{(i-1)!(j-i+k)!} (\L\mu)^{(k+j-i)}(\lambda)
\end{equation*}
in the non-lattice case.
In contrast, in the lattice case,
\begin{equation*}
(M^\lambda)_{i,j}
\defeq \binom{j-1}{i-1} P_{k,j-i}\Big(\frac{\dd}{\dz}\Big) \L\mu(z)_{|z=\lambda},
\end{equation*}
where the polynomials $P_{k,l}$ are given by
\begin{equation*}
P_{k,l}(y) \defeq (-1)^l \sum_{m=1}^{k} \binom{k-1}{m-1} y^{k-m} \frac{B_{l+m}( -y)-B_{l+m}(0)}{l+m},
\end{equation*}
and $B_n$ is the $n^{\mathrm{th}}$ Bernoulli polynomial.
In particular, in both cases,
as $(\L\mu)^{(j)}(\lambda)=0$ for $j=1,\ldots,k-1$,
\begin{equation*}
\det(M^\lambda)=\left(\frac{- \L\mu^{(k)}(\lambda)}{k}\right)^k\neq0
\end{equation*}
and the matrix $M^\lambda$ is invertible.
\end{proposition}

For the proof, we need a lemma which essentially is Jensen's inequality for the total variation operator $\mathrm{V}$
defined in \eqref{eq:def of Vf}.

\begin{lemma}	\label{Lem:Jensen's inequality for V}
Let $\phi = (\phi(t))_{t \in \R}$ be a stochastic process with c\`adl\`ag paths
such that $\phi(t) \in L^1$ for every $t \in \R$ and $t \mapsto \E[\phi](t)$ is again c\`adl\`ag.
Then, finite or not,
\begin{align}	\label{eq:Jensen's inequality for V}
\mathrm{V}\!\E[\phi](t) \leq \E[\mathrm{V}\!\phi(t)]
\end{align}
for every $t \in \R$.
\end{lemma}
\begin{proof}
First notice that $\mathrm{V}\!\phi(t)$ is a random variable.
Indeed, since the paths of $\phi$ are c\'adl\'ag, we have
\begin{align*}
\mathrm{V}\!\phi(t) = \sup\bigg\{\sum_{j=1}^n |\phi(t_j)-\phi(t_{j-1})|: -\infty < t_0 < \ldots < t_n \leq t,\; t_0,\ldots,t_n \in \Q,\ n \in \N \bigg\},
\end{align*}
which is measurable as the supremum of a family of random variables indexed by a countable set.
Since $\E[\phi]$ is also c\`adl\`ag, we infer
\begin{align*}
\mathrm{V}\!\E[\phi](t)
&= \sup\bigg\{\sum_{j=1}^n |\E[\phi(t_j)-\phi(t_{j-1})]|: t_0 < \ldots < t_n \leq t,\, t_0,\ldots,t_n \in \Q,\, n \in \N \bigg\}	\\
&\leq \sup\bigg\{\E \bigg[\sum_{j=1}^n|\phi(t_j)-\phi(t_{j-1})| \bigg]: t_0 < \ldots < t_n \leq t,\, t_0,\ldots,t_n \in \Q,\, n \in \N \bigg\}	\\
&\leq \E \bigg[\sup\bigg\{\sum_{j=1}^n|\phi(t_j)-\phi(t_{j-1})|: t_0 < \ldots < t_n \leq t,\, t_0,\ldots,t_n \in \Q,\, n \in \N \bigg\}\bigg]	\\
&= \E[\mathrm{V}\!\phi](t).
\end{align*}
\end{proof}

\begin{proof}[Proof of Proposition \ref{Prop:determining the coefficients}]

For $\lambda \in \Lambda_{\geq}$, consider the  characteristic
\begin{align*}
\phi(t) &= \e_k^{\ \transp}\E[\phi_\lambda(t)] \e_1\\
&= \1_{[0,\infty)}(t) \E\bigg[\sum_{j=1}^N \e_k^{\ \transp}\1_{[0,X_j)}(t) \exp(\lambda, t-X_j, k)\e_1 \bigg]\\
&=  \1_{[0,\infty)}(t) \E\bigg[\sum_{j=1}^Nf_{X_j}(t)\bigg],	\quad	t \in \R
\end{align*}
where, for any $x \geq 0$,
\begin{equation*}
f_x(t)\defeq \e_k^{\ \transp}\1_{[0,x)}(t) \exp(\lambda, t-x, k)\e_1=\1_{[0,x)}(t)(t-x)^{k-1}e^{\lambda(t-x)}.
\end{equation*}
By Lemma \ref{Lem:Nerman's martingales} (the assumption are fulfilled by Lemma \ref{Lem:phi^1 and phi^2}), the following holds for all $t\ge 0$
\begin{align}
	\label{eq:expansion of phi_lambda_1}
m^\phi_t= \e_k^{\ \transp}\exp({\lambda },t,k)\e_1.
\end{align}
We now aim to apply either Lemma \ref{Lem:nonlattice asymptotic of the mean} in the lattice case
or Lemma \ref{Lem:lattice asymptotic of the mean} in the non-lattice case to the characteristic $\phi$. To do this, we need to  verify that $\phi$ satisfies either \eqref{eq:VEphi integrable} or \eqref{eq:integrability condition in the lattice case} respectively.

Using equations \eqref{eq:derivative of exp} and \eqref{eq:matrix norm}, we get
\begin{align*}
\mathrm{V}\!f_x(t)&\le
|f_x(0)|\1_{[0,\infty)}(t)+\int_0^{t\wedge x}\big|\frac{\dd}{\ds}f_x(s)\big| \, \ds
+\1_{\{t\ge x\}}\\
&\le C\times
\begin{cases}
0&\quad \text {for }t<0,\\
e^{(\Real(\lambda)-\delta)(t-x)}&\quad \text {for }0\le t<x,\\
1&\quad \text {for }x\le t
\end{cases}
\end{align*}
for some $\delta \in (0,\Real(\lambda) - \vartheta)$ and a constant $C$ that does not depend on $x \geq 0$.
Thus,
\begin{align*}
\int \mathrm{V}\!f_x(t) \big(e^{-\vartheta t} +e^{-\alpha t}\big) \, \dt \leq
2\int \mathrm{V}\!f_x(t) e^{-\vartheta t} \, \dt \leq C'e^{-\vartheta x}
\end{align*}
for some other constant $C'$ that also does not depend on $x$.
Now, observe that we can apply \eqref{eq:Jensen's inequality for V} to $\e_k^{\ \transp}\phi_\lambda \e_1$ since both this characteristic and its expectation have c\`adl\`ag paths.
Only the latter requires a proof. By \eqref{eq:matrix norm},
\begin{align*}
|\e_k^{\ \transp}\phi_\lambda(t) \e_1| \leq C \1_{[0,\infty)}(t) e^{\vartheta t} \sum_{j=1}^N e^{-\vartheta X_j}
\end{align*}
for any $t \in \R$. Hence, (A\ref{ass:first moment}) and the dominated convergence theorem
imply that $\mathbb{E}[\e_k^{\ \transp}\phi_\lambda \e_1]$
is c\`adl\`ag.
Using the subadditivity of $\mathrm{V}$, \eqref{eq:Jensen's inequality for V} and assumption (A\ref{ass:first moment}), we get
\begin{align*}
\int \mathrm{V}\!\phi(t) \big(e^{-\vartheta t} +e^{-\alpha t}\big) \, \dt
&=\int_0^\infty \mathrm{V}\!\bigg(\E\bigg[\sum_{j=0}^Nf_{X_j}\bigg]\bigg) (t)\big(e^{-\vartheta t}\!+\!e^{-\alpha t}\big) \, \dt	\\
&\leq \E\bigg[\sum_{j=0}^N\int \mathrm{V}\!f_{X_j} (t)\big(e^{-\vartheta t}\!+\!e^{-\alpha t}\big) \, \dt\bigg]
\leq C' \E\bigg[\sum_{j=0}^Ne^{-\vartheta X_j}\bigg]
<\infty.
\end{align*}
A similar argument gives \eqref{eq:integrability condition in the lattice case} in the lattice case.
An application of either Lemma \ref{Lem:nonlattice asymptotic of the mean}
or Lemma \ref{Lem:lattice asymptotic of the mean} (as mentioned in Remark \ref{Rem:expansion of mean}), respectively,
to the characteristic $\phi$, yields
\begin{align}
	\nonumber
m^\phi_t &= \sum_{\lambda' \in \Lambda_\ge} \vec{b}_{\lambda'}^{\;\transp} \int \exp(\lambda',t-x,k)\phi(x) \, \ell(\dx) \e_1 + O(e^{\theta t})\\
\label{eq:expansion of phi_lambda_2}
&= \sum_{\lambda' \in \Lambda_\ge} \bigg(\bigg(\int \exp(\lambda',-x,k)\phi(x) \, \ell(\dx)\bigg)^\transp \, \vec{b}_{\lambda'}\bigg)^{\transp} \exp(\lambda',t,k) \e_1 + O(e^{\theta t}),
\end{align}
as $t$ goes to infinity. Next take the difference of the two asymptotic expansions \eqref{eq:expansion of phi_lambda_1} and \eqref{eq:expansion of phi_lambda_2} for $m_t^\phi$
and then apply Lemma \ref{lem:asymptotic linear independent} to infer
\begin{equation*}
\bigg(\int\exp(\lambda,-x,k)\phi(x) \, \ell(\dx)\bigg)^\transp \, \vec{b}_\lambda
=\e_k.
\end{equation*}
It now suffices to evaluate the coefficients of the matrix
\begin{align*}
M^\lambda\defeq\int \exp(\lambda,-x,k)\phi(x) \, \ell(\dx).
\end{align*}
First, we deal with the non-lattice case.
For this purpose, recalling  basic properties of the beta function $B$, we infer
\begin{align*}
\int (-x)^le^{-\lambda x}\phi(x) \, \dx
&=(-1)^l  \E\bigg[\sum_{j=1}^N\int_0^{X_j}x^l(x-X_j)^{k-1}e^{-\lambda X_j} \, \dx\bigg]\\
&=(-1)^l\iint_0^yx^l(x-y)^{k-1} \, \dx e^{-\lambda y} \, \mu(\dy)\\
&=(-1)^{l+k-1}\int B(l+1,k)y^{k+l}e^{-\lambda y} \, \mu(\dy)\\
&=-\frac{l!(k-1)!}{(l+k)!} (\L\mu)^{(k+l)}(\lambda),
\end{align*}
and therefore
\begin{align*}
(M^\lambda)_{i,j}
&=\bigg(\int\exp(\lambda,-x,k)\phi(x)\dx\bigg)_{j,i}	\\
&= -\1_{\{j \geq i\}}  \L\mu^{(k+j-i)}(\lambda)\cdot{j-1\choose i-1}\frac{(j-i)!(k-1)!}{(j-i+k)!}	\\
&= -\1_{\{j \geq i\}} \frac{(j-1)!(k-1)!}{(i-1)!(j-i+k)!} (\L\mu)^{(k+j-i)}(\lambda).
\end{align*}

In the non-lattice case, invoking Faulhaber's formula, we have
\begin{align*}
\sum_{x\in\Z} (-x)^le^{-\lambda x}\phi(x)
&=(-1)^l  \E\bigg[\sum_{j=1}^N\sum_{0\le x< X_j}x^l(x-X_j)^{k-1}e^{-\lambda X_j}\bigg]	\\
&=(-1)^l \int\sum_{0\le x<y}x^l(x-y)^{k-1} e^{-\lambda y} \, \mu(\dy)\\
&=(-1)^l \int\sum_{m=1}^{k}{k-1 \choose m-1}(-y)^{k-m}\sum_{0\le x<y}x^{l+m-1} e^{-\lambda y} \, \mu(\dy)	\\
&=(-1)^l \int\sum_{m=1}^{k}{k-1 \choose m-1}(-y)^{k-m}\frac{B_{l+m}(y)-B_{l+m}(0)}{l+m} e^{-\lambda y} \, \mu(\dy)	\\
&=P_{k,l}\Big(\frac{\dd}{\dz}\Big) \L\mu(z)\Big|_{z=\lambda},
\end{align*}
which gives
\begin{align*}
(M^\lambda)_{i,j}&=\Big(\int\exp(\lambda,-x,k)\phi(x)\dx\Big)_{j,i}\\
&=\ind{j\ge i}{j-1\choose i-1}
 P_{k,j-i}\Big(\frac{\dd}{\dz}\Big) \L\mu(z)\Big|_{z=\lambda}.
\end{align*}
\end{proof}

\section{Discussion and open problems}
\label{sec:discussion}
In this section we formulate several open problems which are closely related to the present framework.

\begin{oproblem}
	\label{OP:multitype_CMJ}
Prove a corresponding limit theorem for the multitype CMJ process.
\end{oproblem}
\begin{oproblem}
Provide functional versions of the theorems proved in this paper.
\end{oproblem}

A drawback of our method in the non-lattice case is that,
in order to find the asymptotic of the mean $m^{\varphi}_t$
we need to assume that the measure $\mu$ is absolutely continuous
with respect to Lebesgue measure
 or that at least \eqref{eq:Lmu becomes small} holds.
\begin{oproblem}
In the non-lattice case, work out
a proof that does not require  \eqref{eq:Lmu becomes small}.
\end{oproblem}

The Gaussian fluctuations appearing
in our theorems are caused by
the finiteness of the second moment \eqref{eq:Konrad's condition}.
In the case where the condition is not satisfied one
still may ask for a generalization.

\begin{oproblem}
Prove a version of the limit theorems  with a stable limit.
\end{oproblem}

One of the basic ingredients of the CMJ process is the underlying branching random walk
$(S(u))_{u\in\familytree}$ with positive increments.
However, the process $\cZ^{ \varphi}$ can also be defined
for a branching random walk with two-sided increments and suitable $ \varphi$.

\begin{oproblem}
Investigate the behavior of $\cZ^{ \varphi}_t$ for a
branching random walk $(S(u))_{u\in\familytree}$ with two-sided increments.
\end{oproblem}

A central limit theorem is usually complemented by a law of the iterated logarithm
(see, for instance, \cite{Iksanov+Kolesko+Meiners:2021b} for the central limit theorem
and the law of the iterated logarithm for Nerman's martingale).
This motivates the following problem.

\begin{oproblem}
Prove a corresponding law of the iterated logarithm for $\cZ^{ \varphi}_t$.
\end{oproblem}

The martingale limits $W^{(j)}(\lambda)$ play an important role
in the asymptotic behavior of the general branching process.
It is important to obtain more information about their distributions.
In particular, the following problem seems to be quite relevant.

\begin{oproblem}
Derive the first-order asymptotic behavior of the tail probabilities
$\Prob(|W^{(j)}(\lambda)|>t)$ as $t \to \infty$ for $j=0,\ldots,k(\lambda)-1$.
\end{oproblem}

We also believe that the approach developed in the present paper
might be useful for settling the following.

\begin{oproblem}
Find large deviation estimates for $\cZ^{ \varphi}_t$.
\end{oproblem}

\subsubsection*{Acknowledgements}
 The authors thank two anonymous referees for exceptionally careful and constructive reports
whose consideration led to a significant improvement of the paper. In the preliminary version of our work, there was an error in the variance calculation in the model described in Section \ref{sec:Poisson}, and we would like to express our sincere gratitude to Beno\^{\i}t Henry for his assistance in its correction. Additionally, we thank David Croydon for bringing the papers \cite{Charmoy+al:2017} and \cite{Jagers+Nerman:1984b} to our attention.
 A.\,I.\ was supported by the Grant of the Ministry of Education and Science
of Ukraine for perspective development of a scientific direction
``Mathematical sciences and natural sciences'' at Taras Shevchenko
National University of Kyiv.
M.\,M.\  was supported by DFG grant ME3625/4-1.

\begin{appendix}
\section*{An auxiliary result}
\begin{lemma}	\label{lem:asymptotic linear independent}
Let $r\in\N$, $k_j\in \N_0$, $j=1,\ldots,r$, $z_1,\dots,z_r$ be distinct complex numbers
with $|z_j| \geq \rho$ and $b_{j,l}$, $j=1,\ldots,r$, $l=0,\ldots,k_j$ complex numbers.
Then
\begin{align}	\label{eq:sum of combinations}
\sum_{j=1}^r \sum_{l=0}^{k_j}b_{j,l}n^lz_j^n=o(\rho^n)	\quad\text{as } n \to \infty,\ n\in\N
\end{align}
implies $b_{j,l}=0$ for all $j=1,\ldots,r$, $l=0,\ldots,k_j$.
\end{lemma}
\begin{proof}
We use induction on $K\defeq k_1+\dots+k_r$.
Suppose that $K=0$ and denote by
$V(z_1,\ldots, z_r) \defeq (z_j^{m-1})_{m,j=1,\ldots,r}$ the Vandermonde matrix associated with $z_1,\ldots,z_r$.
Then, putting $b_j\defeq b_{j,0}$, each component of the vector
\begin{equation*}
V(z_1,\ldots, z_r)
\begin{pmatrix}
b_1z_1^n \\ \vdots \\ b_rz_r^n
\end{pmatrix}
= \bigg(\sum_{j=1}^r b_{j} z_j^{m-1} z_j^n\bigg)_{m=1,\ldots,r}
\end{equation*}
is $o(\rho^n)$ as $n \to \infty$.
Since $z_1,\ldots,z_r$ are distinct, $\det V(z_1,\ldots, z_r) \not=0$,
hence we may multiply the last displayed equation by the inverse of $V(z_1,\ldots, z_r)$ from the left and conclude
that $b_{j}z_j^n=o(\rho^n)$ as $n \to \infty$,
which, in turn, gives $b_{j}=0$ for $j =1,\ldots,r$.
For the induction step, we assume that the induction hypothesis holds
whenever $k_1 + \ldots + k_r \leq K$. If now $k_1 + \ldots + k_r = K+1$,
then there is exists some $j_0 \in \{1,\ldots,r\}$ with $k_{j_0}>0$.
We define a linear operator $L$ by
\begin{equation*}
L f(n)\defeq f(n)-z_{j_0} f(n-1)
\end{equation*}
for any $f:\Z\mapsto\C$.
If $f(n)=o(\rho^n)$, then so is $L f(n)$. Moreover, if $f(n)=p(n)z^n$ for some polynomial $p$,
then $Lf(n)=\tilde p(n)z^n$ for another polynomial $\tilde p$ with $\deg \tilde p \leq \deg p$
and if $f(n)=n^l z_{j_0}^n$, then $L f(n)=(ln^{l-1}+p(n))z_{j_0}^n$
for some polynomial $p$ with $\deg p \leq l-2$.
Applying $L$ to both sides of the relation \eqref{eq:sum of combinations} we infer
\begin{align*}
\sum_{j=1}^r\sum_{l=0}^{\tilde k_j}\tilde b_{j,l}n^lz_j^n=o(\rho^n)	\quad \text{as } n \to \infty,\ n\in\N
\end{align*}
for some $\tilde k_j\le k_j$, $\tilde k_{j_0}=k_{j_0}-1$ and $\tilde b_{j_0,k_{j_0}-1}=k_{j_0}b_{j_0,k_{j_0}}$.
The induction hypothesis gives that $\tilde b_{j_0,k_{j_0}-1}=0$
which implies that $b_{j_0,k_{j_0}}=0$ as well.
This allows us to replace $k_{j_0}$ by $k_{j_0}-1$ in \eqref{eq:sum of combinations}.
The claim now follows by induction.
\end{proof}
\end{appendix}

\end{document}